\documentclass[11pt]{amsart}
\usepackage[OT2,T1]{fontenc}
\usepackage{amsmath,amssymb,amsthm,mathrsfs,calligra,dsfont,mathtools}
\usepackage{hyperref} 
\usepackage[dvipsnames]{xcolor}
\usepackage[all]{xy}
\usepackage{bbm}
\usepackage{enumerate,enumitem}
\usepackage{comment}
\usepackage{mathrsfs}
\usepackage{tikz-cd} 
\usepackage{etoolbox}

\usepackage{fancyhdr}
\fancyhfoffset{0pt}

\newtheorem{thm}{Theorem}[section]
\newtheorem{prop}[thm]{Proposition}
\newtheorem{lemma}[thm]{Lemma}
\newtheorem{cor}[thm]{Corollary}
\newtheorem{defn}[thm]{Definition}
\newtheorem{conj}[thm]{Conjecture}
\newtheorem*{definitionwn}{Definition}

\newtheorem{theoremletter}{Theorem}

\theoremstyle{remark}

\newtheorem{rem}[thm]{Remark}

\newcommand{\R}{\mathbb{R}}

\newcommand{\Q}{\mathbb{Q}}
\newcommand{\N}{\mathbb{N}}
\newcommand{\Z}{\mathbb{Z}}

\newcommand{\bV}{\mathbb{V}}

\newcommand{\cL}{\mathscr{L}}
\newcommand{\cB}{\mathcal{B}}
\newcommand{\cC}{\mathcal{C}}

\newcommand{\cF}{\mathcal{F}}
\newcommand{\gh}{\mathfrak{h}}
\newcommand{\gp}{\mathfrak{p}}

\newcommand{\cris}{\mathrm{cris}}

\newcommand{\cA}{\mathcal{A}}
\newcommand{\cM}{\mathcal{M}}
\newcommand{\cO}{\mathcal{O}}

\newcommand{\cP}{\mathcal{P}}
\newcommand{\cR}{\mathcal{R}}
\newcommand{\cS}{\mathcal{S}}
\newcommand{\cT}{\mathcal{T}}
\newcommand{\cH}{\mathcal{H}}
\newcommand{\cU}{\mathcal{U}}
\newcommand{\cV}{\mathcal{V}}
\newcommand{\cW}{\mathcal{W}}

\newcommand{\loc}{\mathrm{loc}}

\newcommand{\rQ}{\mathrm{Q}}
\newcommand{\rP}{\mathrm{P}}

\newcommand{\Gm}{\mathfrak{m}}
\newcommand{\Gp}{\mathfrak{p}}

\newcommand{\ob}[1]{\mkern 1.5mu\overline{\mkern-1.5mu#1\mkern-1.5mu}\mkern 1.5mu}

\DeclareMathOperator{\ad}{ad}

\DeclareMathOperator{\cyc}{cyc}

\DeclareMathOperator{\End}{End}

\DeclareMathOperator{\Frob}{Frob}

\DeclareMathOperator{\Ind}{Ind}
\DeclareMathOperator{\Gal}{Gal}
\DeclareMathOperator{\GL}{GL}
\DeclareMathOperator{\rH}{H}

\DeclareMathOperator{\Hom}{Hom}

\DeclareMathOperator{\ord}{ord}

\DeclareMathOperator{\Spec}{Spec }

\DeclareMathOperator{\res}{res}

\DeclareMathOperator{\Sel}{Sel}
\DeclareMathOperator{\rank}{rank}

\DeclareMathOperator{\tr}{tr}

\DeclareMathOperator{\coker}{coker}
\DeclareMathOperator{\Id}{Id}

\newcommand{\rb}{\mathrm{b}}
\newcommand{\rc}{\mathrm{c}}
\newcommand{\rd}{\mathrm{d}}

\newcommand{\HH}{\mathrm{H}}
\newcommand{\rf}{\mathrm{f}}

\newcommand{\ur}{\mathrm{ur}}

\DeclareSymbolFont{cyrletters}{OT2}{wncyr}{m}{n}
\DeclareMathSymbol{\Sha}{\mathalpha}{cyrletters}{"58}

\DeclareFontFamily{U}{mathx}{\hyphenchar\font45}
\DeclareFontShape{U}{mathx}{m}{n}{
	<5> <6> <7> <8> <9> <10>
	<10.95> <12> <14.4> <17.28> <20.74> <24.88>
	mathx10
}{}
\DeclareSymbolFont{mathx}{U}{mathx}{m}{n}
\DeclareFontSubstitution{U}{mathx}{m}{n}
\DeclareMathAccent{\widecheck}{0}{mathx}{"71}
\DeclareMathAccent{\wideparen}{0}{mathx}{"75}

\makeatletter
\DeclareFontFamily{OMX}{MnSymbolE}{}
\DeclareSymbolFont{MnLargeSymbols}{OMX}{MnSymbolE}{m}{n}
\SetSymbolFont{MnLargeSymbols}{bold}{OMX}{MnSymbolE}{b}{n}
\DeclareFontShape{OMX}{MnSymbolE}{m}{n}{
	<-6>  MnSymbolE5
	<6-7>  MnSymbolE6
	<7-8>  MnSymbolE7
	<8-9>  MnSymbolE8
	<9-10> MnSymbolE9
	<10-12> MnSymbolE10
	<12->   MnSymbolE12
}{}
\DeclareFontShape{OMX}{MnSymbolE}{b}{n}{
	<-6>  MnSymbolE-Bold5
	<6-7>  MnSymbolE-Bold6
	<7-8>  MnSymbolE-Bold7
	<8-9>  MnSymbolE-Bold8
	<9-10> MnSymbolE-Bold9
	<10-12> MnSymbolE-Bold10
	<12->   MnSymbolE-Bold12
}{}

\let\llangle\@undefined
\let\rrangle\@undefined
\let\lsem\@undefined
\let\rsem\@undefined
\DeclareMathDelimiter{\llangle}{\mathopen}%
{MnLargeSymbols}{'164}{MnLargeSymbols}{'164}
\DeclareMathDelimiter{\rrangle}{\mathclose}%
{MnLargeSymbols}{'171}{MnLargeSymbols}{'171}
\DeclareMathDelimiter{\lsem}{\mathopen}%
{MnLargeSymbols}{'102}{MnLargeSymbols}{'102}
\DeclareMathDelimiter{\rsem}{\mathclose}%
{MnLargeSymbols}{'107}{MnLargeSymbols}{'107}                     

\makeatother

\hypersetup{linktocpage,bookmarksopen = true, bookmarksopenlevel = 1, colorlinks=true,
linkcolor = Bittersweet, citecolor = Blue, urlcolor = Bittersweet, pdfstartview = FitR}

\title[\LARGE{T\MakeLowercase{he eigencurve at crystalline points with scalar} F\MakeLowercase{robenius and regulators}}]{ \LARGE{T\MakeLowercase{he eigencurve at crystalline points with scalar} F\MakeLowercase{robenius and }G\MakeLowercase{ross--}S\MakeLowercase{tark regulators}}}
\author{\large{A\MakeLowercase{del} B\MakeLowercase{etina}, A\MakeLowercase{lexandre} M\MakeLowercase{aksoud and} A\MakeLowercase{lice} P\MakeLowercase{ozzi}}}

\address{Faculty of Mathematics, University of Vienna, Oskar--Morgenstern--Platz 1, A--1090 Wien, Austria.}

\email{adelbetina@gmail.com }

\address{Fakultät für Elektrotechnik, Informatik und Mathematik, Universität Paderborn, Warburger Str. 100, 33098 Paderborn, Germany.}
\email{maksoud.alexandre@gmail.com}

\address{School of Mathematics, University of Bristol, Woodland Road., BS8 1TW Bristol, 
United Kingdom.}
\email{alice.pozzi@bristol.ac.uk }

\thanks{A. B. was supported by the PAT4628923 of the Austrian Science Fund (FWF). A.M. acknowledges support from Deutsche Forschungsgemein-schaft 
(DFG, German Research Foundation) – Project-ID 491392403 – TRR 358.}


\begin{document}

	\begin{abstract} 
		A complete description of the local geometry of the $p$-adic eigencurve $\cC$ at $p$-irregular classical weight one cusp forms is given in the cases where the usual $\mathcal{R}=\cT$ methods fall short. 
As an application, we show that the ordinary $p$-adic étale cohomology group attached to the tower of elliptic modular curves $X_1(Np^r)$ is not free over the Hecke algebra, when localized at a $p$-irregular weight one point.
		
	\end{abstract}

\maketitle
	
\tableofcontents
		
	\section{Introduction}

Central conjectures in number theory, including the Birch--Swinnerton-Dyer Conjecture, Bloch-Kato Conjecture and the Main Conjecture in Iwasawa Theory, relate properties of arithmetic objects to special values of $L$-functions. Progress on conjectures of this nature leverages on the theory of automorphic forms. More precisely, when the arithmetic object is attached to an automorphic form of cohomological weight, the approach often consists in constructing certain Galois cohomology classes arising from the geometry of Shimura varieties. For automorphic forms of \emph{non-cohomological} weight, this strategy can be adapted by $p$-adically deforming the Galois cohomology classes constructed for cohomological weights through a suitable limit process. This method has successfully been employed to tackle instances of the Bloch-Kato Conjecture, \cite{BC09,SU06,LZ16,LZ21},  the Perrin-Riou's Conjecture \cite{BDV,burungale2024zetaelementsellipticcurves}, the construction of two variable $p$-adic $L$-functions \cite{Be12,BB24MAMS}. 
Implementing this strategy requires a precise understanding of the $p$-adic families of automorphic forms deforming those of non-cohomological weight; thus, one should understand the local geometry of eigenvarieties at the corresponding points. This is usually applied at points in which the eigenvariety is known to be smooth. Although this assumption can sometimes be relaxed, current techniques are constrained to cases where:
\begin{itemize}
    \item the local ring of the eigenvariety at the limit of discrete series point is Gorenstein; 
    \item certain Hecke modules arising from the \'etale cohomology of Shimura varieties are free. 
\end{itemize}
This paper examines the failure of both properties at points associated to irregular weight one elliptic cusp forms. The careful analysis of the geometry of the eigencurve in this setting can be viewed as a first step towards arithmetic applications beyond current methods, which we hope to explore in future work. \\

\subsection*{The main result} Let $p$ be a prime number and $\cC$ denote the $p$-adic cuspidal eigencurve of prime-to-$p$ level $N$.  It is endowed with a flat and locally finite morphism \[ \kappa: \cC \rightarrow \cW,\] called the weight map, where $\mathcal{W}$ is the rigid analytic generic fiber of $\mathrm{Spa}(\Z_p\lsem \Z_p^\times \rsem, \Z_p\lsem \Z_p^\times \rsem)$.

Given a classical cuspidal eigenform of tame level $N$, the choice of a $p$-stabilization of $f$ with finite slope, when it exists, defines a point $\mathrm{x}$ on $\cC$. If $f$ has (cohomological) weight $k\geq 2$, the map $\kappa$ is often known to be \'etale at the point $x$: if $f$ is ordinary, this follows from a celebrated result of Hida (\cite{hida1986galois}); more generally, if $f$ is $p$-regular and has a non-critical slope, it is a consequence of Coleman's classicality criterion for overconvergent forms (see \cite{Coleman}). 
By contrast, unusual phenomena can occurr at weight 1 points: for example, Mazur and Wiles noted that weight one specializations need not be classical or even Hodge-Tate \cite{mazurwiles}. Even at classical  weight 1 points, the local geometry of the eigencurve depends crucially on local properties of the corresponding Galois representation, as showed in 
 \cite{D-B,betinadimitrovKatz,BDPo,cho-vatsal}. 
	More precisely, let $f$ be a classical newform of weight $1$,  level~$N$ and nebentypus $\chi_f$. A construction of Deligne \cite{deligne-serre} attaches to $f$ an odd irreducible Artin representation 
	$\rho : G_\Q=\Gal(\ob{\Q}/\Q)\rightarrow \GL_2(\overline \Q)$
	which is unramified at all prime numbers $\ell \nmid N$ and and for which the characteristic polynomial of an arithmetic Frobenius $\Frob_{\ell}$ at $\ell$ is given by the Hecke polynomial
	$$
	X^2-a(\ell,f)X+\chi_f(\ell)
	$$
	where  $a(\ell,f)$ is the $T_\ell$-eigenvalue of $f$.  Let $\alpha$ and $\beta$ be the roots of the $p$-th Hecke polynomial. We say that $f$  is $p$-regular if $\alpha \neq\beta$; we call $f$ $p$-irregular otherwise. The choice of a root, say $\alpha$, of the $p$-th Hecke polynomial gives rise to a $p$-stabilization $f_\alpha$ whose $q$-expansion is given by $f_\alpha(q)=f(q)-\beta\cdot f(q^p)$. It is a normalized cuspidal eigenform of level $Np$ with $\textbf{U}_p$-eigenvalue $\alpha$. As $\alpha$ and $\beta$ are roots of unity, each $p$-stabilization $f_\alpha$ and $f_\beta$ of $f$ is $p$-ordinary, and thus defines a point of $\cC.$	
    If $f$ is $p$-regular, the geometry of $\cC$ at $f_\alpha$ has been completely described by seminal work of Bella\"iche and Dimitrov \cite{D-B}, which shows that $\cC$ is smooth at $f_\alpha$. Their argument relies on the constructions of a universal ordinary deformation ring $\mathcal{R}^{\ord}$ of $\rho$ which is compared with the completed local ring of $\cC$ at $f_\alpha$. Crucially, the dimension of the tangent space of $\cR^{\ord}$ is determined via a clever calculation in $p$-adic transcendence theory. 
    
    The geometry at $p$-irregular points is expected to be more involved. In recent work \cite{betinadimitrovKatz}, the first-named author and Dimitrov provide a full description of the geometry when $f$ is $p$-irregular and has CM by a quadratic imaginary field $K$. The method of \cite{betinadimitrovKatz} consists in defining a ``non CM'' deformation ring, that is, a deformation ring designed to parametrize deformations that do not arise from CM families, and establishing for this an $\cR=\cT$ theorem. In this paper, we describe the geometry of the eigencurve in all the remaining irregular cases, thus settling the problem of understanding the geometry of the eigencurve at all weight one points, conditionally on the non-vanishing of certain $p$-adic regulators (which is always consistent with conjectures in $p$-adic transcendence theory, as discussed in \S\ref{sec:regulators}).
    
	\begin{theoremletter}\label{thm:main_result} Let $\mathrm{x}$ be a point of the eigencurve $\cC$ corresponding to the unique $p$--stabilization $f_\alpha$ of a $p$--irregular weight one newform $f$. Let $\cT$ and $\varLambda \simeq \bar{\Q}_p\lsem X \rsem $ be the completed local ring of $\cC$ at $\mathrm{x}$ and $\cW$ at $\kappa(x)$, respectively.
    If hypotheses  \eqref{tag:discriminant} and \eqref{tag:nonvanishing_resultant_Q_PL} of \S\ref{sec:regulators} hold, then  
	\[ \cT \simeq \bar\Q_p\lsem X_1,X_2,X_3,X_4 \rsem/(\{X_iX_j \}_{1 \leq i<j \leq 4}), \]
	where the finite, flat map $\kappa^{\#}:\varLambda \to \cT$ sends $X \mapsto \sum_{i=1}^4 X_i$ under the above identification.\\ In particular, $\cT$ is not Gorenstein and there are four irreducible  components  of $\cC$ containing $\mathrm{x}$, meeting transversally at $\mathrm{x}$ and each irreducible component is \'etale at $\mathrm{x}$  over $\cW$.
		
	\end{theoremletter}
	We note that the techniques of \cite{betinadimitrovKatz} cannot be extended in this setting. The reason is, essentially, that several irreducible components meet at the irregular weight one points. When these points do not have complex multiplication,  none of the corresponding Hida eigenfamilies is attached to Galois representations with a distinguished shape. Thus, the problem cannot be broken down into establishing $\cR=\cT$ theorems for relatively simpler rings. In addition, the assumption that the action of Frobenius at $p$ is scalar on the Galois representation attached to $f$ poses difficulties, as the ordinary deformation functor is not representable, and two distinct Hida families can give rise to distinct  $p$-ordinary lines residually. We circumvent the use of deformation rings altogether and instead study the $p$-adic family of Galois representations attached to each irreducible component $\mathcal{Z}$ of the eigencurve at $x$. In this, a crucial role is played by the line obtained by reducing the ordinary filtration at $x$: we refer to the position of this line with respect to a fixed basis as ``residual slope'' of $\mathcal{Z}$. Our method, based on a thorough analysis of higher infinitesimal deformations of $\rho$ along $\mathcal{Z}$, enables us to compute its residual slope. Furthermore, if $\mathcal{Z}$ and the other irreducible components do not meet transversally at $\mathrm{x}$, then we establish the existence of a nonzero Galois cohomology class with a prescribed local behavior at $p$, contradicting either hypothesis \eqref{tag:discriminant} or \eqref{tag:nonvanishing_resultant_Q_PL}. Additionally, we prove that hypotheses \eqref{tag:discriminant} and \eqref{tag:nonvanishing_resultant_Q_PL} follow from the weak $p$-adic Schanuel Conjecture in transcendental number theory when $\rho$ is exotic, and from the (much weaker) $p$-adic Four Exponentials Conjecture when $\rho$ has RM.  For instance, if $\rho\simeq \Ind_K^\Q\psi$ has RM by a real quadratic field $K$, then \eqref{tag:nonvanishing_resultant_Q_PL} unconditionally holds, and \eqref{tag:discriminant} takes the simple form
	\begin{equation}\label{eq:intro_rm_discriminant_condition}
4\cdot\log_p(u_{\ob{\Gp}})\cdot\log_p(u^\sigma_{\Gp})+\log_p\left(\frac{u_\Gp}{u^\sigma_{\ob{\Gp}}}\right)^2\neq0.
	\end{equation} 
	Here, $\sigma$ is the nontrivial automorphism of $K$, and $u_{\Gp},u_{\ob\Gp}$ are certain canonical $\ob{\Q}$-linear combinations of algebraic numbers attached to the totally odd character $\psi_{\ad}=\psi/\psi^\sigma$ of $\Gal(\ob{\Q}/K)$; these numbers also appear in the formulation of the Gross-Stark conjecture for $\psi_{\ad}$ (see \S\ref{sec:reformulation_consition_S_RM_case}). We note that the above inequality is  \eqref{eq:intro_rm_discriminant_condition} unconditionally satified if $\rho$ has both RM and CM (also known as the ``Klein case''). 
    
    In general, the conditions \eqref{tag:discriminant} and \eqref{tag:nonvanishing_resultant_Q_PL} express certain algebraic independence relations between $12$ $p$-adic logarithms of algebraic numbers which we represent as two matrices $L=(L_j)$ and $M=(M_{ij})$ of size $1\times 3$ and $3\times 3$ respectively. 
	They have natural connections with Iwasawa theory for the \emph{adjoint} representation $\ad^0(\rho)$ with trace zero. For example, hypothesis \eqref{tag:discriminant} is related to a condition involving the Greenberg--Stevens $\cL$-invariant  of a Hida family through $f_{\alpha}$. 
    Hypothesis \eqref{tag:nonvanishing_resultant_Q_PL} is related to Perrin-Riou's notion of a regular submodule (see Remark \ref{rem:p-adic_Stark_regulator}), appearing in the definitions of Perrin-Riou's $p$-adic regulator and Benois' $\cL$-invariant for $\ad^0f_\alpha$ along any Hida family passing through $f_\alpha$  (see \cite{PR95,benoiscrys}). We expect these conditions to play a role in future works on the adjoint $p$-adic $L$-functions for Hida families in the $p$-irregular setting.

	\subsection*{Applications to Iwasawa theory}
We will now present some consequences of our main result that have bearing on Iwasawa theory. 
Assume $p$ is odd.
Let $Y_1(Np^r)_{/\Q}$  be the moduli space for pairs $(E, P)$, where $E$ is an elliptic curve and $P \in E$ is a point of order $Np^r$, and $X_1(Np^r)_{/\Q}$ its compactification obtained by adding cusps. The ordinary $p$-adic \'etale cohomology \[
\mathcal{H}_{r}:=\mathrm{H}^1_{\mathrm{et}}(X_1(Np^r)_{/\ob{\Q}},\Z_p)^{\ord}
\] is defined with respect to the dual Hecke correspondence $\mathbf{U}_p$. The limit 
$\mathcal{H}_{\infty}=\varprojlim_r \mathcal H_r$ is a finite type module over the Iwasawa algebra $\boldsymbol{\Lambda}_{\Z_p}:=\Z_p \lsem 1+p\Z_p \rsem$. It is naturally endowed with a $\boldsymbol{\Lambda}_{\Z_p}$-linear action of Hida's big $p$-ordinary cuspidal Hecke algebra $\mathfrak{h}$ compatible with the action of $G_{\Q}$.  By means of $p$-adic Hodge theory, Ohta \cite{Ohta95,Ohta2000}  showed that $\mathcal H_{\infty}$ fits into an exact sequence of $\gh[G_{\Q_p}]$-modules
\begin{equation}\label{nearbydualinto}
	0 \to \mathfrak{h}      \to  \mathcal{H}_{\infty} \to \Hom_{\boldsymbol{\Lambda}_{\Z_p}}(\mathfrak{h},\boldsymbol{\Lambda}_{\Z_p})   \to 0. \end{equation}

The Galois representation $ \mathcal{H}_{\infty}$ and its analogue constructed out of the étale cohomology of $Y_1(Np^r)$ play a pivotal role in Iwasawa theory. For example, the study of the Eisenstein part of $\cH_\infty$ is crucial in the proof of Mazur-Wiles' theorem \cite{mazurwiles,Ohta2000} and Sharifi's Conjecture \cite{sharifi2011,fukayakato}. When $\gh$ is Gorenstein, one can show that $\cH_\infty$ is $\gh$-free (see \eqref{nearbydual}); this allows, for example, to define canonical periods for the construction of the Mazur-Kitagawa $p$-adic $L$-functions \cite{emerton2006variation}. More generally, understanding its structure as a Hecke module proves very useful in establishing reciprocity laws for $p$-adic families of Euler systems in contexts of Rankin-Selberg or triple product $p$-adic $L$-functions (see \cite{BDR1,BDR2,BDV,BSVAsterisque1,BSV,LZ16} and \cite[\S4.2.7]{burungale2024zetaelementsellipticcurves}).

Let  $\mathcal{H}^{\pm}_{\infty}$ be the $(\pm)$-part of $\mathcal{H}_{\infty}$ for the action of the complex conjugation. In the following, we describe the Hecke and Galois module structure of the localizations of $\mathcal{H}_{\infty}$ and  $\mathcal{H}^{\pm}_{\infty}$ at the prime ideal of $\gh$ corresponding to $f_\alpha$.

\begin{theoremletter}[cf. Theorem \ref{thm:freeness} and Corollary \ref{non-splittingOhta}]\label{thm:main_application} Let $f_\alpha$ be the $p$-stabilization of a $p$-irregular  weight one newform $f$. Let $\mathfrak{p}_{f} \in \Spec \mathfrak{h}$ be the height one prime ideal corresponding  to the Hecke eigensystem of $f_\alpha$. Assume that the hypotheses  \eqref{tag:discriminant} and \eqref{tag:nonvanishing_resultant_Q_PL} of \S\ref{sec:regulators} hold.  Then 
\begin{enumerate}

\item The modules $\mathcal{H}_{\infty,\mathfrak{p}_f}$ and $\mathcal{H}_{\infty,\mathfrak{p}_f}^{\pm}$ are not $\mathfrak{h}_{\mathfrak{p}_f}$-free.

\item The localization of \eqref{nearbydualinto} at $\gp_f$ does not split as Hecke modules over $\mathfrak{h}_{\gp_f}$.

\item $\mathcal{H}_{\infty,\mathfrak{p}_f}/\mathfrak{p}_f \mathcal{H}_{\infty,\mathfrak{p}_f}$ is isomorphic to $\rho \oplus \rho $ as a $G_{\Q}$-representation. 

\end{enumerate}
\end{theoremletter}

Our theorem provides in many cases a negative answer to Ohta's question whether \eqref{nearbydualinto} admits a splitting or not over $\gh$ \cite[p. 558]{Ohta2000}.
We note that the last result can be viewed as a \emph{multiplicity two} statement for a characteristic-zero analogue of
the multiplicity one question regarding the $p$-torsion of the Jacobian of the modular curve $X_0(Np)$, studied by Mazur, Ribet, Gross, Edixhoven, Buzzard, Kilford and Wiese among others \cite{mazureisenstein,ribet1990,gross1990,ribetstein,edixhovenweight,kilfordwiese,wiese2007multiplicities}.

As an important consequence of Theorem \ref{thm:main_application}, the point $\mathrm{x}$ of the eigencurve is not \emph{good} in the sense of \cite[Def. 8.1.2]{eigenbook}; thus, Bella\"iche's construction of a two-variable $p$-adic $L$-function does not apply around $\mathrm{x}$. However, since $\Hom_{\varLambda}(\cT, \varLambda)$ is a Cohen--Macaulay
fractional ideal of $\cT$, we believe that the classification of such ideals given by Bella\"iche in \cite{Be24} can be used to obtain a precise description of $\mathcal{H}^{\pm}_{\infty,\mathfrak{p}_f}$ in the category of derived $\cT$--modules. This would be a necessary step towards extending the study of Euler systems and reciprocity laws carried out  in fundamental works \cite{BSVAsterisque1,BSV,LZ16} to the $p$--irregular setting, and for the determination of the adjoint $L$-ideal around $\mathrm{x} \in \cC$ of \cite[\S9]{eigenbook}.

 We conclude by mentioning another consequence of Theorem \ref{thm:main_result}. The local behavior of the map the weight map $\kappa$ at the point of the eigencurve is intimately related to the generalized eigenspace for the corresponding system of Hecke eigenvalues in the space of overconvergent forms. Theorem \ref{thm:side_result_DLR} settles a conjecture of Darmon--Lauder--Rotger on the dimension of the generalized eigenspace attached to $f_\alpha$ (cf.  \cite[Conj. 4.1]{DLR4}). This is a critical technical assumption in the formulation of the Elliptic Stark Conjecture in the $p$-irregular setting. The latter exploits the arithmetic of overconvergent weight 1 forms to produce formulas for regulators of global points for Artin twists of elliptic curves in the uncharted rank 2 setting.

\subsection*{Structure of the paper}
In \S\ref{sec:preliminary} we collect results on the cohomology of Artin representations and  $p$-adic transcendental number theory. These are used in \S\ref{sec:selmer_groups} to study two Selmer groups associated with $\rho$  and a choice of line $V^+$ in the representation space $V$.  The Selmer group for $\ad\rho$ is shown to be trivial for at most four choices of $V^+$. In \S\ref{HidaLinvariantandcslope}, we relate hypotheses \eqref{tag:discriminant} and \eqref{tag:nonvanishing_resultant_Q_PL} to a Greenberg-Stevens $\cL$-invariant  and results and conjectures in transcendence theory.

The proof of Theorem \ref{thm:main_result} is the technical heart of the paper. It relies on a novel  approach to establish the smoothness of the components of $\Spec \cT$ through the study of $p$-ordinary deformations of $\rho$ over thickenings of the normalization $\widetilde{\cT}$ of $\cT$. We refer the reader to \S \ref{overview1} and \S \ref{overview2} for an overview of the strategy. 
In \S\ref{sec:Ohta} we prove Theorem \ref{thm:main_application}. As an additional application of Theorem \ref{thm:main_result}, we also include a description of the triangulation of Kedlaya--Pottharst--Xiao around an irregular weight one point of the eigencurve.
\medskip

{ \noindent{\bf Acknolwedgements:} { \small
We would like to thank  Jo\"{e}l Bella\"{i}che, Denis Benois, Ashay Burungale, Kazim Buyukboduk, Henri Darmon, Shaunak Deo, Mladen Dimitrov, Francesc Fit\'e, Wassilij Gnedin, Harald Grobner, Ming-Lun Hsieh, Mahesh Kakde, David Loeffler and Sarah Zerbes for helpful comments and for many stimulating conversations. }

	\subsection*{Notation}
Let $G_\Q=\Gal(\ob{\Q}/\Q)$ and, for any prime number $\ell$, fix an embedding $\iota_\ell \colon \ob{\Q} \hookrightarrow \ob{\Q}_\ell$. Let $I_\ell$ be the inertia subgroup of $G_{\Q_\ell}=\Gal(\ob{\Q}_\ell/\Q_\ell)$, which we see as a subgroup of $G_\Q$ via $\iota_\ell$.

\subsubsection*{Galois representations and Selmer groups} 
We shall use the following conventions.
\begin{itemize}    
\item Let $\epsilon_p:G_\Q  \twoheadrightarrow \Gal(\Q(\mu_{p^{\infty}})/\Q) = \Z_p^\times$ be the $p$-adic 
cyclotomic character. Let $\Q_p(1)$ denote the one-dimensional $G_\Q$-representation on which $G_{\Q}$ acts by $\epsilon_p$.

\item The $p$-adic Teichm\"uller character is obtained by the composition \[\omega_p:G_\Q\twoheadrightarrow \Gal(\Q(\mu_{2p})/\Q) = (\Z/2p)^\times \hookrightarrow  \Z_p^\times\].
\item Let $\nu=2$ if $p=2$ and $\nu=1$ otherwise.  The cyclotomic $\Z_p$-extension  $\Q_{\cyc}$ of $\Q$ is the fixed field of 
$\epsilon_p\omega_p^{-1}:G_\Q\to 1+p^\nu\Z_p$.

\item The universal cyclotomic character   $\chi_{\mathrm{cyc}}: G_\Q\to \varLambda^\times$ is obtained by composing 
$\epsilon_p\omega_p^{-1}$ with 
\begin{equation}\label{eq:univ-cyc}
	1+p^\nu\Z_p \to \Z_p\lsem 1+p^\nu\Z_p\rsem^\times \xrightarrow{\sim}\Z_p\lsem X \rsem^\times\hookrightarrow
	\bar\Q_p \lsem 1+p^\nu\Z_p \rsem ^\times = \varLambda^\times
\end{equation}
where the isomorphism in the middle sends  $1+p^\nu$ to $1+X$. Let $\log_p\colon\ob{\Q}_p^{\times}\rightarrow \ob{\Q}_p$ be the branch of the $p$-adic logarithm sending $p$ to $0$.

\item	Given a continuous Galois representation $V$ of the absolute Galois group of a finite unramified extension $E$ of $\Q_p$, we let 
	\[\HH_\rf^1(E,V)=\ker\left[\HH^1(E,V) \to \HH^1(E,V \otimes B_{\cris})\right]\]
	be the usual (local) Bloch-Kato Selmer group associated with $V$, where $B_{\cris}$ denote the crystalline period ring endowed with the semi--linear Frobenius $\Phi$ and the natural $G_{\Q_p}$-action.
	
\item Given a geometric representation $V$ of the absolute Galois group $G_H$ of a number field $H$ and a finite set $S$ of places of $H$, we put
	\[\rH^{1}_{\rf,S}(H,V)=\ker[\rH^1(H,V) \to \prod_{\Gp|p}\frac{\HH^1(H_\Gp,V)}{\HH^1_\rf(H_\Gp,V)} \times \prod_{\ell\not\in S}  \rH^1(I_{\ell}\cap G_H,V)].\]
	We simply write $\HH^1_{\rf,S}$ as $\HH^1_{\rf}$ (resp. as $\HH^1_{\rf,p}$) if $S$ is empty (resp. if $S=\{p\}$).
	\end{itemize}
	
\subsubsection*{Rigid analytic Geometry} 	Let $\mathcal Y /\Q_p$ be a reduced affinoid space. \begin{itemize}
    \item 
We let  $\cO_{\mathcal{Y}}^\circ$ be the subring of  power-bounded elements in $\cO_{\mathcal{Y}}$. We endow the $\Q_p$-algebra $\cO(\mathcal{Y})$ with the coarsest locally convex topology with respect to the Banach topologies of admissible open affinoids $\mathcal{U} \subset \mathcal{Y}$ and the restriction homomorphism $ \cO(\mathcal{Y}) \to \cO(\mathcal{U})$.
\item For any $y \in \mathcal{Y}$, the local ring $\cO_{\mathcal{Y},y}$ is a Henselian local topological ring  endowed with the finest locally convex topology such that, for any admissible affinoid $\mathcal{U}$ of $\mathcal{Y}$ containing $y$, the localization morphism $\cO(\cU) \to \cO_{\mathcal{Y},y}$ is continuous.

\item For any finite type $ \cO_{\mathcal{Y},y} $-module $\mathcal{M}$, we endow $\mathcal{M}$ with the quotient locally convex topology associated to a surjection $  \cO_{\mathcal{Y},y}^n \twoheadrightarrow \mathcal{M}$. The resulting  topology on $\mathcal{M}$ does not depend on the choice of the presentation.
\end{itemize}

	\section{Cohomology of Artin representations}\label{sec:preliminary}
	
	This section collects some preliminary known facts on $p$-adic transcendence theory and cohomology of Artin representations. As we shall see in \S \ref{sec:geometry}-\S \ref{sec:geometry2}, the geometry of Hida families passing through an irregular weight one cusp form is controlled by a $\textbf{P}^1$-family of Selmer classes in $\HH^1(\Q,\ad\rho)$. It will be convenient to work with the cohomology of the arithmetic dual $(\ad\rho)^\vee$ of $\ad\rho$; Poitou--Tate duality leads to computing certain Selmer quotients of the cokernel of the localization map
	\[\loc_p\colon \HH^1_{\rf,p}(\Q,(\ad\rho)^\vee) \to\HH^1(\Q_p,(\ad\rho)^\vee).\]
	Crucially, the domain of this map admits a (motivic) $\ob{\Q}$-structure whose image under $\loc_p$ can be explicitly calculated in terms of $p$-adic logarithms and $p$-adic valuations of $\ob{\Q}$-linearly independent algebraic numbers.
	
\subsection{$p$-adic transcendence theory}\label{sec:transcendence}
We recall in this section the main results from $p$-adic transcendence theory that will be needed in this article. We identify in what follows $\ob{\Q}$ with as subfield of $\ob{\Q}_p$ via $\iota_p$.

\begin{conj}[Weak $p$-adic Schanuel conjecture]\label{conj:schanuel}
		Let $\alpha_1,\ldots,\alpha_n$ be nonzero algebraic numbers. If $\log_p(\alpha_1),\ldots,\log_p(\alpha_n)$ are linearly independent over $\Q$, then they are algebraically independent over $\Q$. 
\end{conj}

The following two results are applications of works of Brumer, Waldschmidt and Roy in $p$-adic transcendence theory.

\begin{prop}\label{prop:baker_brumer}
	Let $H\subset \ob{\Q}$ be a number field. The $\ob{\Q}$-linear extension $\log_p \colon \ob{\Q} \otimes_\Z H^\times \rightarrow \ob{\Q}_p$, $c\otimes x \mapsto c\log_p(\iota_p(x))$ of the $p$-adic logarithm has kernel the line $p^{\ob{\Q}}$ spanned by $1\otimes p$.
\end{prop}
\begin{proof}
	This is an application of the theorem of Baker--Brumer; see \cite[Prop. 2.2]{maksoud_documenta}.
\end{proof}

In the following proposition, we let $\mathbb{L}$ be the $\ob{\Q}$-linear subspace of $ \ob{\Q}_p$ generated by $1$ and by $p$-adic logarithms of non-zero algebraic numbers. 
\begin{prop}\label{prop:roy_waldschmidt}
		Let $N$ be a matrix with entries in $\mathbb{L}$, of size $m\times \ell$ with $m,\ell>0$, and let $n$ be its rank. Assume that all the entries of $N$ are $\ob{\Q}$-linearly independent. Then
	\[n \geq \frac{\ell \cdot m}{\ell+m}.\]
\end{prop}
\begin{proof}
	This follows from \cite[Thm. 2.3]{maksoud_documenta}, where the ratio denoted $\theta(N)$ is equal to $\ell/m$.
\end{proof}

The following conjecture is a variant of the $p$-adic four exponentials conjecture which does not follow from Proposition \ref{prop:roy_waldschmidt}, but is much weaker than Conjecture \ref{conj:schanuel}.

\begin{conj}\label{conj:4exp} 
Let $N$ be a $2\times2$ matrix with entries in $\mathbb{L}$. Assume that all the entries of $N$ are $\ob{\Q}$-linearly independent. Then $N$ has rank $2$.
\end{conj}

\subsection{Cohomology of Artin representations}\label{sec:cohomology_Artin}
We recall some results relating the Galois cohomology of Artin representations to $S$-units.

\subsubsection{Cohomology of $\ob{\Q}_p$} \label{sec:cohomology_Qp} The additive homomorphism $\lambda=\frac{\log_p(\epsilon_p)}{\log_p(1+p^\nu)}$
is a canonical basis of 
\[\HH^1(\Q,\ob{\Q}_p)= \Hom(G_\Q,\ob{\Q}_p).\]
We still denote by $\lambda$ its restriction to $G_{\Q_p}$. Locally at $p$, we have a direct sum decomposition
\begin{equation}\label{eq:decomposition_H1_local_Qp}
	\HH^1(\Q_p,\ob{\Q}_p) = \HH^1_\rf(\Q_p,\ob{\Q}_p)\oplus \HH^1_{\cyc}(\Q_p,\ob{\Q}_p),
\end{equation}
with $\HH^1_\rf(\Q_p,\ob{\Q}_p)=\Hom(\Gal(\Q_p^{\ur}/\Q_p),\ob{\Q}_p)$ and $\HH^1_{\cyc}(\Q_p,\ob{\Q}_p)=\Hom(\Gal(\Q_p\cdot\Q_{\cyc}/\Q_p),\ob{\Q}_p)$.
We identify $\Gal(\Q_p^{\ur}/\Q_p)$ with $p^\Z$ via $\Frob_p \mapsto p$, and $\Gal(\Q_p\cdot\Q_{\cyc}/\Q_p)$ with $1+p\Z_p$ via $\epsilon_p\omega^{-1}_p$.
Letting $\ord\colon p^{\Z} \to \ob{\Q}_p$ be the $p$-adic valuation map, we have
\begin{equation}\label{eq:bases_H1_local_Qp}
	 \HH^1_\rf(\Q_p,\ob{\Q}_p)=\ob{\Q}_p\cdot\ord, \quad  \HH^1_{\cyc}(\Q_p,\ob{\Q}_p)=\ob{\Q}_p\cdot\lambda.
\end{equation}

\subsubsection{Cohomology of $\ob{\Q}_p(1)$} \label{sec:cohomology_Qp1} Given a number field $H\subset \ob{\Q}$ and a finite set of finite places $S$ of $H$, let $\cO_{H,S}^\times$ be the group of $S$-units of $H$. Then Kummer theory provides an identification
\begin{equation}\label{eq:S-units}
	\HH^1_{\rf,S}(H,\ob{\Q}_p(1))\simeq U_H^{S}, 
\end{equation}
where $U_H^{S}:= \cO_{H,S}^\times \otimes_{\Z} \ob{\Q}_p$  (\cite[\S2.1.2]{dimitrov-maksoud}).
Locally at $p$, we have the direct sum decomposition
\begin{equation}\label{eq:decomposition_H1_local_Qp(1)}
	\HH^1(\Q_p,\ob{\Q}_p(1))=\HH^1_\rf(\Q_p,\ob{\Q}_p(1))\oplus \HH^1_{\cyc}(\Q_p,\ob{\Q}_p(1)),
\end{equation}
with $\HH^1_\rf(\Q_p,\ob{\Q}_p(1))=(1+p^\nu\Z_p)\otimes_{\Z_p}\ob{\Q}_p$ and $\HH^1_{\cyc}(\Q_p,\ob{\Q}_p(1))=p^{\Z}\otimes_{\Z}\ob{\Q}_p=p^{\ob{\Q}_p}$. We trivialize these two lines via
\begin{equation}\label{eq:bases_H1_local_Qp(1)}
	\HH^1_\rf(\Q_p,\ob{\Q}_p(1))=\ob{\Q}_p\cdot \lambda^\vee \overset{\Id\otimes\lambda}{\underset{\sim}{\longrightarrow}}\ob{\Q}_p, \qquad \HH^1_{\cyc}(\Q_p,\ob{\Q}_p(1))=\ob{\Q}_p \cdot \ord^\vee\overset{\Id\otimes\ord}{\underset{\sim}{\longrightarrow}}\ob{\Q}_p,
\end{equation}
where $\lambda^\vee:=1+p^\nu$ and $\ord^\vee:=p$. Moreover, $(\lambda^\vee,\ord^\vee)$ is the dual basis of $(\lambda,\ord)$ under Tate's local duality $\HH^1(\Q_p,\ob{\Q}_p)\times \HH^1(\Q_p,\ob{\Q}_p(1))\to \ob{\Q}_p$, as one can check from its explicit description in \cite[Cor. 7.2.13]{neukirch-schmidt-wingberg}.
\subsubsection{Cohomology of an Artin representation}\label{sec:cohomology_artin_representation}
Let $(\pi,W)$ be a $\ob{\Q}_p$-valued $G_\Q$-representation of finite image, and let $H/\Q$ be the finite Galois extension cut out by $\pi$. We define 
\begin{equation}\label{dualdefn}
\widecheck{W}=\Hom_{\ob{\Q}_p}(W,\ob{\Q}_p(1)),
\end{equation} 
the arithmetic dual of $W$. For any finite set $S$ of finite places of $H$, the inflation-restriction exact sequence and \eqref{eq:S-units} provide an identification 
\begin{equation}\label{eq:global_kummer_artin}
	\HH^1_{\rf,S}(\Q,\widecheck{W}) \simeq \Hom_{G_\Q}(W,\HH^1_{\rf,S}(H,\ob{\Q}_p(1)) \simeq \Hom_{G_\Q}(W,U_H^S)
\end{equation}
(see e.g. \cite[Lem. 2.2]{dimitrov-maksoud}). Let $w$ be the place of $H$ determined by $\iota_p$. If $\pi$ is trivial on $G_{\Q_p}$, then $H_w=\Q_p$ and $\HH^1(\Q_p,\widecheck{W})=\Hom_{\ob{\Q}_p}(W,\HH^1(\Q_p,\ob{\Q}_p(1)))$ canonically. Furthermore, we may see $\lambda$ and $\ord$ as maps on $\cO_{H,S}^\times$ using $\iota_p$, so their $\ob{\Q}_p$-linear extensions define $\ob{\Q}_p$-linear forms on $U_H^S$. We obtain a commutative diagram
\begin{equation}\label{eq:diagram_localization}
	\begin{tikzcd}
	\HH^1_{\rf,S}(\Q,\widecheck{W}) \ar[d, "\sim" {anchor=south, rotate=90}] \ar[r, "\loc_p"] & \HH^1(\Q_p,\widecheck{W}) \ar[d, "\sim" {anchor=south, rotate=90}] \\ \Hom_{G_\Q}(W,U_H^S) \ar[r] & \Hom_{\ob{\Q}_p}(W,\ob{\Q}_p^{\oplus 2}),
\end{tikzcd}
\end{equation}
where $\loc_p$ is the localization map, the vertical isomorphisms are induced by \eqref{eq:global_kummer_artin}, \eqref{eq:decomposition_H1_local_Qp(1)} and \eqref{eq:bases_H1_local_Qp(1)}, and the bottom map is induced by $U_H^S \to \ob{\Q}_p^{\oplus 2}$, $u\otimes1 \mapsto (\lambda(u),\ord(u))$. If $\pi_{|G_{\Q_p}}$ is not trivial, then a similar diagram as \eqref{eq:diagram_localization} exists but we will not need it (see e.g. \cite[Rem. 3.9]{maksoud_documenta}).

As $\pi$ has finite image, descent theory provides us with a $\ob{\Q}$-rational structure $W_{\ob{\Q}}$ of $W$, that is, any $\ob{\Q}$-vector space endowed with a $G_\Q$-action and a $G_\Q$-equivariant isomorphism $W_{\ob{\Q}}\otimes_{\ob{\Q}}\ob{\Q}_p\simeq W$. 
On $U_H^S$ there is also the $\ob{\Q}$-structure $U^S_{H,\ob{\Q}}=\cO_{H,S}^\times\otimes_{\Z} \ob{\Q}$, and it makes sense to evaluate $\lambda$ and $\ord$ on $U^{S}_{H,\ob{\Q}}$ using $\iota_p$. The next lemma will be later used to analyze the horizontal bottom map in \eqref{eq:diagram_localization}. Notice that the source of this map has a $\ob{\Q}$-structure given by $\Hom_{G_\Q}(W_{\ob{\Q}},U^S_{H,\ob{\Q}})$.

\begin{lemma}\label{lem:Dirichlet}
	Assume $W$ does not contain the trivial representation and let $W_{\ob{\Q}}$ be any $\ob{\Q}$-structure of $W$. Let $S$ be a finite set of finite places of $H$ and put 
	\[d_+^{S}= \dim_{\ob{\Q}}\HH^0(\R,W_{\ob{\Q}})+ \sum_{\ell\in S}\dim_{\ob{\Q}}\HH^0(\Q_\ell,W_{\ob{\Q}}).\]
	\begin{enumerate}
		\item The $\ob{\Q}$-dimension of $\Hom_{G_\Q}(W_{\ob{\Q}},U^S_{H,\ob{\Q}})$ is equal to $d_+^{S}$.
		\item Let $\mathscr{B}$ (resp. $\mathscr{C}$) be a $\ob{\Q}$-basis of $\Hom_{G_\Q}(W_{\ob{\Q}},U_{H,\ob{\Q}}^{S})$ (resp. of $W_{\ob{\Q}}$). If $\pi$ is irreducible, then the elements $\lambda(\kappa(w))$ for $\kappa\in\mathscr{B},w\in\mathscr{C}$ form a set of $\ob{\Q}$-linearly independent elements of $\ob{\Q}_p$, hence $\ob{\Q}$-algebraically independent under Conjecture \ref{conj:schanuel}. 
		\item More generally, if $(\pi^{(1)},W_{\ob{\Q}}^{(1)}),\ldots,(\pi^{(n)},W_{\ob{\Q}}^{(n)})$ are nontrivial pairwise distinct irreducible representations of $\Gal(H/\Q)$ over $\ob{\Q}$ and, for $1\leq i \leq n$, $\mathscr{B}^{(i)}$ (resp. $\mathscr{C}^{(i)}$) are $\ob{\Q}$-bases of $\Hom_{G_\Q}(W^{(i)}_{\ob{\Q}},U_{H,\ob{\Q}}^{S})$ (resp. of $W^{(i)}_{\ob{\Q}}$), then the same conclusion as in (ii) holds for the elements of the set $\cup_{i=1}^n\{\lambda(\kappa(w)) \colon \kappa\in \mathscr{B}^{(i)},w\in\mathscr{C}^{(i)}\}$.
		\item Take $S=\{p\}$. The post-composition with $\ord$ induces an isomorphism 
		\[\Hom_{G_\Q}(W_{\ob{\Q}},U_{H,\ob{\Q}}^{(p)})/\Hom_{G_\Q}(W_{\ob{\Q}},U_{H,\ob{\Q}}) \simeq \Hom_{G_{\Q_p}}(W_{\ob{\Q}},\ob{\Q}).\]
	\end{enumerate}
\end{lemma}

\begin{proof}
	(i) is a well-known application of Herbrand's strengthening of Dirichlet's unit theorem (see \cite[Lem. 3.6 (3)]{maksoud_documenta}).  We explain the proof of (ii). Notice first that it suffices to prove (ii) for a particular choice of $\mathscr{B}$, so we fix $\mathscr{C}$ and define $\mathscr{B}$ as follows. As $\pi$ is assumed to be irreducible, one may consider its idempotent $e_\pi\in\ob{\Q}[\Gal(H/\Q)]$. By (i), we have an abstract isomorphism
	\begin{equation}\label{eq:idempotent_isom}
		e_\pi \cdot U_{H,\ob{\Q}}^{S} \simeq (W_{\ob{\Q}})^{\oplus d_+^{S}}
	\end{equation} 
	of $\ob{\Q}[\Gal(H/\Q)]$-modules. Fix such an isomorphism, and let $\mathscr{B}$ be the set of elements of the form
	\[ \kappa \colon W_{\ob{\Q}} \to (W_{\ob{\Q}})^{\oplus d_+^{S}} \to U_{H,\ob{\Q}}^{S},\]
	where the first map is one of the $d_+^{S}$ canonical inclusions and the second map is the inverse of \eqref{eq:idempotent_isom}. Then $\mathscr{B}$ is a basis of $\Hom_{G_\Q}(W_{\ob{\Q}},U_{H,\ob{\Q}}^{S})$ such that the linear subspaces $\mathrm{im}(\kappa)$, for $\kappa$ varying in $\mathscr{B}$, are in direct sum by construction. Therefore, the elements of the set $\{\kappa(w)\colon \kappa\in\mathscr{B},w\in\mathscr{C}\}$ are $\ob{\Q}$-linearly independent. Finally, Proposition \ref{prop:baker_brumer} shows that the kernel of the $\ob{\Q}$-linear map $\lambda \colon U_{H,\ob{\Q}}^{S} \to \ob{\Q}_p$ is contained in the line generated by $p$. Note that $e_\pi$ kills $p$ as $\pi$ is not the trivial representation by assumption, so the restriction of $\lambda$ to $e_\pi\cdot U_{H,\ob{\Q}}^{S}$ is injective and the elements of the set $\{\lambda(\kappa(w))\colon \kappa\in\mathscr{B},w\in\mathscr{C}\}$ are $\ob{\Q}$-linearly independent as well. Since $\ob{\Q}/\Q$ is algebraic, these elements are $\ob{\Q}$-algebraically independent under Conjecture \ref{conj:schanuel}.
	
	The proof of (iii) is nearly identical to that of (ii): we use that $\lambda$ is injective on $\oplus_{i=1}^n e_{\pi^{(i)}}\cdot U_H^S$, and that this sum is indeed a direct sum because the idempotents $e_{\pi^{(i)}}$ are orthogonal.
	
	For (iv), put $G=\Gal(H/\Q)$, $G_p=\Gal(H_w/\Q_p)$ and consider the short exact sequence of left $G$-modules
	\[\begin{tikzcd}
		0 \rar & \cO_H^\times \rar &\cO_H[\tfrac{1}{p}]^\times \rar & \prod_{g\in G/G_p} \Z,
	\end{tikzcd}\]
	where the last map is $\prod_{g\in G/G_p} \ord \circ g^{-1}$ and has a finite cokernel. We may apply to this exact sequence the functor $M \mapsto \Hom_G(W_{\ob{\Q}},M\otimes_{\Z}\ob{\Q})$ which is exact, as $G$ is finite. Now, the projection $\prod_{g\in G/G_p} \ob{\Q} \to \ob{\Q}$ to the factor corresponding to $g=\Id$ induces, via Frobenius reciprocity, an isomorphism $\Hom_{G}(W_{\ob{\Q}},\prod_{g\in G/G_p}\ob{\Q})\simeq \Hom_{G_{p}}(W_{\ob{\Q}},\ob{\Q})$. Therefore, post-composition by $\ord$ yields a surjection $\Hom_{G_\Q}(W_{\ob{\Q}},U_{H,\ob{\Q}}^{(p)})\twoheadrightarrow\Hom_{G_{\Q_p}}(W_{\ob{\Q}},\ob{\Q})$ with kernel $\Hom_{G_\Q}(W_{\ob{\Q}},U_{H,\ob{\Q}})$.
\end{proof}

\section{Families of adjoint Selmer groups and units}\label{sec:selmer_groups}

Let $\mathcal{F}=\sum_{n \geq 1} \mathbf{a}_n(\mathcal{F})q^n$ be a Hida eigenfamily specializing to $f_\alpha$ in weight $1$ with coefficients in a finite flat extension $\mathbf{I}$ of $\Z_p\lsem X \rsem$. Then the big $\mathbf{I}$-adic Galois representation $$\rho_{\mathcal{F}}\colon G_\Q \to \GL_2(\mathrm{Frac}(\mathbf{I}))$$ attached to $\mathcal{F}$ is $p$-ordinary, i.e., it has a $G_{\Q_p}$-unramified quotient of generic rank one over $\mathbf{I}$, and on which $\Frob_p$ acts by multiplication by $\mathbf{a}_p(\mathcal{F})$. A normalization process and specialization to $f_\alpha$ explained in \S\ref{sec:etaleness}  gives rise to a $G_{\Q_p}$-stable ordinary line $V^+(\cF)$ of the underlying $\ob{\Q}_p$-vector space $V$ of $\rho$. Conversely, any line $V^+$ of $V$ may \textit{a priori} come from a Hida family, as $\rho$ is unramified at $p$ and scalar at $\Frob_p$. As a result, the ordinary deformation problem of $\rho$ is not even representable!  This is in stark contrast to the $p$-regular case, where $\rho$ possesses a unique $p$-ordinary line with a $G_{\Q_p}$-unramified quotient on which $\Frob_p$ acts by $\alpha$.

One main goal of this section is to determine the possible $p$-ordinary lines of $\rho$ coming from Hida families. We consider a Selmer structure $K_{V^+}\subseteq\HH^1(\Q_p,\ad\rho)$ attached to any line $V^+\in\textbf{P}(V)$ and whose interest lies in the fact that, if $V^+=V^+(\cF)$, then $\cF$ gives rise to a nonzero Selmer class for $K_{V^+}$. In other words, if $V^+=V^+(\cF)$, then the localization map
\[\loc_{V^+}\colon \HH^1(\Q,\ad\rho)\to \dfrac{\HH^1(\Q_p,\ad\rho)}{K_{V^+}}\]
has a nontrivial kernel. After fixing an adequate basis of $V$ and parameterizing $V^+$ by $s\in\textbf{P}^1(\ob{\Q}_p)$, we find in Proposition \ref{prop:dimension_selmer_ad_rho} a remarkable degree $4$ polynomial $\rQ(S)$ whose roots are precisely the parameters $s$ of lines $V^+$ with $\ker(\loc_{V^+})\neq\{0\}$. Its coefficients are obtained from explicit $\ob{\Q}$-linear combinations of elements in $\log_p(\bar\Q^\times)$ which we define in \S\ref{sec:vanishing_tgt_spaces}.

In fact, the construction of the nonzero Selmer class in $\ker(\loc_{V^+(\cF)})$ that is needed to determine the position of $V^+(\cF)$ in $V$ will be contingent to the vanishing of another Selmer class in $\ker(\loc_{V^+(\cF)})$ whose trace is $0$. The second main result of this section, which is stated in Proposition \ref{prop:dimension_selmer_ad_0_rho}, ensures that $\ker(\loc_{V^+})\cap \HH^1(\Q,\ad^0\rho)=\{0\}$ \emph{regardless} of the position of line $V^+$ in $V$. The proof of this fact will be given under the hypothesis \eqref{tag:non_vanishing_resultant_P_L_P_M}, which is slightly weaker than \eqref{tag:nonvanishing_resultant_Q_PL}.

All the arguments that will follow in the rest of the article can be readily adapted to the CM case. However, for the sake of brevity, the focus has been restricted to the non-CM case. We shall see in \S\ref{sec:CM_case} that the hypotheses \eqref{tag:discriminant} and \eqref{tag:nonvanishing_resultant_Q_PL} of Theorem \ref{thm:main_result} are equivalent to those appearing in Theorem B of \cite{betinadimitrovKatz}, and so a proof of Theorem \ref{thm:main_result} is already available by \cite{betinadimitrovKatz}.

\subsection{Adjoint Selmer groups associated with an ordinary filtration}\
	
	Let $\rho \colon G_\Q \to \GL_{\ob{\Q}_p}(V)$ be a two-dimensional odd irreducible and $\ob{\Q}_p$-valued Galois representation having a finite image, and let $\overline G \subset \mathrm{PGL}_{\ob{\Q}_p}(V)$ denote the projective image of $\rho$. It is well-known that one of the following three possibilities occurs:
	\begin{enumerate}
		\item[(CM)] $\rho\simeq\Ind_K^\Q \chi$, where $K/\Q$ is an imaginary quadratic field and $\chi $ is a finite order character of~$G_K$ such that $\chi/\chi^{\tau} \ne \mathbf{1}$. In this case,  $\overline G$ is the dihedral group $D_r$ with $2r$ elements. 
		\item[(RM)]$\rho\simeq \Ind_K^\Q\psi$, where $K/\Q$ is a real quadratic field and $\psi$ is a finite order character of~$G_K$ such that $\psi/\psi^{\sigma}(\tau)=-1$ and $\Gal(K/\Q)=\{1,\sigma\}$. In this case as well, $\overline G$ is dihedral.
		
		\item[(ex)] $\overline G$ is exotic, that is, $\overline G$ is isomorphic to $A_4$, $S_4$, or $A_5$. Here, $A_r$ (resp. $S_r$) denote the alternating (resp. symmetric) group on $r$ letters. 
	\end{enumerate} 
	Note that $\overline{G}$ can both be CM and RM, in which case $\overline{G}$ is isomorphic to the Klein group $\Z/2\Z\times\Z/2\Z$.
	\begin{definitionwn} We say that $\rho$ is CM (resp. RM or exotic) if $\overline G$ is CM (resp. RM or exotic). 
	\end{definitionwn}

	Hereafter, we make the following three assumptions.
	\begin{enumerate}
	\item $\rho$ is not induced from a character over an imaginary quadratic field, 
	\item $\rho$ is unramified at $p$, and 
	\item the action of $\Frob_p$ on $V$ is scalar. 
	\end{enumerate}
	We denote by 
	\begin{equation}
		W:=\End^0(V), \qquad \text{resp. $\mathbf{W}:=\End(V)$}
	\end{equation}
	the underlying $\ob{\Q}_p$-vector space of the traceless adjoint representation $\ad^0\rho$ (resp. of the adjoint representation $\ad \rho$), equipped with the usual conjugation action of $G_\Q$. 	
	Notice that $\ad\rho$ is trivial on $G_{\Q_p}$. In particular, $p$ splits completely in the field extension $H/\Q$ cut out by $\ad \rho$.

	Suppose we are given a $\ob{\Q}_p$-line $V^+\subseteq V$. 
	\begin{defn}\label{def:Selmer}
		\begin{enumerate}
			\item We define the Selmer group associated with $\ad^0 \rho$ and $V^+$ as
		\[\Sel(\ad^0\rho,V^+)=\{ [\xi]\in \HH^1(\Q,W) \,:\, \xi(G_{\Q_p})\cdot V^+ \subseteq V^+ \mbox{ and } \xi(I_p)\cdot V \subseteq V^+ \}. \]
		\item Similarly, we define the Selmer group associated with $\ad\rho$ and $V^+$ as
		\[\Sel(\ad\rho,V^+)=\{ [\xi]\in \HH^1(\Q,\mathbf{W}) \,:\, \xi(G_{\Q_p})\cdot V^+ \subseteq V^+ \mbox{ and } \xi(I_p)\cdot V \subseteq V^+ \}. \]
		\end{enumerate}
	\end{defn}	
	The local condition on the cocycle class $[\xi]$ does not depend on the choice of the representative $\xi$, as $\ad^0\rho$ and $\ad\rho$ are trivial on $G_{\Q_p}$. 
	
	There is a natural $G_\Q$-equivariant isomorphism $\mathbf{W}\simeq W\times\ob{\Q}_p$ given by 
	\[ (\mathrm{pr}_W,\tr)\colon  [\xi] \mapsto ([\xi-\frac{1}{2}\tr(\xi)\cdot \Id_V],\frac{1}{2}\tr(\xi)),\]
	where $\Id_V\in\mathbf{W}$ is the identity map. In the next lemma, we still denote by $\mathrm{pr}_W$ and $\tr$ the induced maps on Galois cohomology groups.
	
	\begin{lemma}\label{lem:first_lemma_selmer}
		\begin{enumerate}
			\item The map $\tr \colon \HH^1(\Q,\mathbf{W})\to \HH^1(\Q,\ob{\Q}_p)$ induces the short exact sequence
	\begin{equation}\label{eq:exact_sequence_with_two_Selmer_groups}
		0 \longrightarrow \Sel(\ad^0\rho,V^+) \longrightarrow \Sel(\ad\rho,V^+) \overset{\tr}{\longrightarrow} \HH^1(\Q,\ob{\Q}_p)=\ob{\Q}_p\cdot\lambda,
	\end{equation}
	where the last equality is recalled in \S\ref{sec:cohomology_Qp}.
	\item The map $\mathrm{pr}_W\colon \HH^1(\Q,\mathbf{W}) \to \HH^1(\Q,W)$ induces an isomorphism $\Sel(\ad\rho,V^+)\simeq \cS(V^+)$, where
	\begin{equation}\label{eq:definition_intermediate_selmer}
		\cS(V^+):=\{[\xi]\in\HH^1(\Q,W) \,:\, \xi(G_{\Q_p})\cdot V^+ \subseteq V^+\}.
	\end{equation}
		\end{enumerate}
	\end{lemma}
	\begin{proof}
		Using the decomposition $\HH^1(\Q,\mathbf{W})= \HH^1(\Q,W)\times\HH^1(\Q,\ob{\Q}_p)$ given by the map $(\mathrm{pr}_W,\tr)$, we see that $\Sel(\ad\rho,V^+)\cap \HH^1(\Q,W)=\Sel(\ad^0\rho,V^+)$, so the exactness of \eqref{eq:exact_sequence_with_two_Selmer_groups} follows. 
		
		Similarly, $\Sel(\ad\rho,V^+)\cap \HH^1(\Q,\ob{\Q}_p)=\{0\}$ because $\lambda$ is ramified, so $\mathrm{pr}_W$ maps $\Sel(\ad\rho,V^+)$ into $\cS(V^+)$. Now, if $[\xi^0]\in\cS(V^+)$, then we may project $\xi^0_{|G_{\Q_p}}$ to the space 
		\[\HH^1(\Q_p,\End(V^-))=\HH^1(\Q_p,\ob{\Q}_p)\cdot \Id_{V^-},\]
		where it can be written as $(x\cdot\lambda+y\cdot\ord)\cdot \Id_{V^-}$ by \eqref{eq:bases_H1_local_Qp}. Therefore, the class $[\xi]:=[\xi^0]-x\cdot\lambda\cdot\Id_V\in \HH^1(\Q,\mathbf{W})$ belongs to $\Sel(\ad\rho,V^+)$ and is mapped to $[\xi^0]$ under $\mathrm{pr}_W$. This finishes the proof of (ii).
	\end{proof}
		
	We fix a $\ob{\Q}$-structure $V_{\ob{\Q}}$ on $V$ and an isomorphism of $G_\Q$-modules $V_{\ob{\Q}}\otimes_{\ob{\Q}}\ob{\Q}_p \simeq V$. Let $(e_1,e_2)$ be a $\ob{\Q}$-basis of $V_{\ob{\Q}}$ which we also consider as a $\ob{\Q}_p$-basis of $V$ via the preceding identification. Using $(e_1,e_2)$, we may identify $W$ with $M_2(\ob{\Q}_p)^{\tr=0}=\{(\begin{smallmatrix}-d & b \\ c & d\end{smallmatrix}) \,:\, b,c,d\in\ob{\Q}_p\}$, and its $\ob{\Q}$-structure $W_{\ob{\Q}}=\End^0(V_{\ob{\Q}})$ with $M_2(\ob{\Q})^{\tr=0}$. 
	
	\begin{defn}\label{def:adapted}
	 We say that a $\ob{\Q}$-basis $(e_1,e_2)$ of $V_{\ob{\Q}}$ 
	 is \emph{adapted} to $V^+$ if $e_2$ does not generate $V^+$. In other words, $V^+$ has a basis of the form $e_1+s \cdot e_2$ for some $s\in\ob{\Q}_p$, and we call $s$ the slope of $V^+$ (with respect to $(e_1,e_2)$).
	\end{defn}
	 
	 Assume $(e_1,e_2)$ is adapted to $V^+$ and let $s$ be the slope of $V^+$. The conjugation by $(\begin{smallmatrix}1 & 0 \\ s & 1 \end{smallmatrix})$ of the canonical basis $\{(\begin{smallmatrix}0 & 1 \\ 0 & 0 \end{smallmatrix}),(\begin{smallmatrix}0 & 0 \\ 1 & 0 \end{smallmatrix}),(\begin{smallmatrix}1 & 0 \\ 0 & -1 \end{smallmatrix})\}$ of $W$ gives the elements 
	\begin{equation}\label{eq:def_w_bcd}
		w_\mathrm{b} := \begin{pmatrix}-s & 1 \\ -s^2 & s \end{pmatrix}, \qquad
		w_\mathrm{c} := \begin{pmatrix}0 & 0 \\ 1 & 0 \end{pmatrix}, \qquad 
		w_\mathrm{d} := \begin{pmatrix}1 & 0\\ 2s & -1 \end{pmatrix}.
	\end{equation}
	Let $W^\mathrm{b}:=\ob{\Q}_p\cdot w_\mathrm{b}$, $W^\mathrm{c}:=\ob{\Q}_p\cdot w_\mathrm{c}$ and $W^\mathrm{d}:=\ob{\Q}_p\cdot w_\mathrm{d}$. Note that $W^\rc=\Hom(V^+,V^-)$ and $W^\rd=(\Hom(V^+,V^+)\oplus \Hom(V^-,V^-))\cap W$, so
	\begin{equation}\label{eq:reformulation_Sel}
		\Sel(\ad^0\rho,V^+)=\ker\left[\HH^1(\Q,W) \to \HH^1(\Q_p,W^\rc)\times\HH^1(I_p,W^\rd)\right]
	\end{equation}
	and 
	\begin{equation}\label{eq:reformulation_S(V+)}
			\cS(V^+)=\ker\left[\HH^1(\Q,W) \to \HH^1(\Q_p,W^\mathrm{c})\right].
	\end{equation}
	For any symbol $e\in\{b,c,d\}$, $W^\mathrm{e}$ is trivial as a $G_{\Q_p}$-module, so we have a decomposition of $\HH^1(\Q_p,W^\mathrm{e})=\HH^1_\rf(\Q_p,W^\mathrm{e})\oplus\HH^1_{\cyc}(\Q_p,W^\mathrm{e})$ with
	\begin{equation}\label{eq:H1_W^e}
		\HH^1_\rf(\Q_p,W^\mathrm{e})=\ob{\Q}_p\cdot \ord \cdot w_\mathrm{e}\simeq \ob{\Q}_p, \qquad \HH^1_{\cyc}(\Q_p,W^\mathrm{e})=\ob{\Q}_p\cdot \lambda \cdot w_\mathrm{e}\simeq \ob{\Q}_p
	\end{equation}
		as in \S\ref{sec:cohomology_Qp}, whereas \eqref{eq:bases_H1_local_Qp(1)} identifies $\HH^1(\Q_p,\widecheck{W}^\mathrm{e})=\HH^1_\rf(\Q_p,\widecheck{W}^\mathrm{e})\oplus\HH^1_{\cyc}(\Q_p,\widecheck{W}^\mathrm{e})$ with $(\ob{\Q}_p)^{\oplus 2}$ via
	\begin{equation}\label{eq:H1_widecheck_W^e}
	\begin{tikzcd}[row sep=tiny]
				\HH_\rf^1(\Q_p,\widecheck{W}^\mathrm{e})=
			\Hom_{\ob{\Q}_p}(W^\mathrm{e},\HH^1_{\rf}(\Q_p,\ob{\Q}_p(1)))  \ar[rr,"\kappa \mapsto \lambda(\kappa(w_e))", "\sim"'] && \ob{\Q}_p, \\
			\HH_{\cyc}^1(\Q_p,\widecheck{W}^\mathrm{e})=
			\Hom_{\ob{\Q}_p}(W^\mathrm{e},\HH^1_{\cyc}(\Q_p,\ob{\Q}_p(1)))  \ar[rr,"\kappa \mapsto \ord(\kappa(w_e))", "\sim"'] && \ob{\Q}_p.
	\end{tikzcd}
	\end{equation}

	\begin{lemma}\label{lem:dimension_sel_ad_0}
		Let $\Upsilon^0\colon\Hom_{G_\Q}(W,U_H^{(p)}) \to (\ob{\Q}_p)^{\oplus 3}$ be the $\ob{\Q}_p$-linear map given by
		\[
		(\kappa \colon W \to U_H^{(p)}) \mapsto (\lambda(\kappa(w_{\mathrm{b}})),\ord(\kappa(w_{\mathrm{b}})),\ord(\kappa(w_{\mathrm{d}}))),
		\]
		and denote by $r^0$ its rank. Then $\dim\Sel(\ad^0 \rho,V^+)=3-r^0$.
	\end{lemma}
	
	\begin{proof}
		As $W$ is trivial as a $G_{\Q_p}$-module, we have from \S\ref{sec:cohomology_Qp} that $\ker[\HH^1(\Q_p,W^\mathrm{d}) \to \HH^1(I_p,W^\mathrm{d})]=\HH^1_\rf(\Q_p,W^\mathrm{d})$, so $\Sel(\ad^0\rho,V^+)$ sits in the short exact sequence
		\begin{equation}\label{eq:beginning_Poitou_Tate}
			\begin{tikzcd}
				0 \rar & \Sha^1_{\{p\}}(\Q,W) \rar & \Sel(\ad^0\rho,V^+) \rar & \HH^1(\Q_p,W^\mathrm{b})\times \HH^1_\rf(\Q_p,W^\mathrm{d}),
			\end{tikzcd}
		\end{equation}
		where $\Sha^1_{\{p\}}(\Q,W)=\ker[\HH^1(\Q,W) \to \HH^1(\Q_p,W)]$ is the first Tate-Shafarevich group of $W$. Note that $\Sha^1_{\{p\}}(\Q,W)=0$ by the finiteness of the class group of $H$. 

		Tate's local duality gives		\[ \HH^1(\Q_p,W^\mathrm{b})\times \HH^1_\rf(\Q_p,W^\mathrm{d})\simeq (\HH^1(\Q_p,\widecheck{W}^\mathrm{b}) \times \HH^1_{\cyc}(\Q_p,\widecheck{W}^\text{d}))^* \]
		where $(-)^*$ denotes the $\ob{\Q}_p$-linear dual (see \S\ref{sec:cohomology_Qp1}).
		Therefore, \eqref{eq:beginning_Poitou_Tate} and Poitou-Tate duality provides us with a natural identification
		\begin{equation}\label{eq:Poitou_Tate}
			\Sel(\ad^0\rho,V^+) \simeq \coker\left[\HH^1_{\rf,p}(\Q,\widecheck{W}) \overset{\loc^0}{\longrightarrow} \HH^1(\Q_p,\widecheck{W}^\mathrm{b}) \times \HH^1_{\cyc}(\Q_p,\widecheck{W}^\text{d}) \right]^*,
		\end{equation}
		where $\loc^0$ is the localization map composed with the natural projections obtained from the direct sums $\HH^1=\HH^1_\rf\oplus \HH^1_{\cyc}$ and $\widecheck{W}=\widecheck{W}^\mathrm{b}\oplus\widecheck{W}^\mathrm{c}\oplus\widecheck{W}^\mathrm{d}$. By  \eqref{eq:diagram_localization}, there is  a commutative diagram
		\[
		\begin{tikzcd}
			\HH^1_{\rf,p}(\Q,\widecheck{W}) \ar[rrr,"\loc^0"] \ar[d,"\sim" {anchor=south, rotate=90}] &&& \HH^1_\rf(\Q_p,\widecheck{W}^\mathrm{b}) \times \HH^1_{\cyc}(\Q_p,\widecheck{W}^\mathrm{b}) \times \HH^1_{\cyc}(\Q_p,\widecheck{W}^\text{d}) \ar[d,"\sim" {anchor=south, rotate=90}] \\
			\Hom_{G_\Q}(W,U_H^{(p)}) \ar[rrr, "\Upsilon^0"] &&& (\ob{\Q}_p)^{\oplus 3},
		\end{tikzcd}
		\]
		where the right vertical map is induced from \eqref{eq:H1_widecheck_W^e}. Hence, $\dim \Sel(\ad^0\rho,V^+)=3-\rank \Upsilon^0$.
	\end{proof}
	
	\begin{lemma}\label{lem:dimension_sel_ad}
		Let $\Upsilon\colon\Hom_{G_\Q}(W,U_H^{(p)}) \to (\ob{\Q}_p)^{\oplus 4}$ be the $\ob{\Q}_p$-linear map given by 
		\[
(\kappa \colon W \to U_H^{(p)}) \mapsto 		(\lambda(\kappa(w_{\mathrm{b}})),\lambda(\kappa(w_{\rd})),\ord(\kappa(w_{\mathrm{b}})),\ord(\kappa(w_{\mathrm{d}}))),
		\]
		and $r$ denote its rank. Then $\dim\Sel(\ad \rho,V^+)=4-r$.
	\end{lemma}
	\begin{proof}
	By Lemma \ref{lem:first_lemma_selmer} (ii), $\dim\Sel(\ad \rho,V^+)=\dim\cS(V^+)$. One then checks using Poitou-Tate duality as in \eqref{eq:Poitou_Tate} that 
	\[	\cS(V^+) \simeq \coker\left[\HH^1_{\rf,p}(\Q,\widecheck{W}) \overset{\loc}{\longrightarrow} \HH^1(\Q_p,\widecheck{W}^\mathrm{b}) \times \HH^1(\Q_p,\widecheck{W}^\text{d}) \right]^*,\]
	where $\loc$ is a natural projection-and-localization map. We may identify $\loc$ with $\Upsilon$ by similar arguments as that in the proof of Lemma \ref{lem:dimension_sel_ad_0}, showing that $\dim\cS(V^+)=4-\mathrm{rk}(\Upsilon)$, as claimed.
	\end{proof}
	
	\subsection{Global units attached to $\ad^0\rho$}\label{sec:vanishing_tgt_spaces}
	It will be convenient to work with a basis $(e_1,e_2)$ of $V_{\ob{\Q}}$ satisfying the following condition:
	\begin{equation}\tag{dih}\label{tag:dihedral_basis_assumption}
		\text{If $f$ has real multiplication by $K$, then $(e_1,e_2)$ is a dihedral basis for $\rho$.}
	\end{equation} 
	In other words, if $f$ has real multiplication by $K$, then 
	\begin{equation}\label{eq:dihedral_decomposition}
		V_{\ob{\Q}}=e_1\ob{\Q}(\psi)\oplus e_2\ob{\Q}(\psi^\sigma)
	\end{equation}
	as $G_K$-module for some ($\ob{\Q}$-valued) character $\psi$ of $G_K$ of finite order with $\Gal(K/\Q)$-conjugate $\ob{\psi}$. Note that exchanging $e_1$ and $e_2$ amounts to switching the roles of $\psi$ and $\psi^\sigma$, so there always exists a basis satisfying \eqref{tag:dihedral_basis_assumption} that is adapted to a given line $V^+$ of $V$. 
	
	Fix a basis $(e_1,e_2)$ of $V_{\ob{\Q}}$ satisfying \eqref{tag:dihedral_basis_assumption}. It induces a $\ob{\Q}$-basis $\mathscr{C}:=\{w_1,w_2,w_3\}$ of $W_{\ob{\Q}}$, where
	\[
	w_1=\begin{pmatrix}
		0 & 1 \\ 0 & 0 
	\end{pmatrix}, \quad 
	w_2=\begin{pmatrix}
		1 & 0 \\ 0 & -1 
	\end{pmatrix}, \quad 
	w_3=\begin{pmatrix}
		0 & 0 \\ 1 & 0 
	\end{pmatrix}.
	\]
	In the RM case, letting $\psi_{\ad}=\psi/\psi^\sigma$, $W_{\ob{\Q}}$ is the sum $\ob{\Q}(\varepsilon_K) \oplus \Ind_K^\Q \ob{\Q}(\psi_{\ad})$ of two irreducible $G_\Q$-representations, and $\{w_2\}$ is a basis for the first summand, whereas $\{w_1,w_3\}$ is a dihedral basis for the second one.
	
	We make a convenient choice of a basis of $\Hom_{G_\Q}(W_{\ob{\Q}},U_{H,\ob{\Q}}^{(p)})$ using Lemma \ref{lem:Dirichlet} as follows.
	We let $\kappa$ be any generator of the $\ob{\Q}$-line $\Hom_{G_\Q}(W_{\ob{\Q}},U_{H,\ob{\Q}})$ and we complete $\{\kappa\}$ into a $\ob{\Q}$-basis 
	\begin{equation}\label{eq:definition_basis_B}
		\mathscr{B}:=\{\kappa,\kappa_1',\kappa_2',\kappa_3'\} 
	\end{equation}
	of $\Hom_{G_\Q}(W_{\ob{\Q}},U_{H,\ob{\Q}}^{(p)})$. In the RM case, we also require $\{\kappa,\kappa'_2\}$ be a basis of $\Hom_{G_\Q}(\ob{\Q}(\varepsilon_K),U_{H,\ob{\Q}}^{(p)})$, and $\{\kappa'_1,\kappa'_3\}$ a basis of subspace $\Hom_{G_\Q}(\Ind_K^\Q \ob{\Q}(\psi_{\ad}),U_{H,\ob{\Q}}^{(p)})$.
	By Lemma \ref{lem:Dirichlet} (iv), the matrix 
	\[O=(\ord(\kappa'_i(w_j)))_{1\leq i,j\leq 3}\] 
	is in $\GL_3(\ob{\Q})$; we modify $\mathscr{B}$ so that $O$ is the identity matrix. Note that this uniquely determines $\kappa_1',\kappa_2'$ and $\kappa_3'$ modulo the subspace $\Hom_{G_\Q}(W_{\ob{\Q}},U_{H,\ob{\Q}})=\ob{\Q}\cdot\kappa$ by the same lemma.
	
	For $1\leq i,j\leq3$, define
	\begin{equation}\label{eq:def_matrices_L_and_M}
	L_j:= \lambda(\kappa (w_j)), \quad  
	M_{ij}:= \lambda(\kappa_i'(w_j)).
	\end{equation}
	Hence, the entries of $L$ (resp. $M$) are $\ob{\Q}$-linear combinations of $p$-adic logarithms of units of $H$ (resp. $p$-units of $H$). 
	
	For later use, we also introduce other matrices $M^{(\ell)}$ defined as follows. If $\ell$ is a prime number at which $\rho$ is unramified, we say that $\ell$ is regular (for $\rho$) if $\rho(\Frob_\ell)$ is not scalar. For regular $\ell$, Lemma \ref{lem:Dirichlet} (i) implies that $\Hom_{G_\Q}(W_{\ob{\Q}},U_{H,\ob{\Q}}^{(\ell)})$ has dimension $2$ and we may choose a $\ob{\Q}$-basis of the form $\{\kappa,\kappa^{(\ell)}\}$. For $1\leq j \leq 3$, we then define
	\begin{equation}\label{eq:def_matrice_M_ell}
		M^{(\ell)}_j=\lambda(\kappa^{(\ell)}(w_j)).
	\end{equation}
	We record some useful properties of the matrices $L$, $M$ and $M^{(\ell)}$ in the next lemma.
	\begin{lemma}\label{lem:independence_entries_matrices_logs} Let $\cP\subset\ob{\Q}_p$ be the set of all entries of $L$, $M$ and $M^{(\ell)}$ with $\ell$ regular for $\rho$.
		\begin{enumerate}
			\item Assume $\rho$ is exotic. Then all the elements of $\cP$ are $\ob{\Q}$-linearly independent (hence nonzero). 
			\item Assume $\rho$ has RM by $K$. Then 
			\[L_1=L_3=M_{12}=M_{23}=M_{21}=M_{32}=0,\]
			and for $\ell$ regular, $M^{(\ell)}_1=M^{(\ell)}_3=0$ if $\ell$ splits in $K$, and $M^{(\ell)}_2=0$ if $\ell$ is inert in $K$.
			Moreover, the remaining entries of $L$, $M$ and $M^{(\ell)}$ (for varying $\ell$) are $\ob{\Q}$-linearly independent.
			\item In all cases, the nonzero elements of $\cP$ are $\ob{\Q}$-algebraically independent under Conjecture \ref{conj:schanuel}.
		\end{enumerate}
	\end{lemma}
	\begin{proof}
		Part (iii) clearly follows from (i) and (ii). If $\rho$ is exotic, then $\ad^0\rho$ is irreducible, so (i) follows from Lemma \ref{lem:Dirichlet} (ii) applied to $\pi=\ad^0 \rho$. 
		
		Assume now $\rho$ has RM by $K$. For regular $\ell$, we have $\kappa^{(\ell)}\in\Hom_{G_\Q}(\ob{\Q}(\varepsilon_K),U_H^{(\ell)})$ if $\ell$ splits in $K$, in which case $\kappa^{(\ell)}(w_1)=\kappa^{(\ell)}(w_3)=1$, and $\kappa^{(\ell)}\in\Hom_{G_\Q}(\Ind_K^\Q\ob{\Q}(\psi_{\ad}),U_H^{(\ell)})$ if $\ell$ is inert in $K$, in which case $\kappa^{(\ell)}(w_2)=1$. Similarly, $\kappa'_2(w_1)=\kappa'_2(w_3)=1$ and $\kappa'_1(w_2)=\kappa'_3(w_2)=1$. We obtain the vanishing of the corresponding entries of $L,M,M^{(\ell)}$ by applying $\lambda$ to these relations. The linear independence of the remaining entries follows from Lemma \ref{lem:Dirichlet} (iii) applied to $\pi^{(1)}=\varepsilon_K$ and $\pi^{(2)}=\Ind_K^\Q\psi_{\ad}$, which are indeed irreducible because we assumed $\rho$ does not have both RM and CM.
	\end{proof}
	Consider in the next proposition the two polynomials 
	\begin{equation}\label{eq:def_pol_P_L_and_P_M}
			\rP_L(S)=L_1-L_2S-L_3S^2, \qquad \rP_M(S)=\begin{vmatrix}
				M_{11}-M_{12}S-M_{13}S^2 & 2 & 0 \\
				M_{21}-M_{22}S-M_{23}S^2 & -S & 1 \\
				M_{31}-M_{32}S-M_{33}S^2 & 0  & 2S
			\end{vmatrix},
	\end{equation}
	and the condition 
	\begin{equation}\tag{res}\label{tag:non_vanishing_resultant_P_L_P_M}
		\res_S(\rP_L(S),\rP_M(S))\neq 0,
	\end{equation}
	where $\res_S$ stands for the resultant of two polynomials in the same variable $S$. It is not hard to see that, while the definitions of $L$ and $M$ depend on the choice of the bases $(e_1,e_2)$ and $\mathscr{C}$, the validity of \eqref{tag:non_vanishing_resultant_P_L_P_M} does not depend on them.
	\begin{prop}\label{prop:dimension_selmer_ad_0_rho}
			If \eqref{tag:non_vanishing_resultant_P_L_P_M} holds, then $\Sel(\ad^0\rho,V^+)=0$ for every line $V^+$ of $V$.
	\end{prop}
	
	\begin{proof}
		We may assume $(e_1,e_2)$ is adapted to $V^+$. Let $s$ be the slope of $V^+$ with respect to $(e_1,e_2)$.
		Since $w_{\mathrm{b}}=(\begin{smallmatrix}-s & 1 \\ -s^2 & s \end{smallmatrix})=w_1-sw_2-s^2w_3$ and  $w_{\mathrm{d}}=(\begin{smallmatrix}1 & 0\\ 2s & -1 \end{smallmatrix})=w_2+2sw_3$, Lemma \ref{lem:dimension_sel_ad_0} says that $\dim \Sel(\ad^0\rho,V^+)$ is $3-r^0$, where $r^0$ is the rank of the $4\times3$-matrix
		\small\begin{equation*}\label{eq:matrix_logs}
			\begin{pmatrix}
				L_1-L_2s-L_3s^2 & 0 & 0 \\ 
				M_{11}-M_{12}s-M_{13}s^2 & 1 & 0 \\
				M_{21}-M_{22}s-M_{23}s^2 & -s & 1 \\
				M_{31}-M_{32}s-M_{33}s^2 & -s^2  & 2s
			\end{pmatrix}
		\end{equation*}
		\normalsize
		One deduces that $r^0=3$ unless $\rP_L(s)=\rP_M(s)=0$, which is excluded by \eqref{tag:non_vanishing_resultant_P_L_P_M}. This proves the proposition.
	\end{proof}
	
	\begin{rem}\label{rem:p-adic_Stark_regulator}
		In light of \cite{maksoud_TAMS}, the expression $\rP_L(s)$ can be interpreted as a $p$-adic Stark regulator as follows. One may associate to $V^+\subset V$ the line $W^+\subset W$ defined as $\Hom(V^-,V^+)$, that is, $W^+=W^\rb$. Then $W^+$ is a $p$-stabilization of $W$ in the sense of \cite[Def. 3.1]{maksoud_TAMS} whose associated $p$-adic regulator (with respect to the bases $\omega_p^+=w_\rb$ and $\kappa$) is given by $\mathrm{Reg}_{\omega_p^+}(\ad\rho)=\log_p(1+p^\nu)\cdot \rP_L(s)$ (see (3.1) in \textit{loc. cit.}). This expression appears in the conjectural leading term formula at $s=0$ for the cyclotomic $p$-adic $L$-function attached to the pair $(\ad\rho,W^+)$ proposed in Conjecture 1.1.1 of \textit{loc. cit.}.
		
		We may also reformulate these facts in terms of Perrin-Riou's version of cyclotomic Iwasawa theory developed in \cite{PR95}. Indeed, by \cite[Lem. 3.23]{maksoud_TAMS}, one may naturally attach to $W^+$ a $2$-dimensional $\varphi$-submodule $D_{W^+}$ of $D_{\cris}(\widecheck{W})$, and $\rP_L(s)\neq0$ if and only if $D_{W^+}$ is regular in the sense of Perrin-Riou (see \cite[\S3.1.2, Definition and \S3.2.4]{PR95} for the definition of regularity).
	\end{rem}
	
	\subsection{Slopes and cocycles of the first deformations of $\rho$}\label{sectionslopes}
	We now determine the ordinary lines $V^+$ for which the Selmer group $\Sel(\ad \rho,V^+)$ is one-dimensional. To this end, fix as in \S\ref{sec:vanishing_tgt_spaces} a $\ob{\Q}$-basis $(e_1,e_2)$ of $V_{\ob{\Q}}$ satisfying \eqref{tag:dihedral_basis_assumption}. Let $V^+$ be a line of $V$ such that $(e_1,e_2)$ is adapted to $V^+$, and let $s\in\ob{\Q}_p$ be the slope of $V^+$ with respect to $(e_1,e_2)$. 
	
	We introduce the polynomial
	\small
	\begin{equation}\label{eq:def_polynomial_computing_residual_slopes}
		\rQ(S)=\begin{vmatrix}
		2L_1-SL_2 & L_2+2SL_3 & 0 & 0 \\
		2M_{11}-SM_{12} & M_{12}+2SM_{13} & 2 & 0 \\
		2M_{21}-SM_{22} & M_{22}+2SM_{23} & -S & 1 \\
		2M_{31}-SM_{32} & M_{32}+2SM_{33} & 0 & 2S
	\end{vmatrix}.
	\end{equation}
	\normalsize
	Note that $\rQ(S)$ is of degree at most $4$. By Lemma \ref{lem:independence_entries_matrices_logs} (ii), if $\rho$ has RM, then 
	\begin{equation}\label{RMslopepolynome} \rQ(S)= -2L_2\cdot \left( -M_{31}+(M_{11}-M_{33})S^2+M_{13}S^4\right).
	\end{equation}
	
	\begin{prop}\label{prop:dimension_selmer_ad_rho}
		Assume \eqref{tag:non_vanishing_resultant_P_L_P_M}. Then $\dim \Sel(\ad\rho,V^+)\leq 1$, with equality if and only if $s$ is a root of $\rQ(S)$. 
	\end{prop}	
	\begin{proof}
		Proposition \ref{prop:dimension_selmer_ad_0_rho} shows that $\Sel(\ad^0\rho,V^+)=\{0\}$, so $\dim \Sel(\ad\rho,V^+)\leq 1$ by \eqref{eq:exact_sequence_with_two_Selmer_groups}.
	As in the proof of Proposition \ref{prop:dimension_selmer_ad_0_rho}, we deduce from Lemma \ref{lem:dimension_sel_ad} that $\dim \Sel(\ad\rho,V^+)=4-r$, where $r$ equals
		\small \begin{equation}\label{eq:matrix_giving_dim_Sel}
			\textrm{rk}\begin{pmatrix}
			L_1-sL_2-s^2L_3 & L_2+2sL_3 & 0 & 0 \\
			M_{11}-sM_{12}-s^2M_{13} & M_{12}+2sM_{13} & 1 & 0 \\
			M_{21}-sM_{22}-s^2M_{23} & M_{22}+2sM_{23} & -s & 1 \\
			M_{31}-sM_{32}-s^2M_{33} & M_{32}+2sM_{33} & -s^2 & 2s
		\end{pmatrix}=\textrm{rk} 
		\begin{pmatrix}
			 2L_1-sL_2 &L_2+2sL_3 & 0 & 0 \\
			 2M_{11}-sM_{21} & M_{21}+2sM_{31} & 2 & 0\\
			 2M_{12}-sM_{22} & M_{22}+2sM_{32} & -s & 1 \\
			 2M_{13}-sM_{23} & M_{23}+2sM_{33} & 0 & 2s
		\end{pmatrix}
		\end{equation}\normalsize
	It is then clear that $\dim \Sel(\ad\rho,V^+)=1$ if and only if $\rQ(s)=0$, as claimed.
	\end{proof}

	\begin{lemma}\label{lem:definition_xi_1}
		Assume \eqref{tag:non_vanishing_resultant_P_L_P_M} holds, and suppose $\dim \Sel(\ad\rho,V^+)=1$. 
		Then there exists a unique cocycle class $[\xi_{V^+}]\in\Sel(\ad\rho,V^+)$ (depending on $V^+$) with trace $\lambda$, and the cocycle class $[\xi_{V^+}^0]:=[\xi_{V^+}]-\frac{1}{2}\lambda\cdot\Id_V$ generates $\cS(V^+)$.
	\end{lemma}
	\begin{proof}
		We know that $\Sel(\ad^0\rho,V^+)=\{0\}$ by Proposition \ref{prop:dimension_selmer_ad_0_rho}, so the existence and uniqueness of $[\xi_{V^+}]$ is clear from \eqref{eq:exact_sequence_with_two_Selmer_groups}. The fact that $[\xi_{V^+}^0]$ generates $\cS(V^+)$ follows from Lemma \ref{lem:first_lemma_selmer} (ii).
	\end{proof}

	\section{A Greenberg-Stevens $\cL$-invariant and other $p$-adic regulators}\label{HidaLinvariantandcslope}
	
	This section, which is an extension of \S\ref{sec:selmer_groups}, is devoted to the calculation of certain $p$-adic regulators which play a role in the proof of Theorem \ref{thm:main_result}. Keeping the notations of the previous section, we explicitly describe the restriction at $G_{\Q_p}$ of the Selmer class $[\xi_{V^+}]$ introduced in Lemma \ref{lem:definition_xi_1} and which plays a crucial role in the proof of Theorem \ref{thm:main_result}. When the line $V^+$ comes from a Hida family $\cF$, we will later interpret the quantity $t$ defined by \eqref{eq:xi_1_in_ordinary_basis} (essentially) as a Greenberg-Stevens $\cL$--invariant of $\cF$, that is, $a_p(f_\alpha)^{-1}\cdot\frac{\rd}{\rd k}\textbf{a}_p(\cF)_{|k=1}$, where $k$ is the weight variable parameterizing $\cF$ in the neighborhood of $f_\alpha$ (see Remark \ref{rem:GS_L_invariant}). Another canonical expression, which we prosaically call \textbf{c}-slope of $V^+$, will later appear in the proof of Theorem \ref{thm:main_result} through the hypothesis \eqref{tag:slope_2t+z_one_line}.
	
	Another goal of this section is to study and prove the hypotheses such as \eqref{tag:non_vanishing_resultant_P_L_P_M}, \eqref{tag:discriminant}, \eqref{tag:nonvanishing_resultant_Q_PL} which involve $p$-adic regulators, at least conditionally on conjectures in $p$-adic transcendental number theory.

		\subsection{Local description of a Selmer class}		
	We fix in this paragraph a $\ob{\Q}$-basis $(e_1,e_2)$ of a $\ob{\Q}$-structure $V_{\ob{\Q}}$ of $V$ satisfying \eqref{tag:dihedral_basis_assumption}. Let $V^+$ be a line of $V$ such that $(e_1,e_2)$ is adapted to $V^+$, and let $s\in\ob{\Q}_p$ be the slope of $V^+$ with respect to $(e_1,e_2)$. 
	\begin{prop}\label{prop:calculation_Up}
		Assume \eqref{tag:non_vanishing_resultant_P_L_P_M} holds and $\dim\Sel(\ad\rho,V^+)=1$. Let $[\xi_{V^+}]\in\Sel(\ad\rho,V^+)$ be as in Lemma \ref{lem:definition_xi_1} and let $x,y,t\in\ob{\Q}_p$ be such that 
		\begin{equation}\label{eq:xi_1_in_ordinary_basis}
			\begin{pmatrix}1 & 0 \\ -s &1 \end{pmatrix}\xi_{V^+|G_{\Q_p}}\begin{pmatrix}1 & 0 \\ s &1 \end{pmatrix}=\begin{pmatrix}\lambda-t\cdot\ord & x\cdot\lambda+y\cdot\ord \\ 0 &t\cdot \ord \end{pmatrix}
		\end{equation}
		in the basis $(e_1,e_2)$. If $\rP_L(s)\neq 0$, then $x=X(s)$ and $t=T(s)$, where
			\begin{equation}\label{eq:formulas_X_T}
				X(S)=-\frac{1}{2}\cdot\frac{L_2+2SL_3}{\rP_L(S)},\quad \mbox{and} \quad  T(S)=-\frac{1}{4}\cdot\frac{\begin{vmatrix}
						2L_1-SL_2  & L_{2}+2S L_{3} & 0 \\
						2M_{11}- S M_{12} & M_{12}+2 S M_{13} & 1 \\
						2M_{21}- S M_{22} & M_{22}+2 S M_{23} &-S
				\end{vmatrix}}{\rP_L(S)}. 
			\end{equation}
			In particular, if $\rho$ has RM, then $x=\frac{1}{2s}$ and $t=\frac{1}{2}(M_{11}+s^2M_{13})$.
	\end{prop}
	\begin{proof}
		Let $[\xi_{V^+}^0]$ be as in Lemma \ref{lem:definition_xi_1}, so that we have 
		\[\begin{pmatrix}1 & 0 \\ -s &1 \end{pmatrix}\xi^0_{V^+|G_{\Q_p}}\begin{pmatrix}1 & 0 \\ s &1 \end{pmatrix}=\begin{pmatrix}\frac{1}{2}\lambda-t\cdot\ord & x\cdot\lambda+y\cdot\ord \\ 0 &t\cdot \ord-\frac{1}{2}\lambda \end{pmatrix}.\]
		Using \eqref{eq:H1_W^e} and \eqref{eq:H1_widecheck_W^e}, we reinterpret $x$ and $t$ as 'slopes' of the space $\cS(V^+)$ introduced in \eqref{eq:definition_intermediate_selmer} as follows. First consider $t$, and let \[(\rho_{\cyc},\rho_\rf)\colon \cS(V^+) \to \HH^1_{\cyc}(\Q_p,W^\mathrm{d}) \times \HH^1_\rf(\Q_p,W^\mathrm{d})=\ob{\Q}_p \times\ob{\Q}_p\]
		be the maps induced from the localization map $\HH^1(\Q,W) \to \HH^1(\Q_p,W)$. Then  $\rho_{\cyc}([\xi_{V^+}^0])=-\frac{1}{2}$, and $\rho_\rf([\xi_{V^+}^0])=t$, so we have $-2t=\det(\rho_\rf \circ \rho_{\cyc}^{-1})$.
		
		By similar arguments as in the proof of Lemma \ref{lem:dimension_sel_ad_0}, Poitou-Tate duality identifies the diagram
		\begin{equation*}\label{eq:diagram_slope}
			\HH_{\cyc}^1(\Q_p,W^{\rd})=\ob{\Q}_p \overset{\rho_{\cyc}}{\longleftarrow} \cS(V^+) \overset{\rho_\rf}{\longrightarrow} \ob{\Q}_p=\HH_{\mathrm{f}}^1(\Q_p,W^{\rd})
		\end{equation*} 
		with the $\ob{\Q}_p$-linear dual of
		\[\HH_\rf^1(\Q_p,\widecheck{W}^{\rd})=\ob{\Q}_p\overset{j_\rf}{\longrightarrow} \coker\left[\HH^1_{\rf,p}(\Q,\widecheck{W}) \overset{\loc^\mathrm{bd}}{\longrightarrow} \HH^1(\Q_p,\widecheck{W}^{\mathrm{bd}}) \right] \overset{j_{\cyc}}{\longleftarrow} \ob{\Q}_p= \HH_{\cyc}^1(\Q_p,\widecheck{W}^{\rd}),\]
		where $j_\rf$ and $j_{\cyc}$ are induced by the projection $W^{\mathrm{bd}}\twoheadrightarrow W^{\rd}$, and $\loc^\mathrm{bd}$ is a localization map followed by the natural map $\HH^1(\Q_p,\widecheck{W}) \to \HH^1(\Q_p,\widecheck{W}^{\mathrm{bd}})$. In particular, $-2t=\det(j_\rf^{-1}\circ j_{\cyc})$, so $t$ is determined by the condition $j_{\cyc}(1)+j_\rf(2t)\in\mathrm{im}(\loc^\mathrm{bd})$.
		
		As $\cS(V^+)$ is a line, $\loc^\mathrm{bd}$ has rank $3$. We may form its $4\times4$ matrix using the basis $\mathscr{B}$ of $\HH^1_{\rf,p}(\Q,\widecheck{W})=\Hom_{G_\Q}(W,U_H^{(p)})$ (defined in \eqref{eq:definition_basis_B}) and the standard basis of \[\HH^1(\Q_p,\widecheck{W}^{\mathrm{bd}})=\HH^1_\rf(\Q_p,\widecheck{W}^{\mathrm{b}})\times \HH^1_\rf(\Q_p,\widecheck{W}^{\mathrm{d}})\times \HH^1_{\cyc}(\Q_p,\widecheck{W}^{\mathrm{b}})\times \HH^1_{\cyc}(\Q_p,\widecheck{W}^{\mathrm{d}})\simeq (\ob{\Q}_p)^{\oplus 4}. \]
       	Then $t$ is such that the augmented matrix
    	\[\begin{pmatrix}
    		L_1-sL_2-s^2L_3 & L_2+2sL_3 & 0 & 0  \\ 
    		M_{11}-sM_{12}-s^2M_{13} & M_{12}+2sM_{13} & 1 & 0 \\
    		M_{21}-sM_{22}-s^2M_{23} & M_{22}+2sM_{23} & -s & 1 \\
    		M_{31}-sM_{32}-s^2M_{33} & M_{32}+2sM_{33} & -s^2  & 2s \\
    		0 & 2t & 0 & 1
    	\end{pmatrix}
    	\]
    	has rank $3$. Removing the second-to-last row and using the assumption that $\rP_L(s)=L_1-sL_2-s^2L_3\neq0$, a simple determinant computation yields the expression \eqref{eq:formulas_X_T} for $t$. The argument to compute $x$ is similar and one has to replace the last row of the above matrix by $(1,-2x,0,0)$.
    	
    	If $\rho$ has RM, then Lemma \ref{lem:equivalence_resultant_Q_PL} below shows that $\rP_L(s)\neq0$, and a simplification of the expressions in \eqref{eq:formulas_X_T} yields the computation of $x$ and $t$ in this case.
	\end{proof}
	\begin{rem}\label{rem:independence_t_from_basis}
		Let $t$ be defined as in \eqref{eq:xi_1_in_ordinary_basis}. One sees from the local conditions satisfied by  $\xi_{V^+}$ given in Definition \ref{def:Selmer} (ii) that $\xi_{V^+}(\Frob_p)$ induces a linear endomorphism of $V/V^+$ whose determinant is $t$. In particular, the value of $t$ does not depend on any choice of bases.
	\end{rem}

		Assume \eqref{tag:non_vanishing_resultant_P_L_P_M} holds, and suppose $\dim\Sel(\ad\rho,V^+)=1$. As $\HH^1(\Q,W)$ is of dimension $2$ (see e.g. \cite[Lem. 3.2]{D-B}), Lemma \ref{lem:first_lemma_selmer} (ii) and \eqref{eq:reformulation_S(V+)} show that the image of the localization map 
		\begin{equation}\label{eq:def_loc_c}
			\loc^\rc\colon \HH^1(\Q,W) \to \HH^1(\Q_p,W^\rc)
		\end{equation}
		is a line.

\begin{defn} \label{def:c_slope}
	The $\mathbf{c}$-slope of $V^+$ is the slope of the image of the map in \eqref{eq:def_loc_c} with respect to the basis of $\HH^1(\Q_p,W^\rc)=\HH^1_{\cyc}(\Q_p,W^\rc)\times\HH_\rf^1(\Q_p,W^\rc)\simeq \ob{\Q}_p\times\ob{\Q}_p$ provided by \eqref{eq:H1_W^e}.
	\end{defn}
	In other words, given any cocycle class $[\xi]\in\HH^1(\Q,W)$ not in $\cS(V^+)$, the $\textbf{c}$-slope of $V^+$ is the ratio $y_\rc/x_\rc\in\mathbb{P}^1(\ob{\Q}_p)$, where $x_\rc,y_\rc\in\ob{\Q}_p$ are such that the $(2,1)$-entry of 
	$(\begin{smallmatrix}1 & 0 \\ -s &1 \end{smallmatrix})\xi_{|G_{\Q_p}}(\begin{smallmatrix}1 & 0 \\ s &1 \end{smallmatrix})$ is $x_\rc\cdot\lambda+y_\rc\cdot\ord$.
	\begin{prop}\label{prop:calcul_c_slope}
		 Assume \eqref{tag:non_vanishing_resultant_P_L_P_M} holds, and suppose $\dim\Sel(\ad\rho,V^+)=1$. If $\rP_L(s)=0$, then the $\textbf{c}$-slope of $V^+$ is $\infty$. Otherwise, it is given by $Z(s)$, where
		\small
\begin{equation}\label{eq:calcul_c_slope}
			 Z(S)= -\frac{1}{2}\frac{ \begin{vmatrix}
			 		L_1-SL_2 & L_3 & 0 & 0 \\ 
			 		M_{11}-SM_{12} & M_{13} & 2 & 0 \\
			 		M_{21}-SM_{22} & M_{23} & -S & 1 \\
			 		M_{31}-SM_{32} & M_{33} & 0 & 2S
			 \end{vmatrix} }{\rP_L(S)}.
\end{equation} \normalsize
		If $\rho$ has RM, then $Z(s)=s^2M_{13}-M_{33}$.
	\end{prop}	
	\begin{proof}
				Let $z\in\textbf{P}^1(\ob{\Q}_p)$ be the $\textbf{c}$-slope of $V^+$, and consider the space
		 \[\cS_1(V^+):=\ker(\HH^1(\Q,W) \to \HH_{\cyc}^1(\Q_p,W^{\mathrm{d}})).\]
		 Since $\cS(V^+)\cap \cS_1(V^+)=\Sel(\ad^0\rho,V^+)=\{0\}$ by Proposition \ref{prop:dimension_selmer_ad_0_rho}, $\cS_1(V^+)$ is a line. Let 
		 \[(\psi_{\cyc},\psi_\rf)\colon \cS_1(V^+) \to \HH^1_{\cyc}(\Q_p,W^\rc)\times\HH_\rf^1(\Q_p,W^\rc)\]
		  be the maps induced by \eqref{eq:def_loc_c}. Arguing as in the proof of Proposition \ref{prop:calculation_Up}, we may use Poitou-Tate duality to identify the diagram
		\begin{equation*}\label{eq:diagram_cslope}
			\HH_{\cyc}^1(\Q_p,W^{\rc})\overset{\psi_{\cyc}}{\longleftarrow} \cS_1(V^+) \overset{\psi_\rf}{\longrightarrow} \HH_{\mathrm{f}}^1(\Q_p,W^{\rc})
		\end{equation*}
		with the $\ob{\Q}_p$-linear dual of
		\small\begin{equation*}\label{dualdetz}
		\HH_\rf^1(\Q_p,\widecheck{W}^{\rc})=\ob{\Q}_p \overset{\varphi_\rf}{\longrightarrow} \coker\left[\HH^1_{\rf,p}(\Q,\widecheck{W}) \overset{\loc^1}{\longrightarrow} \HH^1(\Q_p,\widecheck{W}^{\mathrm{bc}}) \oplus \HH^1_{\cyc}(\Q_p,\widecheck{W}^{\mathrm{d}}) \right] \overset{\varphi_{\cyc}}{\longleftarrow} \ob{\Q}_p=\HH_{\cyc}^1(\Q_p,\widecheck{W}^{\rc}),
		\end{equation*}\normalsize
 		where $\varphi_{\cyc},\varphi_\rf$ are the natural projection maps and $\loc^1$ is a localization map. As $\cS_1(V^+)$ is one-dimensional, $\loc^1$ has full rank $4$. Moreover, writing $z$ as $y_\rc/x_\rc$, we have $\varphi_\rf(y_\rc)-\varphi_{\cyc}(x_\rc)\in\mathrm{im}(\loc^1)$. 
 		Hence, working with the basis $\mathscr{B}$ of $\HH^1_{\rf,p}(\Q,\widecheck{W})$ in \eqref{eq:definition_basis_B} and the standard basis of \[\HH^1_\rf(\Q_p,\widecheck{W}^{\mathrm{b}})\times \HH^1_\rf(\Q_p,\widecheck{W}^{\mathrm{c}})\times \HH^1_{\cyc}(\Q_p,\widecheck{W}^{\mathrm{b}})\times \HH^1_{\cyc}(\Q_p,\widecheck{W}^{\mathrm{c}})\times \HH^1_{\cyc}(\Q_p,\widecheck{W}^{\mathrm{d}})\simeq (\ob{\Q}_p)^{\oplus 5}, \]
 		one sees that $z$ is determined by the condition that

		\begin{equation}\label{c-slope}
		\begin{vmatrix}
			L_1-sL_2-s^2L_3 & L_3& 0 & 0 & 0  \\ 
			M_{11}-sM_{12}-s^2M_{13} & M_{13}& 1 & 0 & 0 \\
			M_{21}-sM_{22}-s^2M_{23} & M_{23}& -s & 0 & 1 \\
			M_{31}-sM_{32}-s^2M_{33} & M_{33}& -s^2  & 1 & 2s \\
			0 & y_\rc & 0 & -x_\rc & 0 
			\end{vmatrix}=0.
		\end{equation}
		This implies $x_\rc=0$ if $\rP_L(s)=L_1-sL_2-s^2L_3=0$, and formula \eqref{eq:calcul_c_slope} otherwise.
	\end{proof}

	
	\medskip
		
	\subsection{Non-vanishing of regulators}\label{sec:regulators}
	The following lemma will be useful to apply the results of the preceding paragraphs to a collection of lines with respect to a single basis $(e_1,e_2)$ of $V_{\ob{\Q}}$.
	

	\begin{lemma}\label{lem:existence_basis_adapted_to_ordinary_lines}
	Let $\cV$ be a finite set of lines $V^+$ of $V$ satisfying $\dim_{\ob{\Q}_p}\Sel(\ad\rho,V^+)=1$. Then there exists a basis $(e_1,e_2)$ of $V_{\ob{\Q}}$ satisfying \eqref{tag:dihedral_basis_assumption} and that is adapted (in the sense of Definition \ref{def:adapted}) to every line in $\cV$.
\end{lemma}
\begin{proof}
	If $\rho$ is exotic, then \eqref{tag:dihedral_basis_assumption} is empty, so the claim is clearly true.
	If $\rho$ has RM, we claim that any basis of a $\ob{\Q}$-structure $V_{\ob{\Q}}$ of $V$ satisfying \eqref{tag:dihedral_basis_assumption} is automatically adapted to any $V^+$ such that $\dim_{\ob{\Q}_p}\Sel(\ad\rho,V^+)=1$. To check this, it is enough to fix one such $V^+$ and show that it does not descend to $V_{\ob{\Q}}$. To this end, fix a basis $(e_1,e_2)$ of $V_{\ob{\Q}}$ adapted to $V^+$ and let $s\in\ob{\Q}_p$ be such that $e_1+s\cdot e_2$ generates $V^+$. Then, by Proposition \ref{prop:dimension_selmer_ad_rho}, $-M_{31}+(M_{11}-M_{33})s^2+M_{13}s^4=0$. But Lemma \ref{lem:independence_entries_matrices_logs} (ii) then implies that $s$ must be transcendental, hence proving the claim.
\end{proof}

\begin{lemma}\label{lem:nonvanishing_resultant_PL_PM}
	Condition \eqref{tag:non_vanishing_resultant_P_L_P_M} holds in the RM case, and it also holds in the exotic case under Conjecture \ref{conj:schanuel}.
\end{lemma}
\begin{proof}
	If $\rho$ has RM, one sees from Lemma \ref{lem:independence_entries_matrices_logs} (ii) that $\rP_L(s)=0$ implies $s=0$, and $\rP_M(0)=0$ implies $L_2\cdot M_{13}=0$ which is false by Lemma \ref{lem:independence_entries_matrices_logs} (ii). This proves that \eqref{tag:non_vanishing_resultant_P_L_P_M} holds true in this case.
	If $\rho$ is exotic, then Lemma \ref{lem:independence_entries_matrices_logs} (iii) shows that all the coefficients of $\rP_L(S)$ and $\rP_M(S)$ are algebraically independent under Conjecture \ref{conj:schanuel}. In particular, $\res_S(\rP_L(S),\rP_M(S))$, which is a nontrivial polynomial expression in their coefficients, cannot vanish.
\end{proof}

	Assume $(e_1,e_2)$ is adapted to $V^+$ and let $s\in\ob{\Q}_p$ be the slope of $V^+$ with respect to $(e_1,e_2)$. Assume \eqref{tag:non_vanishing_resultant_P_L_P_M} holds and $\dim\Sel(\ad\rho,V^+)=1$, and consider the assumption
	\begin{equation}\tag{$S_{V^+}$}\label{tag:slope_2t+z_one_line}
		\text{the \textbf{c}-slope of $V^+$ is different from $-2t$,}
	\end{equation}
	where $t$ and the \textbf{c}-slope are defined in \eqref{eq:xi_1_in_ordinary_basis} and Definition \ref{def:c_slope} respectively. Note that \eqref{tag:slope_2t+z_one_line} does not depend on the choice of bases by Remark \ref{rem:independence_t_from_basis}. With the  notation $\rQ(S)$ in \eqref{eq:def_polynomial_computing_residual_slopes}, we define
	\begin{equation}\tag{sl}\label{tag:discriminant}
		\mathrm{Disc}(\rQ(S))\neq0.
	\end{equation}
	We implicitly assume in \eqref{tag:discriminant} that $\deg \rQ(S)\geq2$, so that its discriminant is well-defined.
	
	\begin{lemma}
	Assume \eqref{tag:non_vanishing_resultant_P_L_P_M} holds. 
	\begin{enumerate}
		\item If \eqref{tag:discriminant} holds, then \eqref{tag:slope_2t+z_one_line} holds for every line $V^+$ of $V$ with $\dim\Sel(\ad\rho,V^+)=1$.
		\item If $\rho$ is exotic, then \eqref{tag:discriminant} is implied by Conjecture \ref{conj:schanuel}.
		\item If $\rho$ has RM, then \eqref{tag:discriminant} is equivalent to
		\begin{equation}\label{eq:RM_relation_implying_S}
			4\cdot M_{13}\cdot M_{31}+(M_{11}-M_{33})^2\neq0,
		\end{equation}
		which is a consequence of Conjecture \ref{conj:4exp}.
	\end{enumerate}
	\end{lemma}
	
	\begin{proof}
		Parts (ii) and (iii) are easy consequences of Lemma \ref{lem:independence_entries_matrices_logs}, so we only prove (i).
		
		Let $V^+$ be a line and fix a basis $(e_1,e_2)$ of $V_{\ob{\Q}}$ satisfying \eqref{tag:dihedral_basis_assumption} and that is adapted to $V^+$. 
		Recall that the slope $s$ of $V^+$ then satisfies the relation $\rQ(s)=0$ by Proposition \ref{prop:dimension_selmer_ad_rho}.
		
		 If $\rP_L(s)=0$, then the \textbf{c}-slope of $V^+$ is $\infty$ by Proposition \ref{prop:calcul_c_slope} and \eqref{tag:slope_2t+z_one_line} trivially holds, so we may assume $\rP_L(s)\neq 0$. By Propositions \ref{prop:calculation_Up} and \ref{prop:calcul_c_slope}, if \eqref{tag:slope_2t+z_one_line} does not hold, then $Z(s)+2T(s)= 0$, \textit{i.e.}, $\mathcal{K}(s)= 0$, where
		\[ \mathcal{K}(S):= \begin{vmatrix}
				L_1-SL_2 & L_3 & 0 & 0 \\ 
				M_{11}-SM_{12} & M_{13} & 2 & 0 \\
				M_{21}-SM_{22} & M_{23} & -S & 1 \\
				M_{31}-SM_{32} & M_{33} & 0 & 2S
			\end{vmatrix}+\begin{vmatrix}
			2L_1-SL_2  & L_{2}+2S L_{3} & 0 \\
			2M_{11}- S M_{12} & M_{12}+2 S M_{13} & 1 \\
			2M_{21}- S M_{22} & M_{22}+2 S M_{23} &-S
			\end{vmatrix}. \]
		An elementary calculation shows that $4\cdot \mathcal{K}(S)$ is the derivative of the polynomial $\rQ(S)$. Therefore, $\mathcal{K}(s)=\rQ(s)=0$ contradicts \eqref{tag:discriminant}, so \eqref{tag:slope_2t+z_one_line} must hold.
	\end{proof}

	Consider the condition
	\begin{equation}\tag{reg}\label{tag:nonvanishing_resultant_Q_PL}
		\res_S(\rQ(S),\rP_L(S))\neq0,
	\end{equation}
	where $\rP_L(S)$ and $\rQ(S)$ are defined in \eqref{eq:def_pol_P_L_and_P_M} and \eqref{eq:def_polynomial_computing_residual_slopes} respectively.
	\begin{lemma}\label{lem:equivalence_resultant_Q_PL}
		\begin{enumerate}
			\item If $\rho$ has RM, then \eqref{tag:nonvanishing_resultant_Q_PL} holds.
			\item If $\rho$ is exotic, then \eqref{tag:nonvanishing_resultant_Q_PL} holds if and only if \eqref{tag:non_vanishing_resultant_P_L_P_M} holds and $\mathrm{Disc}(\rP_L(S))\neq0$, where 
			\begin{equation}\label{eq:discriminant_PL}
				\mathrm{Disc}(\rP_L(S))=L_2^2+4\cdot L_1\cdot L_3
			\end{equation}
			is the discriminant of $\rP_L(S)$. Moreover, \eqref{tag:nonvanishing_resultant_Q_PL} holds under Conjecture \ref{conj:schanuel}.
		\end{enumerate}
	\end{lemma}
	\begin{proof}
		If $\rho$ has RM, then $\rP_L(s)=0$ implies $s=0$ by Lemma \ref{lem:independence_entries_matrices_logs} (ii), and the same lemma shows that $\rQ(0)=2\cdot L_2\cdot M_{31}\neq0$. Therefore, $\rP_L(S)$ and $\rQ(S)$ do not have a common root, and (i) follows. 
		
		Assume $\rho$ is exotic and let us prove the equivalence of the two conditions. Adding the second column to the first one in \eqref{eq:def_polynomial_computing_residual_slopes} and expanding along the first row shows that $\rQ(S)$ can be written in the form $\rP_L(S)\cdot R(S)-4\cdot \rP_L'(S)\cdot \rP_M(S)$ for some polynomial $R(S)$. One deduces that $\res_S(\rQ(S),\rP_L(S))=-16\cdot \mathrm{Disc}(\rP_L(S))\cdot\res_S(\rP_M(S),\rP_L(S))$, and the equivalence to prove then becomes clear.
		
		If Conjecture \ref{conj:schanuel} holds, then \eqref{tag:non_vanishing_resultant_P_L_P_M} holds by Lemma \ref{lem:nonvanishing_resultant_PL_PM} and $L_2^2+4\cdot L_1\cdot L_3\neq 0$ by Lemma \ref{lem:independence_entries_matrices_logs} (iii), so \eqref{tag:nonvanishing_resultant_Q_PL} holds as well.
	\end{proof}
	
	\begin{rem}
		Assume \eqref{tag:nonvanishing_resultant_Q_PL} holds. It is shown in Proposition \ref{prop:first_deformations_rho_A} (ii) below that the slope $s$ of any line $V^+$ of automorphic origin, \textit{i.e.}, obtained as the specialization of the ordinary line associated to a Hida family passing through $f_\alpha$, must be a root of $\rQ(S)$. As \eqref{tag:nonvanishing_resultant_Q_PL} implies that $\rP_L(s)\neq0$ for such $s$, the sub-$\varphi$-module of $D_{\cris}(W)$ attached by Remark \ref{rem:p-adic_Stark_regulator} to any such $V^+$ is a regular submodule in the sense of Perrin-Riou.
	\end{rem}

\subsection{Regulators in dihedral cases}
We give a more explicit description of the relevant regulators in case $\rho$ when $\rho$ has dihedral image.

	\subsubsection{Gross--Stark units in the RM case}
	\label{sec:reformulation_consition_S_RM_case}
		We reformulate in this paragraph the nonvanishing condition \eqref{eq:RM_relation_implying_S} in terms of Gross-Stark units when $\rho$ has RM by $K$. As $p$ splits in $K$, we have $p\cO_K=\Gp\ob{\Gp}$, where $\Gp$ is the prime of $K$ induced by $\iota_p$ and $\ob{\Gp}\neq\Gp$. Let $\psi_{\ad}=\psi/\psi^\sigma$ be the totally odd character of $G_K$ obtained from the decomposition \eqref{eq:dihedral_decomposition}, and let $u_{\Gp}=\kappa_1(w_1)$ and $u_{\ob{\Gp}}=\kappa_3(w_1)$. Then $u_{\Gp}$ and $u_{\ob{\Gp}}$ form a basis of the $\psi_{\ad}$-isotypical subspace $U_{H}^{(p)}[\psi_{\ad}]$ of $U_H^{(p)}$. The nontrivial element $\sigma$ of $\Gal(K/\Q)$ exchanges $w_1$ and $w_3$, so we have $\kappa_1(w_3)=u^\sigma_{\Gp}$, $\kappa_3(w_3)=u^\sigma_{\ob{\Gp}}$. Moreover, $u_{\Gp}$ (resp. $u_{\ob{\Gp}}$) is characterized as the unique $\Gp$-unit (resp. $\ob{\Gp}$-unit) in $U_{H}^{(p)}[\psi_{\ad}]$ with $\Gp$-adic valuation $1$ (resp. with $\ob{\Gp}$-adic valuation $1$). Condition \eqref{eq:RM_relation_implying_S} then reads
		\[4\cdot\log_p(u_{\ob{\Gp}})\cdot\log_p(u^\sigma_{\Gp})+\log_p\left(\frac{u_\Gp}{u^\sigma_{\ob{\Gp}}}\right)^2\neq0.\]
		With the above notations, Gross's regulator $\mathscr{R}_p(\psi_{\ad}^{-1})$ for $\psi_{\ad}^{-1}$ can be expressed as 
		\[\mathscr{R}_p(\psi_{\ad}^{-1})=\log_p(u_\Gp)\cdot \log_p(u^\sigma_{\ob{\Gp}})-\log_p(u_{\ob{\Gp}})\cdot\log_p(u^\sigma_{\Gp})\] 
		(compare with \cite[(8)]{DKV}).
	
	\subsubsection{The CM case}\label{sec:CM_case} Although this paper primarily focuses on the non-CM case, it is instructive to compare the conditions \eqref{tag:discriminant} and \eqref{tag:nonvanishing_resultant_Q_PL} in the CM case with the conditions on anti-cyclotomic $\mathscr{L}$-invariants needed in \cite[Thm. B]{betinadimitrovKatz}. We will see that all these conditions are, in fact, equivalent, even though our proof of Theorem \ref{thm:main_result} is different from that in \textit{loc. cit.}.
	
	We assume now that $\rho\simeq\Ind_K^\Q\chi$ has CM by an imaginary quadratic field $K$ and we denote by $\tau\in G_\Q$ the complex conjugation. We may fix a dihedral basis $(e_1,e_2)$ of $V_{\ob\Q}$, in which $\rho(\tau)=(\begin{smallmatrix}0 & 1 \\ 1 & 0\end{smallmatrix})$, and $\rho_{|G_K}=\chi\oplus \chi^\tau$. Then $\ad^0(\rho)\simeq \varepsilon_K \oplus \Ind_K^\Q \varphi$, where $\varepsilon_K$ be the odd quadratic character of $G_\Q$ associated with $K$ and $\varphi=\chi/\chi^\tau$ is a nontrivial character of $G_K$ which is anticyclotomic, \textit{i.e.}, $\varphi^\tau=\ob{\varphi}$. Arguing as in Lemma \ref{lem:independence_entries_matrices_logs} (ii), we then have $L_2=M_{12}=M_{21}=M_{23}=M_{32}=0$, and the remaining entries of $L$ and $M$ are $\ob{\Q}$-linearly independent if $\rho$ does not have RM. In the Klein case where $\rho$ also has RM, we have the additional relations $\ob{\varphi}=\varphi$, $L_1=L_3$, $M_{11}=M_{33}$ and $M_{13}=M_{31}$, and the elements of the set $\{L_1,M_{11},M_{13}\}$ are $\ob{\Q}$-linearly independent. 
	
	The slope $\mathscr{S}_{\varphi}$, the $\mathscr{L}$-invariant $\mathscr{L}_{\Gp}$ attached to $K$ and the two anticyclotomic $\mathscr{L}$-invariants $\mathscr{L}_{-}(\varphi),\mathscr{L}_{-}(\ob{\varphi})$ introduced in \cite{betinadimitrovJTNB,betinadimitrovKatz} are given by 
	\[ \mathscr{S}_{\varphi}=-\dfrac{L_1}{L_3}, \qquad \mathscr{L}_{\Gp}= -\dfrac{1}{2}\cdot M_{22}, \]
	and
	\[\cL_{-}(\varphi)=-\frac{L_1M_{33}-L_3M_{31}}{L_1}+M_{22}, \qquad \cL_{-}(\ob{\varphi})=-\frac{L_3M_{11}-L_1M_{13}}{L_3}+M_{22}.\]
	We have $\mathscr{S}_{\varphi}\neq 0$ and
	\[\rQ(S)=-8\cdot L_3\cdot \left(\cL_{-}(\ob{\varphi}) \cdot S^3- \cL_{-}(\varphi)\cdot \mathscr{S}_{\varphi}\cdot S\right), \]
	and \eqref{tag:discriminant} is then equivalent to $\cL_{-}(\varphi)\cdot \cL_{-}(\ob{\varphi})\neq0$. Furthermore, $\rP_L(S)=-L_3\cdot(\mathscr{S}_{\varphi}+S^2)$, and \eqref{tag:nonvanishing_resultant_Q_PL} is equivalent to $\cL_{-}(\varphi)+\cL_{-}(\ob{\varphi})\neq0$. In the special case where $\rho$ has RM and CM, both conditions \eqref{tag:discriminant} and \eqref{tag:nonvanishing_resultant_Q_PL} unconditionally hold.
	
\section{Congruences between Hida families at weight one points}\label{sec:geometry}	

We shall now begin the proof of Theorem \ref{thm:main_result}. We retain the notation of the introduction: let $\cC$ denote the Coleman--Mazur eigencurve, and let $\mathrm {x}\in \cC$ denote the point corresponding to the $p$-stabilized cusp form $f_{\alpha}$.
The proof will be divided into two main parts:
\begin{enumerate}[label=\Roman*.]
\item
Showing that there are at most four components of $\cC$ passing through the point $x$, and that each is \'etale over the weight space;
\label{part1} 
\item
Showing that there are exactly four components of $\cC$ meeting transversally at $x$.
 \label{part2} 
\end{enumerate}
 The first of these is the subject of the present section. Part \ref{part2} will be carried out in \S \ref{sec:geometry2}.
 
 \subsection{Overview of the proof of Theorem \ref{thm:main_result}, Part \ref{part1}}\label{overview1}
In this section, we prove that  if the assumption \eqref{tag:discriminant} and the additonal hypothesis \eqref{tag:non_vanishing_resultant_P_L_P_M} (weaker than \eqref{tag:nonvanishing_resultant_Q_PL}) both hold, then the weight map $\kappa$ is \'etale at $\mathrm{x}$ along any irreducible component  (Theorem \ref{thm:etaleness}) and that any irreducible local component of $\cC$ at $\mathrm x$, say $\mathrm{Sp}(\cA_0)$,  is entirely determined by the specialization $V^+_{\cA_0}\in\textbf{P}(V)$ at $\mathrm{x}$ of its ordinary filtration (Theorem \ref{thm:uniqueness_residual_slope}). 

Our argument, which has a certain resemblance to Ribet's method on the Eisenstein ideal, is based on a fine analysis of higher infinitesimal deformations of $\rho$ along $\Spec(\cA_0)$. 
The idea is that, if an irreducible component is not \'etale at $\mathrm{x}$, it gives rise to a non-zero class $[\xi]\in\HH^1(\Q,\ad^0\rho)$ with a prescribed local behavior at $p$.  This leads to a contradiction because:
\begin{itemize}
\item If $[\xi]$ is a Selmer class in $\Sel(\ad^0,V^+_{\cA_0})$, this  is inconsistent with \eqref{tag:non_vanishing_resultant_P_L_P_M} by Proposition \ref{prop:dimension_selmer_ad_0_rho};
\item If $[\xi]$ does not lie in $\Sel(\ad^0,V^+_{\cA_0})$, the formula for the \textbf{c}-slope of $V^+_{\cA_0}$ contradicts \eqref{tag:discriminant}. 
\end{itemize}
A pivotal point of our method is the fact, established in the proof of Theorem \ref{thm:etaleness}, that the residual slope $s\in\ob{\Q}_p$ keeping track of the position of the line $V^+_{\cA_0}\in\textbf{P}(V)$ with respect to a suitable basis of $V$ is a root of the polynomial $\rQ(S)$ defined in \S\ref{sectionslopes}. Since the degree of $\rQ(S)$ is at most $4$, it follows that there exist at most four Hida families passing through $f_\alpha$.

	\subsection{\'Etaleness of the irreducible components of $\cC$ at $x$}\label{sec:etaleness}
	By construction of $\cC$, there exist bounded  global sections $\{ \mathbf{T}_{\ell},\mathbf{U}_q\}_{\ell \nmid Np,q \mid Np} \subset \mathcal{O}_{\cC}(\cC)$  such that the usual application ``system of eigenvalues'' 
	\[ \cC(\bar\Q_{p}) \ni x \mapsto \{\mathbf{T}_{\ell}(x),\mathbf{U}_q(x)\}_{\ell \nmid Np, q\mid Np}\] 
	is injective, and produces all systems of eigenvalues for $\{ \mathbf{T}_{\ell},\mathbf{U}_q\}_{\ell \nmid Np, q \mid Np}$ acting on the space of overconvergent cusp forms  with  coefficients in $\bar \Q_p$, tame level $N$, whose weight is in $\mathcal{W}(\bar\Q_p)$ and having a non-zero $\mathbf{U}_p$-eigenvalue.  Here we work with the full eigencurve rather than the reduced eigencurve generated locally by $\{ \mathbf{T}_{\ell},\mathbf{U}_p\}_{\ell \nmid Np}$.
	
	A fundamental arithmetic tool in the study of the geometry of $\cC$ is the  universal pseudo-character $\mathrm{Ps}_{\cC}:G_{\Q} \rightarrow  \mathcal{O}_{\cC}(\cC)$ of dimension $2$ carried by $\cC$ which is unramified at all $\ell\nmid Np$ and 
	such that $\mathrm{Ps}_{\cC}(\Frob_{\ell})=\mathbf{T}_{\ell}$. This pseudo-character $p$-adically interpolates the traces of semi-simple $p$-adic Galois representations attached to the classical points of $\cC$. 
		
	We recall that $\varLambda:=\bar{\Q}_p \lsem X \rsem$ is the completed local ring of $\cW$ at $\kappa(\mathrm{x})$ and let $\cT$ be the completed local ring of $\cC$ at $\mathrm{x}$, where $\mathrm{x}$ corresponds to the (unique) $p$--stabilisation $f_\alpha$ of our irregular weight one form $f$. Let $\rho\colon G_\Q \to \GL_2(\ob{\Q}_p)$ be the Galois representation attached to $f$ by Deligne and Serre and let $\mathrm{Ps}_{\cT}:G_\Q \to \cO(\cC) \to \cT$ be the $2$-dimensional  $\cT$-valued pseudo-character given by the localization of $\mathrm{Ps}_{\cC}$ via the localization morphism $\cO(\cC) \to \cT$. Since $\rho$ is irreducible and odd, there exists a deformation \[\rho_{\cT}:G_\Q \to \GL_2(\cT) \] of $\rho$ such that $\mathrm{tr}(\rho_\cT)=\mathrm{Ps}_{\cT}$ (\cite{nyssen1996pseudo,rouquier1996caracterisation}).

	\begin{rem}It follows from \cite[Prop.1.4]{betinadimitrovJTNB} that $\cT$ is reduced because $f$ is $N$-new.
	\end{rem}

	We fix an arbitrary minimal prime ideal $\mathfrak{q}$ of $\cT$, and set $\cA_0:=\cT/\mathfrak{q}$. We denote by $\mathbf{K}$ the quotient field of $\cA_0$. As $\cT$ is reduced, $\textbf{K}$ is the localization $\cT_{\mathfrak{q}}$ of $\cT$ at $\mathfrak{q}$. Recall that the minimal prime ideal $\mathfrak{q}$ gives rise to a $\mathrm{Aut}_\varLambda(\textbf{K})$-conjugacy class $[\cF]$ of $p$-ordinary cuspidal Hida families of tame level $N$ passing through $f_\alpha$. 
	Indeed, if $\mathfrak{h}'$ is the $p$-ordinary Hida--Hecke algebra of level $Np^\infty$, then the rigid analytic fiber of the adic space $\mathrm{Spa}(\mathfrak{h}',\mathfrak{h}')$ is exactly the ordinary locus $\cC^{\ord}$ of $\cC$, {\it i.e.,} the open admissible locus where the slope of $\mathbf{U}_p$ is zero.
	
	Let $\cA$ be the normalization of $\cA_0$ in $\mathbf{K}$. It is a discrete valuation ring since it is a $1$-dimensional normal noetherian local domain. Let $e$ be the ramification index of $\cA$ over $\varLambda$. Then, one has $\cA=\bar\Q_p \lsem Y\rsem$, $\mathbf{K} = \overline{\Q}_p (\!(Y)\!)$ and $Y^e=X$. In general, the completion of an integral domain may not be a domain but since the $\kappa^{\#}:\cO_{\cW,\kappa(\mathrm{x})} \to \cO_{\cC,\mathrm{x}}$ is finite, the natural map $\mathrm{Spec} \text{ } \cT \to \mathrm{Spec}\text{ }   \cO_{\cC,\mathrm{x}}$ given by the completion is in fact a bijection. 
	
	Let $\rho_{\cA_0}:G_\Q \to \GL_2(\cA_0)$ be the deformation of $\rho$ given by the pushforward of $\rho_{\cT}$ along the surjection $\cT \twoheadrightarrow \cA_0$. By a result of Wiles \cite{Wil88}, the $2$--dimensional $\mathbf{K}$-valued representation $\rho_{ \mathbf{K}}:=\rho_{\cA_0} \otimes_{\cA_0} \mathbf{K}:G_{\Q}\to \GL_{2}(\mathbf{K})$ attached to $\mathcal{F}$ is ordinary at $p$. Hence,  letting $\mathbb{V}_\textbf{K}:=\mathbf{K}^2$ be the underlying space of $\rho_{\mathbf{K}}$,  there exists a short exact sequence of $\mathbf{K}[G_{\Q_p}]$-modules
\begin{equation}\label{eq:ordinary_fil_rho_K}
		0\to \mathbb{V}_\textbf{K} '\to \mathbb{V}_\textbf{K} \to \mathbb{V}_\textbf{K}''\to 0,
\end{equation}
	where $\mathbb{V}_\textbf{K}^+,\mathbb{V}_\textbf{K}^-$ are $G_{\Q_p}$-stable $\mathbf{K}$-lines such that $G_{\Q_p}$ acts on $\mathbb{V}^-$ via the unramified character $\delta:G_{\Q_p} \to \textbf{K}^{\times}$ characterized by $\delta(\Frob_p)=\mathbf{U}_p\in \cA_0^\times$.
	 
	Let $\rho_{\cA}\colon G_\Q \to \GL_2({\cA})$ be the deformation of $\rho$ given by the pushforward of $\rho_{\cA_0}$ along the finite injection $\cA_0 \hookrightarrow {\cA}$, and let $\mathbb{V}$ be the underlying space of $\rho_{\cA}$. One has $\det\rho_{\cA}=\det \rho_{\cA_0}=\det(\rho) \cdot \chi_{\mathrm{cyc}}$ and $\mathrm{tr} \rho_{\cA}=\mathrm{tr}\rho_{\cA_0}$. Since $\cA$ is a discrete valuation ring, the exact sequence \eqref{eq:ordinary_fil_rho_K} induces a short exact sequence of $\cA[G_{\Q_p}]$-modules
		\begin{equation}\label{eq:ordinarity_L}
				0\to \mathbb{V}^+\to \mathbb{V}  \to \mathbb{V}^-\to 0,
		\end{equation}
	where $\mathbb{V}^+$ is a free direct summand of $\mathbb{V} $ of rank one over $\cA$, and $G_{\Q_p}$ acts on $\mathbb{V}^-\simeq \cA$ via $\delta$. 
	In fact, the above $\mathbb{V}^+$ and $\mathbb{V}^-$ can be defined as $\mathbb{V}^+=\mathbb{V}  \cap \mathbb{V}^+_\textbf{K}$ and $\mathbb{V}^-=\mathrm{im}(\mathbb{V}  \to  \mathbb{V}^-_\textbf{K})$. Here $\mathbb{V}^+$ is  a free rank one direct summand because $\cA$ is a discrete valuation ring.
	
		We also obtain by specialization at $Y=0$ an ordinary filtration of $V=\mathbb{V}/Y \mathbb{V} $:
	\[0\to V^+\to V \to V^-\to 0.\]
	
	The rest of this section is devoted to the proof of the following theorem.	
	\begin{thm}\label{thm:etaleness}
		Assume \eqref{tag:non_vanishing_resultant_P_L_P_M} and and \eqref{tag:slope_2t+z_one_line} hold.
			Then $\cA=\cA_0=\varLambda$, i.e., $e=1$ and $Y=X$. In other words, $\cF$ is étale at $f_\alpha$.
	\end{thm}
	
	We start the proof with some useful observations. Fix a $\cA$-basis $(e_{\cA,1},e_{\cA,2})$ of $\mathbb{V}$ and write
	\begin{equation}\label{eq:decomposition_rho_modulo_powers_of_Y} 
		\rho_{\cA} = \left(\Id+Y\xi_1+Y^2\xi_2+\ldots\right)\cdot \rho 
	\end{equation} 
	in $\GL_2(\cA)=\GL_2(\ob{\Q}_p\lsem Y \rsem)$. In the next lemma, we say that $\bV^+$ has slope $F(Y)\in\cA$ with respect to $(e_{\cA,1},e_{\cA,2})$ if $\bV^+$ is generated by $e_{\cA,1}+F(Y)\cdot e_{\cA,2}$.

	\begin{lemma}\label{lem:change_of_basis}
		Assume that $\xi_1=\ldots=\xi_{k-1}=0$ in \eqref{eq:decomposition_rho_modulo_powers_of_Y} for some integer $k\geq1$.
		\begin{enumerate}
			\item The map $\xi_k\colon G_{\Q} \to \End(\bV/Y\bV)=\End(V)$ is a cocycle in $\mathrm{Z}^1(G_\Q,\End(V))$, and its cohomology class lies in $\Sel(\ad\rho,V^+)$.
			\item Let $\xi_k'\colon G_{\Q} \to \End(V)$ be another cocycle cohomologous to $\xi_k$. Assume $\bV^+$ has slope $F(Y)\in\cA$ with respect to $(e_{\cA,1},e_{\cA,2})$. Then there exists an $\cA$-basis $(\tilde{e}_{\cA,1},\tilde{e}_{\cA,2})$ of $\bV$ such that
			\begin{enumerate}
				\item[(a)] $({e}_{\cA,1},{e}_{\cA,2})$ and $(\tilde{e}_{\cA,1},\tilde{e}_{\cA,2})$ induce the same basis of $\bV/Y^k\bV$,
				\item[(b)] We have $\rho_{\cA} = \left(\Id+Y^k\xi_k'+\ldots\right)\cdot \rho$ in the basis $(\tilde{e}_{\cA,1},\tilde{e}_{\cA,2})$, and
				\item[(c)] $\bV^+$ has slope $\tilde{F}(Y)$ with respect to $(\tilde{e}_{\cA,1},\tilde{e}_{\cA,2})$ satisfying $F(Y)\equiv\tilde{F}(Y) \mod Y^k$.
			\end{enumerate}
		\end{enumerate}
	\end{lemma}
	\begin{proof}
		We first check (i). The term in $Y^k$ in the relation $\rho_\cA(gh)=\rho_\cA(g)\rho_\cA(h)$ yields
		\[ \xi_k(gh)=\xi_k(g)+\rho(g)\xi_k(h)\rho(g)^{-1} \text{ for all $g,h\in G_\Q$,} \]
		so $\xi_k$ is indeed a cocycle. The ordinarity condition for $\rho_\cA$ reads 
		\[\rho_\cA(G_{\Q_p})\cdot \bV^+ \subseteq \bV^+ \quad \mbox{and}\quad (\rho_\cA-\Id)(I_p)\cdot \bV \subseteq \bV^+\]
		As $\rho_{|G_{\Q_p}}$ is scalar and $\rho_{|I_p}$ is trivial, the relation $\xi_k\equiv Y^{-k}(\rho_\cA\rho^{-1}-\Id) \mod Y$ shows that
		\[\xi_k(G_{\Q_p})\cdot V^+ \subseteq V^+ \quad \mbox{and}\quad \xi_k(I_p)\cdot V \subseteq V^+,\]
		\textit{i.e.}, $[\xi_k]\in \Sel(\ad\rho,V^+)$.
		
		To prove (ii), let $M\in\End(V)$ be such that $\xi_k=\xi_k'-\rho M \rho^{-1}+M$, and define $(\tilde{e}_{\cA,1},\tilde{e}_{\cA,2})$ as the image of $({e}_{\cA,1},{e}_{\cA,2})$ under $\Id+Y^kM\in\mathrm{Aut}_\cA(\bV)$. Then (a) clearly holds, and $\rho_\cA$ in this new basis has the form
		\begin{align*}
			(\Id+Y^kM)^{-1}(\Id+Y^k\xi_k+\ldots)\rho(\Id+Y^kM) &\equiv (\Id-Y^kM)(\Id+Y^k\xi_k)\rho(\Id+Y^kM) \mod Y^{k+1} \\
			&\equiv (\Id+Y^k\xi'_k)\rho \mod Y^{k+1},
		\end{align*}
		proving (b). Moreover, the slopes are related by the condition 
		\[\begin{pmatrix}1 & 0 \\ \tilde{F}(Y) & 1 \end{pmatrix}\cdot(1+Y^kM)=\begin{pmatrix}1 & 0 \\ F(Y) & 1 \end{pmatrix}\]
		which clearly implies (c).
	\end{proof}
	
	We turn to the proof of Theorem \ref{thm:etaleness} and we fix a basis $(e_1,e_2)$ of a $\ob{\Q}$-structure $V_{\ob{\Q}}$ of $V$ satisfying \eqref{tag:dihedral_basis_assumption} and adapted to $V^+$, so $V^+$ is generated by $e_1+s\cdot e_2$ for some $s\in\ob{\Q}_p$. Note that the slope $F(Y)$ of $\bV^+$ with respect to any $\cA$-basis $(e_{\cA,1},e_{\cA,2})$ of $\bV$ lifting $(e_1,e_2)$ must satisfy $F(0)=s$. We call $s$ the residual slope of $\bV^+$ (with respect to $(e_{\cA,1},e_{\cA,2})$).
	
	\begin{prop}\label{prop:first_deformations_rho_A}
		Assume \eqref{tag:non_vanishing_resultant_P_L_P_M}. 
		\begin{enumerate}
			\item There exists an $\cA$-basis $(e_{\cA,1},e_{\cA,2})$ of $\bV$ lifting $(e_1,e_2)$ in which $\rho_\cA$ can be written in the form $\left(\Id+Y^e\xi_e+Y^{e+1}\xi_{e+1}+\ldots\right)\cdot \rho$. 
			\item The residual slope $s$ is a root of the polynomial $\rQ(S)$ introduced in \eqref{eq:def_polynomial_computing_residual_slopes}, and $\xi_e$ is a cocycle whose cohomology class is the class $[\xi_{V^+}]\in\Sel(\ad\rho,V^+)$ of Lemma \ref{lem:definition_xi_1}.
		\end{enumerate}
	\end{prop}
	\begin{proof}
		We first prove (i). Let $(e_{\cA,1},e_{\cA,2})$ be any $\cA$-basis of $\bV$ lifting $(e_1,e_2)$ in which $\rho_\cA$ is of the form $\left(\Id+Y^n\xi_n+Y^{n+1}\xi_{n+1}+\ldots\right)\cdot \rho$ with $n\geq 1$ maximal. We claim that $n=e$. Indeed, we have 
		\begin{equation}\label{eq:computation_det_mod_Y_n}
			\det \rho_{\cA}\equiv \det(1+Y^n\xi_n)\det\rho\equiv (1+Y^n\tr\xi_n)\det \rho \mod Y^{n+1}.
		\end{equation}
		by assumption. But we also know that
		\begin{equation}\label{eq:computation_det_mod_Y_e}
			\det \rho_{\cA}=\chi_{\cyc} \cdot \det \rho \equiv (1+Y^e\lambda)\cdot \det\rho \mod Y^{e+1}.
		\end{equation}
		In particular, we must have $n\leq e$. Assume now $n<e$, so that $\tr\xi_n=0$. As $[\xi_n]\in \Sel(\ad\rho,V^+)$ by Lemma \ref{lem:change_of_basis} (i), this shows that, in fact, $[\xi_n]$ belongs to $\Sel(\ad^0\rho,V^+)$, which is trivial by Proposition \ref{prop:dimension_selmer_ad_0_rho}. Therefore, Lemma \ref{lem:change_of_basis} (ii) with $k=n$ and $\xi_n'=0$ provides us with another basis lifting $(e_1,e_2)$ in which $\rho_\cA$ has the form $\left(\Id+Y^{n+1}\xi_{n+1}+\ldots\right)\cdot \rho$, contradicting the choice of $(e_{\cA,1},e_{\cA,2})$. Hence, $n= e$, proving (i).
		
		It is clear from \eqref{eq:computation_det_mod_Y_n} and \eqref{eq:computation_det_mod_Y_e} that $\tr\xi_e=\lambda$. Since $[\xi_e]\in \Sel(\ad\rho,V^+)$, this forces $[\xi_e]=[\xi_{V^+}]$ by Lemma \ref{lem:definition_xi_1}. The fact that $s$ is a root of $\rQ(S)$ then follows from Proposition \ref{prop:dimension_selmer_ad_rho}, hence proving (ii).
	\end{proof}	
	
	We analyze the terms of higher order appearing in $\rho_\cA$ in the next proposition.
	
	\begin{prop}\label{prop:higher_order_deformation_rho_A}
		Assume \eqref{tag:non_vanishing_resultant_P_L_P_M} and \eqref{tag:slope_2t+z_one_line} hold.
		Then for all integers $k\geq e$, there exists an $\cA$-basis $(e^{(k)}_{\cA,1},e^{(k)}_{\cA,2})$ of $\bV$  lifting $(e_1,e_2)$, satisfying the following three conditions.
		\begin{enumerate}
			\item[(a)] $\rho_\cA$ is of the form 
			\[\rho_{\cA}\equiv(\Id+Y^e\xi_e+Y^{2e}\xi_{2e}+\ldots +Y^{me}\xi_{me}+Y^{k}\xi_{k})\cdot\rho \mod Y^{k+1}\]
			in $(e^{(k)}_{\cA,1},e^{(k)}_{\cA,2})$, where $m\geq 0$ is such that $me < k\leq me+e$.
			\item[(b)] If $k>e$, then the slope $F(Y)$ of $\bV^+$ with respect to $(e^{(k)}_{\cA,1},e^{(k)}_{\cA,2})$ is of the form 
			\[F(Y)\equiv s+ s_eY^e+s_{2e}Y^{2e}+\ldots +s_{(m-1)e}Y^{(m-1)e}+s_{k-e}Y^{k-e} \mod Y^{k-e+1}. \]
			\item[(c)] If $k$ is not a multiple of $e$, then $\xi_k=0$ and $s_{k-e}=0$.
		\end{enumerate}
	\end{prop}
	\begin{proof}
		If $k=e$, then this follows from Proposition \ref{prop:first_deformations_rho_A}. We prove the result by induction on $k$ and we assume we already constructed $(e_{\cA,1},e_{\cA,2}):=(e^{(k-1)}_{\cA,1},e^{(k-1)}_{\cA,2})$ for some $k>e$. If $k$ is a multiple of $e$, \textit{i.e.}, $k=me+e$, then we may simply take $(e^{(k)}_{\cA,1},e^{(k)}_{\cA,2})=(e_{\cA,1},e_{\cA,2})$ and there is, in fact, nothing to prove.
		
		Assume $k$ is not a multiple of $e$. One sees from parts (a) and (c) for $(e_{\cA,1},e_{\cA,2})$ that we may write $\rho_\cA$ as $(\Id+Y^e\xi_e+Y^{2e}\xi_{2e}+\ldots +Y^{me}\xi_{me}+Y^{k}\xi+\ldots)\cdot\rho$ in the basis $(e_{\cA,1},e_{\cA,2})$ for some map $\xi\colon G_{\Q} \to \End(V)$. 
		For $g,h\in G_\Q$, considering the term in $Y^{k}$ in the equality $\rho_{\cA}(gh)=\rho_{\cA}(g)\rho_{\cA}(h)$ shows that $\xi(gh)=\xi(g)+\rho(g)\xi(h)\rho(g)^{-1}$, \textit{i.e.}, $\xi\in \mathrm{Z}^1(G_\Q,\End(V))$. Moreover, one has $\tr \xi=0$. Indeed, $\tr\xi$ arises as the coefficient in $Y^k$ of the expansion of $\det\rho_{\cA}$ as a power series in $Y$. But this power series belongs to $\varLambda$ and $e\nmid k$, so the term in $Y^k$ must be zero.
		
		Let $W=\End^0(V)$. We claim that the cohomology class $[\xi]\in \HH^1(\Q,W)$ is trivial.
		As $\rho_{|G_{\Q_p}}$ is scalar, the ordinarity condition for $\rho_\cA$ in $(e_{\cA,1},e_{\cA,2})$ implies that 
		 \begin{equation}\label{eq:expansion_ordinarity_condition}
		 	\begin{pmatrix}1 & 0 \\ -F(Y) &1 \end{pmatrix} (\Id+Y^e\xi_e+Y^{2e}\xi_{2e}+\ldots +Y^{me}\xi_{me}+Y^{k}\xi+\ldots) \begin{pmatrix}1 & 0 \\ F(Y) &1 \end{pmatrix}=\begin{pmatrix}*& * \\ 0&\mathbf{unr} \end{pmatrix}
		 \end{equation}
		on $G_{\Q_p}$, where $\textbf{unr}$ stands for an undetermined unramified character. Define 
		\begin{equation}\label{eq:def_xi_ord}
			\xi_{\ord} = \begin{pmatrix}1 & 0 \\ -s &1 \end{pmatrix} \xi \begin{pmatrix}1 & 0 \\ s &1 \end{pmatrix},
		\end{equation}
		and similarly define $\xi_{e,\ord}$.
		 By (b) for $(e_{\cA,1},e_{\cA,2})$, we may write $(\begin{smallmatrix}1& 0\\ F(Y)&1 \end{smallmatrix})$ as $(\begin{smallmatrix}1 & 0 \\ s &1 \end{smallmatrix}) + s_eY^e(\begin{smallmatrix}0 & 0 \\ 1 &0  \end{smallmatrix})+\ldots + s_{k-e}Y^{k-e}(\begin{smallmatrix}0 & 0 \\ 1 &0 \end{smallmatrix})+\ldots$ and using that $(\begin{smallmatrix}0 & 0 \\ 1 &0 \end{smallmatrix})(\begin{smallmatrix}1 & 0 \\ s &1 \end{smallmatrix})=(\begin{smallmatrix}1 & 0 \\ -s &1 \end{smallmatrix})(\begin{smallmatrix}0 & 0 \\ 1 &0 \end{smallmatrix})=(\begin{smallmatrix}0 & 0 \\ 1 &0 \end{smallmatrix})$, the term in $Y^k$ of \eqref{eq:expansion_ordinarity_condition} gives the relation
		\begin{equation}\label{eq:relation_giving_2t+z}
			\xi_{\ord} + s_{k-e} \left[  \xi_{e,\ord}; \begin{pmatrix} 0  & 0 \\ 1 & 0 \end{pmatrix} \right]=\begin{pmatrix}\ast & \ast \\ 0 &\mathbf{unr} \end{pmatrix},
		\end{equation}
		on $G_{\Q_p}$, where we have put $[A;B]=AB-BA$ for two square matrices $A$ and $B$. 
		
		We claim that $s_{k-e}=0$. As $[\xi_e]=[\xi_{V^+}]$ by Proposition \ref{prop:first_deformations_rho_A} (ii), $\xi_{e,\ord|G_{\Q_p}}$ is the left hand side of \eqref{eq:xi_1_in_ordinary_basis}. If $s_{k-e}\neq 0$, then a computation of the $(2,1)$-entries in \eqref{eq:relation_giving_2t+z} with Proposition \ref{prop:calculation_Up} shows that the \textbf{c}-slope of $\xi$ is given by $-2t$, where $t$ is defined in \eqref{eq:xi_1_in_ordinary_basis}. But this contradicts \eqref{tag:slope_2t+z_one_line}, so $s_{k-e}=0$. In particular, \eqref{eq:relation_giving_2t+z} implies that $[\xi]\in\Sel(\ad^0\rho,V^+)$, which is trivial by Proposition \ref{prop:dimension_selmer_ad_0_rho}. We may therefore apply Lemma \ref{lem:change_of_basis} (ii) to find a new basis $(e^{(k)}_{\cA,1},e^{(k)}_{\cA,2})$ in which (a), (b) and (c) are valid. This finishes the proof of the proposition.
	\end{proof}
	
	\begin{proof}[Proof of Theorem \ref{thm:etaleness}]
		 As $\varLambda$ is closed in $\cA$ for the $(Y)$-adic topology, Proposition \ref{prop:higher_order_deformation_rho_A} shows that $\tr\rho_\cA$ takes values in $\varLambda$. The same proposition shows that the image of $\mathbf{U}_p$ in $\cA_0$ also lies in $\varLambda$ modulo $Y^k$ for all integers $k>0$, hence in $\varLambda$. Indeed, given $k>0$, it is obtained as the $(2,2)$-entry of the conjugate of $\rho_\cA(\Frob_p)$, computed in the basis $(e^{(k)}_{\cA,1},e^{(k)}_{\cA,2})$, by $(\begin{smallmatrix}1 & 0 \\ F_k(Y) & 1 \end{smallmatrix})$. Now, $\cA_0$ is generated by the values of $\tr\rho_\cA$, so $\cA_0=\varLambda$, as wanted.
	\end{proof}
	
\subsection{Residual slopes of Hida families}\label{sloplesresidually}
	As a consequence of Theorem \ref{thm:etaleness}, if  hypotheses \eqref{tag:non_vanishing_resultant_P_L_P_M} and \eqref{tag:discriminant} hold, then there is a one-to-one correspondence between minimal prime ideals of $\cT$ and Hida eigenfamilies $\cF$ with coefficients in $\varLambda$ passing through the point corresponding to $f_\alpha$. We shall now show that these Hida families are entirely determined by the residual slope of their ordinary filtration. 
	\medskip
	
	Hereafter, we denote by $\rho_\cF$ the ordinary $\varLambda$-adic Galois representation $\rho_{\cA_0}$, and by $\bV_\cF$ its underlying space . We also let  $\bV_\cF^+$, and $\bV_\cF^-$  be the ordinary line and quotient of $\bV_\cF$.
	
\newpage 
\begin{thm}\label{thm:uniqueness_residual_slope}
		Assume that hypothesis \eqref{tag:non_vanishing_resultant_P_L_P_M} holds. 
		\begin{enumerate}
			\item Let $\cF$ and $\cF'$ be two Hida eigenfamilies passing through $f_\alpha$. Assume that \[V^+:=\mathbb{V}_{\mathcal{F}}^+/X\mathbb{V}_{\mathcal{F}}^+=\mathbb{V}_{\mathcal{F}'}^+/X\mathbb{V}_{\mathcal{F}'}^+\]
            as $\ob{\Q}_p$-lines of $V$, and that \eqref{tag:slope_2t+z_one_line} holds. Then $\cF=\cF'$.
		\item Assume that hypothesis \eqref{tag:discriminant} holds as well. Then there are at most four Hida eigenfamilies passing through $f_\alpha$.
		\end{enumerate}
	\end{thm}
	
	Let $\cF$, $\cF'$ and $V^+$ be as in Theorem \ref{thm:uniqueness_residual_slope} (i). Fix a basis $(e_1,e_2)$ of a $\ob{\Q}$-structure $V_{\ob{\Q}}$ of $V$ satisfying \eqref{tag:dihedral_basis_assumption} and adapted to $V^+$. 
	We deduce Theorem \ref{thm:uniqueness_residual_slope} from the following result whose proof is similar to that of Proposition \ref{prop:higher_order_deformation_rho_A}.
	
	\begin{prop}\label{prop:congruences_between_Hida_families}
		Fix a $\varLambda$-basis of $\bV_{\cF'}$ lifting $(e_1,e_2)$, and let $F'(X)$ be the slope of $\bV_{\cF'}$ with respect to this basis. Then for all $k\geq 1$, there exists a $\varLambda$-basis $(e^{(k)}_{\cF,1},e^{(k)}_{\cF,2})$ of $\bV_\cF$ such that 
		\begin{enumerate}
			\item[(a)] $\rho_{\cF} \equiv \rho_{\cF'} \mod X^{k+1}$ after identifying $\GL_{\varLambda}(\bV_\cF)$ and  $\GL_{\varLambda}(\bV_{\cF'})$ with $\GL_2(\varLambda)$, and
			\item[(b)] the slope $F_{k}(X)$ of $\bV^+_{\cF}$ with respect to $(e^{(k)}_{\cF,1},e^{(k)}_{\cF,2})$ satisfies $F_{k}(X)\equiv F'(X) \mod X^{k}$.
		\end{enumerate}
	\end{prop}
	
	\begin{proof}
		We proceed by induction on $k$. Assume first $k=1$. As $\mathbb{V}_{\mathcal{F}}^+/X\mathbb{V}_{\mathcal{F}}^+=\mathbb{V}_{\mathcal{F}'}^+/X\mathbb{V}_{\mathcal{F}'}^+$ by assumption, any basis of $\bV_\cF$ lifting $(e_1,e_2)$ will satisfy (b). Choose an arbitrary such basis and write $\rho_{\cF}$ (resp. $\rho_{\cF'}$) as 
		\begin{equation}\label{eq:decompositions_rho_F_and_rho_F'}
			(\Id+X\xi_1+X^{2}\xi_2+\ldots )\cdot\rho \qquad \text{(resp. $(\Id+X\xi'_1+X^{2}\xi'_2+\ldots )\cdot\rho$)}
		\end{equation}
		in this basis (resp. in the fixed basis of $\bV_{\cF'}$). Then, by Proposition \ref{prop:first_deformations_rho_A} (ii) (with $e=1$, \textit{i.e.}, $Y=X$), $[\xi_1]=[\xi_{V^+}]=[\xi'_1]$. We then obtain from Lemma \ref{lem:change_of_basis} (ii) (with $k=1$) that there exists another basis $(e^{(1)}_{\cF,1},e^{(1)}_{\cF,2})$ of $\bV_{\cF}$ lifting $(e_1,e_2)$ in which \eqref{eq:decompositions_rho_F_and_rho_F'} holds with $\xi_1=\xi_1'$. This basis then satisfies (a) and (b), proving the result for $k=1$.
		
		Let $k\geq 1$ and assume we have constructed $(e^{(k)}_{\cF,1},e^{(k)}_{\cF,2})$. Write $\rho_{\cF}$ (resp. $\rho_{\cF'}$) as in \eqref{eq:decompositions_rho_F_and_rho_F'} in the basis $(e^{(k)}_{\cF,1},e^{(k)}_{\cF,2})$ (resp. in the fixed basis of $\bV_{\cF'}$). In particular, we have $\xi_i=\xi_i'$ for $i\leq k$. Moreover, for $F(X):=F_k(X)=s_0+s_1X+\ldots$ and $F'(X)=s_0'+s_1'X+\ldots$, we have $s_i=s'_i$ for $i\leq k-1$.
		
		Using that $\rho_{\cF}$ and $\rho_{\cF'}$ are group homomorphisms and $\xi_i=\xi_i'$ for $i\leq k$, one sees that $\xi:=\xi_{k+1}-\xi_{k+1}'$ is a cocycle in $\mathrm{Z}^1(G_{\Q},\End(V))$. Moreover, we have \[1=\det(\rho_{\cF})\cdot\det(\rho_{\cF'})^{-1}\equiv 1+X^{k+1}\tr\xi \mod X^{k+2},\]
		so $[\xi]\in\HH^1(\Q,W)$, where $W=\End^0(V)$. 
		
		We now determine the local properties of $\xi$ using the ordinarity of $\rho_{\cF}$ and $\rho_{\cF'}$.
		For $i\geq 1$, we put $\xi_{i,\ord}=(\begin{smallmatrix}
			1 & 0 \\ -s & 1
		\end{smallmatrix})\xi_i (\begin{smallmatrix}1 & 0 \\ s & 1\end{smallmatrix})$ and we similarly define $\xi'_{i,\ord}$ and $\xi_{\ord}$. 
		Using the ordinarity condition for $\cF$, a similar computation to that proving \eqref{eq:relation_giving_2t+z} gives
		\[
		\xi_{k+1,\ord}+s_k\cdot [\xi_{1,\ord},(\begin{smallmatrix}0 & 0 \\ 1 & 0\end{smallmatrix})]+\zeta=(\begin{smallmatrix}
			* & * \\ 0 & \textbf{unr}
		\end{smallmatrix})
		\]
		on $G_{\Q_p}$, where $\zeta$ is a term only depending on $s_{i-1},\xi_{i}$ for $1\leq i\leq k$, and $\textbf{unr}$ is an undetermined unramified character. Considering the ordinarity condition for $\cF'$, a subtraction then yields  
		\begin{equation}\label{eq:ordinarity_induction}
			\xi_{\ord}+ (s_k-s_k')[\xi_{1,\ord},(\begin{smallmatrix}0 & 0 \\ 1 & 0\end{smallmatrix})]=(\begin{smallmatrix}
			* & * \\ 0 & \ur
		\end{smallmatrix}) 
		\end{equation}
		on $G_{\Q_p}$. If $s_k\neq s_k'$, then the exact same argument as in Proposition \ref{prop:higher_order_deformation_rho_A} shows that \eqref{tag:slope_2t+z_one_line} does not hold, so $s_k=s_k'$. But \eqref{eq:ordinarity_induction} then implies that the class of $\xi$ lies in the space $\Sel(\ad^0\rho,V^+)$, which is $\{0\}$ by Proposition \ref{prop:dimension_selmer_ad_0_rho}, so $\xi_k$ and $\xi_k'$ are cohomologous. By Lemma \ref{lem:change_of_basis} (ii), we may find a basis $(e^{(k+1)}_{\cF,1},e^{(k+1)}_{\cF,2})$ of $\bV_\cF$ in which \eqref{eq:decompositions_rho_F_and_rho_F'} with $\xi_{i}=\xi'_i$ holds for $1\leq i\leq k+1$, and such that $\bV_\cF^+$ has slope congruent to $F'(X)$ modulo $X^{k+1}$. This proves (a) and (b) for $k+1$.
	\end{proof}
	
	\begin{proof}[Proof of Theorem \ref{thm:uniqueness_residual_slope}]
	We first prove (i). If $\cF$ and $\cF'$ satisfy the conditions in (i), then Proposition \ref{prop:congruences_between_Hida_families} shows that $\tr\rho_{\cF}= \tr\rho_{\cF'}$, so $\cF=\cF'$ by strong multiplicity one for Hida families.
		
		Consider now (ii). As $\cT$ is finite over $\varLambda$, there exists by Lemma \ref{lem:existence_basis_adapted_to_ordinary_lines} a basis $(e_1,e_2)$ of $V_{\ob{\Q}}$ satisfying \eqref{tag:dihedral_basis_assumption} and which is adapted to $\bV^+_\cF/X\bV_\cF^+$ for any family $\cF$ passing through $f_\alpha$.
		By Proposition \ref{prop:first_deformations_rho_A} (ii), the slope $s$ in the basis $(e_1,e_2)$ of $V^+=\bV^+_\cF/X\bV_\cF^+$ for any $\cF$ is a root of the polynomial $\rQ(S)$, which is (nonzero by \eqref{tag:discriminant}, and) of degree at most $4$. As $s$ determines $\cF$ by (i), there are indeed at most $4$ families $\cF$ passing through $f_\alpha$.
	\end{proof}
	
	\section{Structure of the Hecke algebra at weight one points}\label{sec:geometry2}
	In \S \ref{sec:geometry}, we established that at most four components of the eigencurve meet at $x$ and that each is \'etale over the weight space. To conclude the proof of  Theorem  \ref{thm:main_result}, it remains to show that the four components exist and intersect at $\mathrm x$ as generically as possible.

\subsection{Overview of the proof of Theorem \ref{thm:main_result}, part \ref{part2}}\label{overview2}
	The main tool of this section is the study of the overconvergent generalized eigenspace corresponding to the system of Hecke eigenvalues of $f_{\alpha}$, denoted by $S_1^{\dagger}(N,\ob{\Q}_p)\lsem f_\alpha\rsem$. As noted by Darmon, Rotger and Lauder \cite{DLR4}, the Fourier coefficients of elements of this generalized eigenspace can be described in terms of suitable $S$-units. As a consequence of the étaleness of each irreducible component of $\cC$ through $\mathrm x$ established in \S \ref{sec:geometry}, combined with Hida duality \eqref{eq:Hida_duality}, the number of Hida eigenfamilies passing through $\mathrm x$ is equal to $\ob{\Q}_p$-dimension of the space $S_1^{\dagger}(N,\ob{\Q}_p)\lsem f_\alpha\rsem$. Since, by assumption, the form $f$ is $p$-irregular, it follows that
 $f,f_\alpha \in S_1^{\dagger}(N,\ob{\Q}_p)\lsem f_\alpha\rsem$, which implies that there are at least two irreducible components of $\cC$ passing through $f$.  We use the explicit description of the Fourier coefficients of forms in the generalized eigenspace and results in $p$-adic transcendental number theory to show that the dimension is at least four.  The idea is that, given two Hida eigenfamilies $\cF_1$ and $\cF_2$ specializing to $f$ at weight one, the specialization of $\frac{\cF_1-\cF_2}{X}$ is an overconvergent cusp form  $g^{\dagger} \in S_1^{\dagger}(N,\ob{\Q}_p)\lsem f_\alpha\rsem$. 
 To show that the dimension of $S_1^{\dagger}(N,\ob{\Q}_p)\lsem f_\alpha\rsem$ is greater than 2, it suffices to show that $g^{\dagger}$ is not in the  span of $f$ and $f_\alpha$ using $p$-adic transcendence theory. Similarly, the argument can be strengthened to show that the dimension is greater than 3. We note that our proof also settles a conjecture on the generalized eigenspace overconvergent eigenspace of $f_{\alpha}$ formulated in \cite{DLR4}.

	\subsection{Generalized overconvergent eigenforms}
	
	We let $\mathcal{U}:=\mathrm{Sp}(A)$ be a (sufficiently small) affinoid of the weight space $\cW$ containing $\kappa(f_\alpha)$  and $\textbf{M}^{\ord}_{\kappa^{\mathcal{U}}}(N)$ be the finite type projective $A$-module of Coleman families of slope zero and weight varying in $\mathcal{U}$ \cite[Def.5.1]{Pi13}. Let $\cT_{\mathcal{U}}$ be the finite flat  Hecke $A$-algebra acting faithfully on $\textbf{M}^{\ord}_{\kappa^{\mathcal{U}}}(N)$. In fact, 
	$\mathcal{V}:=\mathrm{Sp}(\cT_{\mathcal{U}})$ is an admissible open of the eigencurve $\cC$ and the weight map arises from the finite flat map  $A \to \cT_{\mathcal{U}}$.
	
	 Let $\textbf{S}^{\ord}(N,\varLambda)_{f_\alpha}$ be the $\cT$-module of cuspidal Hida families with coefficients in $\varLambda$ which specialize to $f_\alpha$ at $X=0$, i.e., $\textbf{S}^{\ord}(N,\varLambda)_{f_\alpha}$ is the completed localization of $\textbf{M}^{\ord}_{\kappa^{\mathcal{U}}}(N)$ at the maximal ideal of Hecke algebra $\cT_{\mathcal{U}}$ corresponding to the weight cusp form $f_\alpha$ ($f_\alpha$ lies over $\kappa(f_\alpha) \in \mathcal{U}$).

	By Hida duality, the pairing $(T,\cF) \mapsto a_1(T\cdot \cF)$ is a $\varLambda$-linear perfect pairing which identifies $\cT$ with $\Hom_{\varLambda}(\textbf{S}^{\ord}(N,\varLambda)_{f_\alpha},\varLambda)$ (\cite[Prop. 1.6]{betinadimitrovJTNB}). We obtain by moding out by $X$ a $\ob{\Q}_p$-linear perfect pairing
	\begin{equation}\label{eq:Hida_duality}
		\cT/\Gm_{\varLambda}\cdot \cT \times S_1^{\dagger}(N,\ob{\Q}_p)\lsem f_\alpha\rsem  \to \ob{\Q}_p. 
	\end{equation}
	
	\par

	We explain how to compute the Fourier coefficients of certain generalized eigenforms, following the main results of \cite{DLR4}. Working with a basis of $V$ satisfying \eqref{tag:dihedral_basis_assumption}, we keep the notations in \S\ref{sec:selmer_groups} and, in particular, the basis $(w_1,w_2,w_3)$ of $W=\End^0(V)$, the basis $(\kappa,\kappa^{(\ell)})$ of $\Hom_{G_\Q}(W_{\ob{\Q}},U_{H,\ob{\Q}}^{(\ell)})$ and the matrices $L,M,M^{(\ell)}$ of \eqref{eq:def_matrices_L_and_M} and \eqref{eq:def_matrice_M_ell} for regular primes $\ell$.
	
	In order to state the next proposition, we put $G=\Gal(H/\Q)$, we define $U\subseteq W$ to be the line generated by $L_1w_1+L_2w_2+L_3w_3$ and, for any regular prime $\ell\nmid Np$, we let $w_\ell=\rho(\Frob_{\ell})-\frac{1}{2}\tr(\rho(\Frob_{\ell}))$ be the canonical generator of the line $W^{G_\ell}$ fixed by $G_\ell=\Gal(H_{\Gp_\ell}/\Q_\ell)$. We also denote by $\ord_\ell \colon \cO_H[1/\ell]^\times \to \Z$ the valuation map associated to the chosen prime $\Gp_\ell$ and, by slight abuse of notation, we also denote by $\ord_\ell$ its $\ob{\Q}_p$-linear extension $U_H^{(\ell)} \to \ob{\Q}_p$.
	\begin{prop}\label{prop:calcul_coeff_DLR}
		There is an isomorphism $\theta \colon \HH^1(\Q,W) \to W/U$ and, for all primes $\ell\nmid Np$ at which $f$ is regular, a element $o(\ell)\in\ob{\Q}_p^\times$ with the following properties. Let $[\xi]\in\HH^1(\Q,W)$ and write $\theta([\xi])$ as $x_1w_1+x_2w_2+x_3w_3$. Then we have $\tr(\xi(\Frob_{\ell})\rho(\Frob_{\ell}))=0$ if $f$ is irregular at $\ell$, and otherwise,
		\begin{equation}\label{eq:formula_DLR_coeffs}
			\tr(\xi(\Frob_{\ell})\rho(\Frob_{\ell}))=o(\ell)\cdot\begin{vmatrix}
			x_1 & L_1 & M_1^{(\ell)} \\
			x_2 & L_2 & M_2^{(\ell)} \\
			x_3 & L_3 & M_3^{(\ell)} 
		\end{vmatrix},
		\end{equation}
		with $o(\ell)$ characterized by the relation 
		\begin{equation}\label{eq:def_o_ell}
			\frac{\#G_\ell}{\#G}\cdot w_\ell=o(\ell)\cdot \left(\ord_\ell(\kappa^{(\ell)}(w_3))w_1+\frac{1}{2}\ord_\ell(\kappa^{(\ell)}(w_2))w_2+\ord_\ell(\kappa^{(\ell)}(w_1))w_3\right)
		\end{equation}
		in $W$.
	\end{prop}
	\begin{proof}
		We use the calculations and the notations in \cite[Proof. of Thm. 26]{DLR4}. Fix for the moment a prime $\ell\nmid Np$ at which $f$ is regular. We first compute the elements $w(1)$ and $w(\ell)\in W$ of \textit{loc. cit.} in our preferred basis $\{w_1,w_2,w_3\}$. To this end, we identify $W$ with its linear dual $W^*$ using the canonical nondegenerate pairing $\langle w,w'\rangle=\tr(ww')$. The Galois action then preserves $\langle\cdot ,\cdot \rangle$ and the dual basis of $\{w_1,w_2,w_3\}$ is $\{w_3,\tfrac{1}{2}w_2,w_1\}$, so $W^*\simeq W$ is $G$-equivariant and there are isomorphisms 
		\begin{align}
			\mu\colon &\Hom_{G}(W,U_H)\simeq (U_H\otimes W)^{G}, \\ \mu^{(\ell)}\colon & \Hom_{G}(W,\dfrac{U_H^{(\ell)}}{U_H})\simeq (\dfrac{U_H^{(\ell)}}{U_H}\otimes W)^{G},
		\end{align}
		both given by 
		\[\underline{\kappa} \mapsto \underline{\kappa}({w}_3)\otimes{w}_1+\frac{1}{2}\underline{\kappa}({w}_2)\otimes{w}_2+\underline{\kappa}({w}_1)\otimes{w}_3.\]
		The element $w(1)\in W$ is, by definition, any nonzero element in the image of the map $\lambda \otimes \Id \colon (U_H\otimes W)^{G} \to W$, $\sum_{1\leq i\leq 3} u_i\otimes w_i \mapsto \lambda(u_i)w_i$, where $\lambda$ is the normalized logarithm introduced in \S\ref{sec:cohomology_Artin}. We may take for $w(1)$ the image of $\mu(\kappa)$ under this map, \textit{i.e.}, 
		\begin{equation}\label{eq:def_w1}
			w(1)=L_3w_1+\tfrac{1}{2}L_2w_2+L_1w_3.
		\end{equation}
		
		The definition of $w(\ell)$ involves the choice of an element $u_\ell\in U_H^{(\ell)}$ of the form $\tilde{u} \otimes h^{-1}$, where $h$ is the class number of $H$ and $\tilde{u}\in H^\times$ generates $\Gp_\ell^h$. Notice that, for $\sigma\in \Gal(H/\Q)$, $\ord_\ell(\sigma(u_\ell))=1$ if $\sigma \in \Gal(H_{\Gp_\ell}/\Q_\ell)$ and $\ord_\ell(\sigma(u_\ell))=0$ otherwise. First define the element (denoted by $\xi(u_\lambda,w_\lambda)$ in \textit{loc. cit.})
		\begin{equation}
			J^{(\ell)}:=\frac{1}{\#G} \cdot \sum_{\sigma \in G} \sigma(u_\ell) \otimes \sigma(w_\ell)\in (U_H^{(\ell)}/U_H\otimes W)^{G}.
		\end{equation}
		Let $o(\ell)\in\ob{\Q}_p$ be determined by the relation $J^{(\ell)}=o(\ell)\cdot\mu^{(\ell)}(\kappa^{(\ell)})$. If we apply to this relation the map $\ord_\ell\otimes \Id \colon (U_H^{(\ell)}/U_H\otimes W)^{G} \to W^{G_\ell}$ given by $u \otimes w \mapsto \ord_\ell(u)w$, then we obtain exactly \eqref{eq:def_o_ell}. As $w_\ell\neq 0$, this also shows that $o(\ell)\neq 0$. 
		
		The element $w(\ell)\in W$ is defined in \textit{loc. cit.} as the image under $\log_p \otimes \Id \colon (U_H^{(\ell)}/U_H\otimes W)^{G} \to W/U$ of $J^{(\ell)}$, so we have the relation 
		 \begin{equation}\label{eq:def_w_ell}
		 	w(\ell)=o(\ell)\cdot \log_p(1+p^\nu)\cdot \left(M_3^{(\ell)}w_1+\frac{1}{2}M_2^{(\ell)}w_2+M_1^{(\ell)}w_3\right).
		 \end{equation} 
		 We now define $\tilde{\theta}\colon \HH^1(\Q,W) \simeq W/\langle w(1)\rangle$ as the inverse of the isomorphism $w\mapsto \varphi_w$ defined in \cite{DLR4}, where it is shown (in (50) and the following equation) that, for $w=\tilde{\theta}([\xi])$ with $[\xi]\in\HH^1(\Q,W)$, 
		 \begin{equation}\label{eq:formula_DLR_cited_as_such}
		 	\tr(\xi(\Frob_{\ell})\rho(\Frob_{\ell}))=\det(w,w(1),w(\ell)).
		 \end{equation}
		 We finally re-normalize $\tilde{\theta}$ by letting $\theta$ be $\tilde{\theta}$ followed by the linear isomorphism $W/\langle w(1)\rangle \to W/U$ induced by $x_2w_1+x_2w_2+x_3w_3\mapsto \log_p(1+p^\nu)/2\cdot (x_3w_1+2x_2w_2+x_1w_3)$. Formula \eqref{eq:formula_DLR_coeffs} then follows from \eqref{eq:formula_DLR_cited_as_such} together with \eqref{eq:def_w1} and \eqref{eq:def_w_ell}.
	\end{proof}	
	
	\begin{cor}\label{cor:vanishing_cocycle_at_frob_ell}
		If a cocycle class $[\xi]\in\HH^1(\Q,W)$ is such that $\tr(\xi(\Frob_{\ell})\rho(\Frob_{\ell}))=0$ for all primes $\ell\nmid Np$, then $[\xi]=0$.
	\end{cor}
	\begin{proof}
		Let $x_1w_1+x_2w_2+x_3w_3=\theta([\xi])$ for some $[\xi]\in\HH^1(\Q,W)$. Assume first $\rho$ is exotic, in which case all the $M_i^{(\ell)}$, $\ell\nmid Np$ are $\ob{\Q}$-linearly independent. Take $7$ regular primes $\ell_1,\ldots,\ell_7$ (there are infinitely many such, by Cebotarev's theorem). Then, by Lemma \ref{lem:Dirichlet} and Proposition \ref{prop:roy_waldschmidt}, the $3$-by-$7$ matrix $(M_{i}^{(\ell_j)})$ ($1\leq i\leq 3$, $1\leq j\leq 7$) has rank $3$. Hence, if $\tr(\xi(\Frob_{\ell})\rho(\Frob_{\ell_j}))=0$ for $1\leq j\leq7$, then $(x_1,x_2,x_3)$ and $(L_1,L_2,L_3)$ are proportional, and $[\xi]=0$ as $\theta$ is injective by Proposition \ref{prop:calcul_coeff_DLR}. 
		
		The RM case is two treated similarly: using Lemma \ref{lem:Dirichlet} and Proposition \ref{prop:roy_waldschmidt} one sees that the $3$-by-$4$ matrix $(M_{i}^{(\ell_j)})$ ($1\leq i\leq 3$, $1\leq j\leq 4$) has rank $3$ if $\ell_1$ is chosen to be a regular prime which splits in the real quadratic field $K$ and $\ell_2,\ell_3,\ell_4$ are inert in $K$.
	\end{proof}
	
	\begin{prop}\label{prop:difference_two_Hida_families} Assume \eqref{tag:non_vanishing_resultant_P_L_P_M} and \eqref{tag:discriminant} hold.
		Let $\cF$ and $\cF'$ be two Hida families passing through $f_\alpha$, and let $V^+(\cF)=\bV^+_{\cF}/X\bV^+_{\cF}$,  $V^+(\cF')=\bV^+_{\cF'}/X\bV^+_{\cF'}$.
		\begin{enumerate}
			\item If $\cF \equiv \cF' \mod X^2$, then $\cF=\cF'$.
			\item Assume $\cF\neq\cF'$ and define
		\begin{equation}\label{eq:def_g_F_F'}
			g_{\cF,\cF'}=\frac{\cF-\cF'}{X}\big|_{X=0}\in S_1^{\dagger}(N,\ob{\Q}_p)\lsem f_\alpha\rsem .
		\end{equation}
		With the notation introduced in Lemma \ref{lem:definition_xi_1}, the class 
		\[ [\xi]:=[\xi_{V^+(\cF)}]-[\xi_{V^+(\cF')}]\in\HH^1(\Q,W)\]
		is nonzero. Moreover, for all primes $\ell\nmid Np$, we have  $a_\ell(g_{\cF,\cF'})=\tr(\xi(\Frob_{\ell})\rho(\Frob_{\ell}))$, and $a_p(g_{\cF,\cF'})=\alpha\cdot(t-t')$, where $t$ and $t'$ are defined by \eqref{eq:xi_1_in_ordinary_basis} for the lines $V^+=V^+(\cF)$ and $V^+=V^+(\cF')$ respectively.
		\end{enumerate}
	\end{prop}
	\begin{proof}
		Note that (i) is a direct consequence of (ii) and Corollary \ref{cor:vanishing_cocycle_at_frob_ell}, so we only prove (ii). 
		
		Assume $\cF\neq \cF'$. By Theorem \ref{thm:uniqueness_residual_slope} (i), we have $V^+(\cF)\neq V^+(\cF')$. Therefore, if $[\xi]=0$, then $\xi_{V^+(\cF)}(I_p)\cdot V \subseteq V^+(\cF)\cap V^+(\cF')=\{0\}$ by definition, so $[\xi_{V^+(\cF)}]$ is unramified. But this contradicts the fact that $\tr(\xi_{V^+(\cF)})=\lambda$. Hence, $[\xi]\neq 0$.
		
		For $\ell\nmid Np$, we have $a_\ell(\cF)=\tr(\rho_{\cF}(\Frob_\ell))$ and similarly for $\cF'$, so the computation of $a_\ell(g_{\cF,\cF'})$ follows from the observation that $\tr\rho_{\cF}\equiv \tr\rho + X \tr(\xi_{V^+(\cF)} \cdot \rho) \mod X^2$ and similarly for $\cF'$. 
		
		The ordinarity condition for $\cF$ implies that $a_p(\cF)\equiv(1+tX)\cdot\alpha \mod X^2$ and similarly for $\cF'$, and the formula $a_p(g_{\cF,\cF'})=\alpha\cdot(t-t')$ easily follows.
		\end{proof}
		\begin{rem}\label{rem:GS_L_invariant}
			With the notations of Proposition \ref{prop:difference_two_Hida_families} (ii), the term $-2t$ appearing in \eqref{tag:slope_2t+z_one_line} can be interpreted as a Greenberg-Stevens $\cL$-invariant 
		\[\cL_\text{GS}(\ad\cF):=-2a_p(f)^{-1}\cdot \frac{\rd}{\rd X}\textbf{a}_p(\cF)_{|X=0}\]
		by analogy with \cite[Equation before Theorem 3]{dasguptapadicartin}.
		\end{rem}
		
	\subsection{Number of Hida families and structure of $\cT$}
	
	Recall from Theorem \ref{thm:etaleness} that the irreducible components of $\Spec\cT$ are in bijection with $\varLambda$-adic Hida eigenfamilies specializing to $f_\alpha$ at $X=0$. We prove in this paragraph the following result.
		
	\begin{thm}\label{thm:number_Hida_families}
		Assume \eqref{tag:discriminant} and \eqref{tag:nonvanishing_resultant_Q_PL} hold.
		\begin{enumerate}
			\item There exist exactly four Hida families passing through $f_\alpha$.
			\item We have $S^{\dagger}_1(N,\ob{\Q}_p)\lsem f_\alpha\rsem =S^{\dagger}_1(N,\ob{\Q}_p)[I_{f_\alpha}^2]$, where $I_{f_\alpha}$ is the ideal of the Hecke algebra acting on $S^{\dagger}_1(N,\ob{\Q}_p)$ generated by $\mathbf{U}_p-\alpha$ and $T_\ell-a_\ell(f)$ for $\ell\nmid Np$.
		\end{enumerate}
	\end{thm}
	
	As $\varLambda$ is a DVR and $\cT$ is torsion-free, $\cT$ is free of rank $n$ over $\varLambda$, where $n$ is the number of Hida families passing through $f_\alpha$. Hence, the perfect duality in \eqref{eq:Hida_duality} shows that $n$ is the $\ob{\Q}_p$-dimension of $S_1^{\dagger}(N,\ob{\Q}_p)\lsem f_\alpha\rsem$. Note that this space contains the linear subspace $S_1(Np,\ob{\Q}_p)\lsem f_\alpha\rsem=f_\alpha\ob{\Q}_p\oplus f(q^p)\ob{\Q}_p$ consisting of classical cusp forms, so $n\geq2$. We will deduce Theorem \ref{thm:number_Hida_families} from the following proposition.
	
	\begin{prop}\label{prop:linear_independence_generalized_eigenforms}
		Assume \eqref{tag:discriminant}. 
		\begin{enumerate}
			\item Let $\cF$ and $\cF'$ be two distinct families passing through $f_\alpha$. If \eqref{tag:non_vanishing_resultant_P_L_P_M} holds, then the modular form $g_{\cF,\cF'}$ defined in \eqref{eq:def_g_F_F'} is not proportional to $f(q^p)$.
			\item Let $\cF$, $\cF'$ and $\cF''$ be three distinct families passing through $f_\alpha$. If \eqref{tag:nonvanishing_resultant_Q_PL} holds, then $g_{\cF,\cF'}$, $g_{\cF',\cF''}$ and $f(q^p)$ are $\ob{\Q}_p$-linearly independent.
			\item Let $\cF$, $\cF'$, $\cF''$ and $\cF'''$ be four distinct families passing through $f_\alpha$. If \eqref{tag:nonvanishing_resultant_Q_PL} holds, then $g_{\cF,\cF'}$, $g_{\cF',\cF''}$ and $g_{\cF'',\cF'''}$ are $\ob{\Q}_p$-linearly independent.
		\end{enumerate}
	\end{prop}
	
	\begin{proof}
		If $g_{\cF,\cF'}$  is a multiple of $f(q^p)$, then $a_\ell(g_{\cF,\cF'})=0$ for all primes $\ell\nmid Np$. Therefore, if \eqref{tag:non_vanishing_resultant_P_L_P_M} holds, then the nonzero cocycle class $[\xi]$ in Proposition \ref{prop:difference_two_Hida_families} (ii) must be zero by Corollary \ref{cor:vanishing_cocycle_at_frob_ell}. This proves (i).
		
		Assume \eqref{tag:nonvanishing_resultant_Q_PL} holds. To prove (ii), we consider a linear combination of the form $\mu\cdot g_{\cF,\cF'}+\tilde{\mu}\cdot g_{\cF',\cF''}+\nu\cdot  f(q^p)=0$ for some $\mu,\tilde{\mu},\nu\in\ob{\Q}_p$ and we prove that $\mu=\tilde{\mu}=0$ (which clearly implies $\nu=0$ as well). Let $[\xi_1]$ be the cocycle class $[\xi_{V^+}]$ defined in Lemma \ref{lem:definition_xi_1} for $V^+=\bV_\cF^+/X\bV_\cF^+$ and similarly define $[\xi_1']$ and $[\xi_1'']$ for $\cF'$ and $\cF''$ respectively. Then, by the same argument as in (i), we have the relation $\mu\cdot ([\xi_1]-[\xi_1'])+\tilde{\mu}\cdot ([\xi_1']-[\xi_1''])=0$ in $\HH^1(\Q,\End(V))$ which, after restriction to $G_{\Q_p}$, becomes $\mu\cdot (\xi_1-\xi_1')+\tilde{\mu}\cdot (\xi_1'-\xi_1'')=0$, since $\HH^1(\Q_p,\End(V))=\mathrm{Z}^1(\Q_p,\End(V))$. Therefore,
		\begin{equation}\label{eq:relation_three_cocycles}
			\mu\cdot \xi_1+\mu'\cdot\xi_1'+\mu''\cdot\xi_1''=0,
		\end{equation}
		on $G_{\Q_p}$, where $\mu'=\tilde{\mu}-\mu$, and $\mu''=-\tilde{\mu}$. Notice that $\mu+\mu'+\mu''=0$. We will prove that $\mu=\mu'=\mu''=0$. 
		
		Fix a basis $(e_1,e_2)$ of $V_{\ob{\Q}}$ adapted to the residual lines of $\cF,\cF',\cF''$, whose existence is ensured by Lemma \ref{lem:existence_basis_adapted_to_ordinary_lines}, and let $s,s',s''$ be the residual slopes with respect to $(e_1,e_2)$ of $\cF,\cF',\cF''$ respectively. We may then consider the elements $x,x',x''$ defined by the relation \eqref{eq:xi_1_in_ordinary_basis} for $\xi_1$, $\xi_1'$ and $\xi_1''$ respectively. One sees by looking at the $\lambda$-coordinates of the $(1,2)$- and $(2,1)$-entries of \eqref{eq:relation_three_cocycles} that $(\mu,\mu',\mu'')^\textbf{t}$ lies in the kernel of 
		\begin{equation}
			\mathcal{D}=\begin{pmatrix}
				x&x'&x'' \\
				s(1-sx)&s'(1-s'x')&s''(1-s''x'')\\
				1&1&1
			\end{pmatrix}.
		\end{equation}
		As we assumed \eqref{tag:nonvanishing_resultant_Q_PL}, $\rP_L(S)$ does not vanish at $s,s',s''$, and $x,x',x''$ can be computed using Proposition \ref{prop:calculation_Up}. An explicit computation then yields
		\[\det(\mathcal{D})=-\frac{1}{4}\cdot (s-s')(s-s'')(s'-s'')\cdot \dfrac{L_2\cdot (L_2^2+4 L_1 L_3)}{\rP_L(s)\cdot \rP_L(s')\cdot\rP_L(s'')}. \]
		The slopes $s,s',s''$ are pairwise distinct by Theorem \ref{thm:uniqueness_residual_slope} (i), $L_2\neq0$ by Lemma \ref{lem:independence_entries_matrices_logs}, and $L_2^2+4 L_1 L_3\neq0$ as well by Lemma \ref{lem:equivalence_resultant_Q_PL}. Therefore, $\det\mathcal{D}\neq0$, and $\mu=\mu'=\mu''$, as wanted.
		
		We now prove (iii) and consider a $\ob{\Q}_p$-linear combination of the form $\mu\cdot g_{\cF,\cF'}+\nu_1\cdot g_{\cF',\cF''}+\nu_2\cdot g_{\cF'',\cF'''}=0$. Letting $\mu'=\nu_1-\mu$, $\mu''=\nu_2-\nu_1$ and $\mu'''=-\nu_2$, and using the same argument (and notations) as in the proof of (ii), we obtain
		\begin{equation}\label{eq:relation_four_cocycles}
			\mu\cdot \xi_1+\mu'\cdot\xi_1'+\mu''\cdot\xi_1''+\mu'''\cdot\xi_1'''=0,
		\end{equation}
		on $G_{\Q_p}$. We fix as in (ii) a basis $(e_1,e_2)$ adapted to the residual lines of the four families. Considering the $a_p$-coefficients in the linear combination, we obtain the additional relation $\mu\cdot t+\mu'\cdot t'+\mu''\cdot t''+\mu'''\cdot t'''=0$ which we derive from Proposition \ref{prop:difference_two_Hida_families} (ii), where $t,t',t'',t'''$ are defined by the relation \eqref{eq:xi_1_in_ordinary_basis} for the ordinary line of $\cF,\cF',\cF'',\cF'''$ respectively. If $s,s',s'',s'''$ are their respective residual slopes, then $(\mu,\mu',\mu'',\mu''')^\textbf{t}$ lies in the kernel of
			\begin{equation}
			\mathcal{E}=\begin{pmatrix}
				t&t'&t''&t''' \\
				x&x'&x''&x''' \\
				s(1-sx)&s'(1-s'x')&s''(1-s''x'')&s'''(1-s'''x''')\\
				1&1&1&1
			\end{pmatrix}.
		\end{equation}
		Using explicit expressions provided by Proposition \ref{prop:calculation_Up}, we see that
		\[ \dfrac{\det(\mathcal{E})}{(s-s')(s-s'')(s-s''')(s'-s'')(s-s''')(s''-s''')}=\dfrac{L_2\cdot (L_2^2+4 L_1 L_3)\cdot(4 L_2 M_{13}-4 L_3 M_{12})}{32\cdot \rP_L(s)\cdot \rP_L(s')\cdot\rP_L(s'')\cdot\rP_L(s''')}.\]
		We only need to justify that the right-most expression in parentheses does not vanish. But this is precisely the coefficient in $S^4$ of the polynomial $\rQ(S)$, which is of degree at most $4$, and in fact of degree exactly $4$ because it has four  distinct roots $s,s',s'',s'''$. This proves that $\mu=\mu'=\mu''=\mu'''=0$ and finishes the proof of the proposition.
	\end{proof}
	
	\begin{proof}[Proof of Theorem \ref{thm:number_Hida_families}]
	Consider (i). We have already seen that the number of Hida families passing through $f_\alpha$ is $n=\dim_{\ob{\Q}_p} S^{\dagger}_1(N)\lsem f_\alpha\rsem \geq2$. By Theorem \ref{thm:uniqueness_residual_slope} (ii), we also have $n\leq 4$.
	
	Let $\cS_0$ be the subspace of $S^{\dagger}_1(N)\lsem f_\alpha\rsem$ consisting of forms with a trivial $a_1$ coefficient. Then $\cS_0$ contains $f(q^p)$ and it is of dimension $n-1$. If $n=2$, then, letting $\cF$ and $\cF'$ be the two distinct Hida families passing through $f_\alpha$, one has $g_{\cF,\cF'}\in\cS_0$, which contradicts Proposition \ref{prop:linear_independence_generalized_eigenforms} (i). Similarly, if $n=3$, then Proposition \ref{prop:linear_independence_generalized_eigenforms} (ii) shows that $\cS_0$ has dimension $\geq 3$, which is impossible. This proves that $n=4$.

	To prove (ii), choose three distinct Hida families $\cF,\cF',\cF''$ and consider $g_{\cF,\cF'}$ and $g_{\cF',\cF''}$ as in (i). Then part (i) and Proposition \ref{prop:linear_independence_generalized_eigenforms} (ii) show that $(f_\alpha,f(q^p),g_{\cF,\cF'},g_{\cF',\cF''})$ is a basis of $S^{\dagger}_1(N)\lsem f_\alpha\rsem $. It remains to check that each element of this basis is killed by $I_{f_\alpha}^2$. It is clear for $f_\alpha$ and also for $f(q^p)$, as $(\mathbf{U}_p-\alpha)f(q^p)=f_\alpha$. For the remaining two forms, \eqref{eq:def_g_F_F'} shows (e.g. for $g_{\cF,\cF'}$) that, for $\ell\nmid Np$,
	\[(\textbf{T}_\ell-a_\ell(f))g_{\cF,\cF'}=a_\ell(g_{\cF,\cF'})\cdot f_\alpha, \quad \mbox{and}\quad (\mathbf{U}_p-\alpha)g_{\cF,\cF'}=a_p(g_{\cF,\cF'})\cdot f_\alpha,\]
	so $g_{\cF,\cF'}$ is indeed killed by $I_{f_\alpha}^2$. This ends the proof of (ii) and of the theorem.
\end{proof}	
	
\begin{proof}[Proof of Theorem \ref{thm:main_result}]
	First note that $\bar\Q_p\lsem X_1,X_2,X_3,X_4 \rsem/(\{ X_iX_j\}_{1 \leq i<j \leq 4})$ is isomorphic to
	\[\cB=\varLambda\times_{\bar\Q_p}\varLambda\times_{\bar\Q_p}\varLambda\times_{\bar\Q_p}\varLambda=\{(c_1(X),c_2(X),c_3(X),c_4(X) \, : \, c_1(0)=c_2(0)=c_3(0)=c_4(0)\} \]
	as $\varLambda$-algebras.
	By Theorem \ref{thm:number_Hida_families} (i), there is an injection $\iota\colon \cT\hookrightarrow \cB$ of rank-$4$ local $\varLambda$-algebras given by $T\mapsto (a_1(T\cF),a_1(T\cF'),a_1(T\cF''),a_1(T\cF'''))$, where $\cF,\cF',\cF'',\cF'''$ are the $4$ Hida families passing through $f_\alpha$. We prove that $\iota$ is an isomorphism. Since $\cT$ and $\cB$ are complete and have the same residue field, it is enough to check that $\iota$ is unramified, \textit{i.e.}, $\iota(\Gm_\cT)=\Gm_\cB$. This will follow from Nakayama's lemma, once we prove that the map $\Gm_\cT/(\Gm_\cT^2+\Gm_\varLambda\cdot \cT)\to \Gm_\cB/(\Gm_\cB^2+\Gm_\varLambda\cdot \cB)$ at the level of cotangent spaces is surjective.
	
	The $\ob{\Q}_p$-vector space $\Gm_\cB/(\Gm_\cB^2+\Gm_\varLambda\cdot \cB)$ is of dimension $3$ with $\ob{\Q}_p$-basis given by $b_1=(0,X,0,0)$, $b_2=(0,0,X,0)$ and $b_3=(0,0,0,X)$. Letting $\sum_{n\geq 1}a_nq^n$ be the $q$-expansion of $f_\alpha$, $\Gm_\cT$ is generated by $T_n-a_n$ for $n\geq 1$ and we have from the definition of $g_{\cF',\cF}$ given in \eqref{eq:def_g_F_F'} that
	\begin{align*}
		\iota(T_n-a_n)&\equiv (a_n(\cF)-a_n)\cdot 1_\cB+a_n(g_{\cF',\cF})\cdot b_1+a_n(g_{\cF'',\cF})\cdot b_2+a_n(g_{\cF''',\cF})\cdot b_3 \mod \Gm_\cB^2 \\
		&\equiv a_n(g_{\cF',\cF})\cdot b_1+a_n(g_{\cF'',\cF})\cdot b_2+a_n(g_{\cF''',\cF})\cdot b_3 \mod (\Gm_\cB^2+\Gm_{\varLambda}\cdot \cB). 
	\end{align*}
	But Proposition \ref{prop:linear_independence_generalized_eigenforms} (iii) shows that the $q$-expansions of $g_{\cF',\cF},g_{\cF'',\cF},g_{\cF''',\cF}$ are $\ob{\Q}_p$-linearly independent, so $\iota(\Gm_\cT)/(\Gm_\cB^2+\Gm_\varLambda\cdot \cB)=\Gm_\cB/(\Gm_\cB^2+\Gm_\varLambda\cdot \cB)$. This finishes the proof of the theorem.
\end{proof}

	The following result was conjectured in \cite[Conj. 4.1]{DLR4}.
	\begin{thm}\label{thm:side_result_DLR}
		Assume \eqref{tag:discriminant} and \eqref{tag:nonvanishing_resultant_Q_PL} hold. The (two-dimensional) subspace $S_1^\dagger(N,\ob{\Q}_p)\lsem f_\alpha \rsem_0$ of $S_1^\dagger(N,\ob{\Q}_p)\lsem f_\alpha \rsem$ consisting of overconvergent modular forms with trivial first and $p$-th Fourier coefficients is naturally isomorphic to $\HH^1(\Q,\ad^0(\rho))$.
	\end{thm}
\begin{proof}
	Given any $f^\flat\in S_1^\dagger(N,\ob{\Q}_p)\lsem f_\alpha \rsem_0$, the modular form $f_\alpha +X\cdot f^\flat$ with coefficients in $\ob{\Q}_p[X]/(X^2)$ is an eigenform for all Hecke operators, and its associated  $\ob{\Q}_p[X]/(X^2)$-valued Galois representation takes the form $(\Id+X\cdot \xi_{f^\flat})\cdot \rho$ some cocycle class $[\xi_{f^\flat}]\in\HH^1(\Q,W)$. We then have $a_\ell(f^\flat)=\tr(\xi_{f^\flat}(\Frob_\ell)\cdot\rho(\Frob_\ell))$ for $\ell\nmid Np$.
	
	We claim that the assignment $f^\flat \mapsto [\xi_{f^\flat}]$ yields an isomorphism between $ S_1^\dagger(N,\ob{\Q}_p)\lsem f_\alpha \rsem_0$ and $\HH^1(\Q,W)$. By comparing dimensions, it is enough to show that $[\xi_{f^\flat}]=0$ implies $f^\flat=0$. Assume $[\xi_{f^\flat}]=0$ and choose three Hida families $\cF,\cF',\cF''$ passing through $f_\alpha$. As $a_1(f^\flat)=0$, the proof of Theorem \ref{thm:number_Hida_families} (ii) shows that $f^\flat$ can be written as $\mu\cdot g_{\cF,\cF'}+\mu'\cdot g_{\cF',\cF''}+\nu\cdot f(q^p)$. Arguing as in the proof of Proposition \ref{prop:linear_independence_generalized_eigenforms} (ii), we have $[\xi_{f^\flat}] = \mu\cdot [\xi_{g_{\cF,\cF'}}]+\mu'\cdot [\xi_{g_{\cF',\cF''}}]$, and $[\xi_{f^\flat}]=0$ implies $\mu=\mu'=0$. Finally, $a_p(f(q^p))=1$ shows that $\nu=0$ as well.
\end{proof}

\section{Hecke structure  of the $p$-ordinary \'etale cohomology of the tower of modular curves}\label{sec:Ohta}

We assume in this section that $p$ is odd and we denote by $\boldsymbol{\Lambda}_{\Z_p}= \Z_p \lsem X \rsem $ the Iwasawa algebra with $\Z_p$-coefficients. The adic space $\mathrm{Spa}(\boldsymbol{\Lambda}_{\Z_p},\boldsymbol{\Lambda}_{\Z_p})$ is a formal model of any connected component of $\cW$.

Hida constructed a big $\boldsymbol{\Lambda}_{\Z_p}$--adic Galois representation  by using the \'etale cohomology of the tower of modular curves
${X_1(Np^r)}$ of level $\Gamma_1(Np^r)$ over $\Q$ and the Eichler--Shimura  relations. More precisely, letting $J_1(Np^r)$ be the Jacobian variety over $\Q$ of $X_1(Np^r)$, Hida considers the projective limit \[\rH^1_{\mathrm{et}}(X_1(Np^{\infty})_{\bar{\Q}},\Z_p):= \underset{r}{\varprojlim} \rH^1_{\mathrm{et}}(X_1(Np^r)_{\bar{\Q}},\Z_p)=\underset{r}{\varprojlim}\text{ } T_p(J_1(Np^r))^{\vee},\]
taken with respect to the trace mappings. It is naturally a module for the ''big'' $p$-adic Hecke algebra (with the dual action), which is itself an algebra over the completed group ring $\boldsymbol{\Lambda}_{\Z_p}$ 
via the diamond operators. 

Using the $p$-ordinary projector $\mathbf{e}^{\mathrm{ord}}$ attached to the dual $\mathbf{U}_p$--operator, we obtain a module \[ \rH^1_{\mathrm{et}}(X_\infty,\Z_p)^{\mathrm{ord} } :=\mathbf{e}^{\mathrm{ord}}\cdot \rH^1_{\mathrm{et}}(X_1(Np^{\infty})_{\bar{\Q}},\Z_p)\]
on which the full Hida--Hecke algebra $\mathfrak{h}_{\Z_p}$ acts faithfully. Hida proves in \cite[Thm.3.1]{hida1986galois} using
the comparison isomorphism between \'etale and Betti cohomology and explicit calculations in
cohomology of arithmetic groups that $ \rH^1_{\mathrm{et}}(X_\infty,\Z_p)^{\mathrm{ord} } $ is a finite free $\boldsymbol{\Lambda}_{\Z_p}$--module,  and that the resulting big $p$-adic Galois
representation
\begin{equation}\label{bigrepresentation}
\rho_{\mathfrak{h}_{\Z_p}}:G_{\Q}
\to \mathrm{Aut}_{\boldsymbol{\Lambda}_{\Z_p}}( \rH^1_{\mathrm{et}}(X_\infty,\Z_p)^{\mathrm{ord} } )\end{equation} 
is of generic rank two over the field of fractions $\mathrm{Fr}(\mathfrak{h}_{\Z_p})$ (after taking the nilreduction) and sends $\Frob_\ell$ on $\mathbf{T}_\ell$ when $\ell \nmid Np$ (as a representation generically of rank $2$).

\medskip

Let $\cO$ be the ring of integers of a $p$-adic field containing the image of the nebentype character $\psi_f$ and $\Q_p(\mu_{p^{\infty}})$ and put  $\boldsymbol{\Lambda}= \boldsymbol{\Lambda}_{\Z_p} \hat{\otimes}_{\Z_p} \cO=\cO\lsem X\rsem$, $\mathfrak{h}=\mathfrak{h}_{\Z_p}\hat{\otimes}_{\Z_p}\cO$. Cais \cite{Cais18T}, Mazur--Wiles \cite{mazurwiles}, Ohta \cite{Ohta95} and Tilouine \cite{Tilouine87} defined a ``good'' abelian variety quotient of $J_1(Np^r)_{\Q}$ with good reduction at $p$ over $\Z_p[\zeta_{p^r}]$. Moreover, by analyzing the connected part and the \'etale quotient of its corresponding ordinary $p$-divisible group, they obtained (after letting $r \to \infty$) an exact sequence of $\gh[G_{\Q_p}]$-modules
\begin{equation}\label{nearbydual}
	0 \to \mathcal{L}^+      \to \rH^1_{\mathrm{et}}(X_\infty,\Z_p)^{\mathrm{ord} } \hat{\otimes}_{\Z_p}\cO \to \mathcal{L}^-  \to 0, \end{equation} 
with the following properties:
\begin{enumerate}
\item $\mathcal{L}^+$  is exactly the subspace of $I_p$-invariants, and Frobenius acts on it by $\mathbf{U}_p$ (this goes through a deep study of the correspondence $\mathbf{U}_p$  on the ordinary part of the special fiber of the Jacobian), 
\item  $\mathcal{L}^+ \simeq \gh$, and
\item $ \mathcal{L}^-  \simeq \mathbf{S}(N,\boldsymbol{\Lambda})^{\ord} (\epsilon_p \langle \cdot \rangle )$, where   $\mathbf{S}(N,\boldsymbol{\Lambda})^{\ord}$ is the finite free $\boldsymbol{\Lambda}$--module of $p$-ordinary Hida families of tame level $N$ and $\langle \cdot \rangle $ is the character of $G_{\Q_p}$ induced by the diamond operators via class field theory.

\end{enumerate}

\subsection{Ohta's $\Lambda$-adic Eichler Shimura and non-freeness of the \'etale cohomology}

Let $\mathfrak{p}_f$ be the height one prime  ideal of $\mathfrak{h}$ attached to $f_\alpha$ and $\mathfrak{p}_1$ be the height one prime ideal of ${\bf \Lambda}$ corresponding to the weight of $f_\alpha$ under the finite flat morphism $\kappa: \Spec \mathfrak{h} \to \Spec {\bf \Lambda}$. In Iwasawa theory, it is crucial to know whether the localization $\rH^1_{\mathrm{et}}(X_\infty,\Z_p)^{\mathrm{ord}}_{\mathfrak{p}_f}$ of $\rH^1_{\mathrm{et}}(X_\infty,\Z_p)^{\mathrm{ord}}$ is free of rank two over $\mathfrak{h}_{\mathfrak{p}_f}$.  This question is the characteristic zero variant of the multiplicity one question in \cite{BLR91}. 
If the answer is negative, then multiplicity one automatically fails in \cite{BLR91}, that is, if $\mathfrak{m}$  the maximal ideal of $\mathfrak{h}$ corresponding to the reduction of $f_\alpha$ modulo $p$ is not Eisenstein, then 
\[ J_1(Np)(\bar\Q)[\mathfrak{m}] \simeq \ob{\rho}^r. \]
with $r \geq 2$, where $\ob{\rho}$ is the irreducible modular residual representation of $f_\alpha$ mod $p$ (see the proof of \cite[Prop. 3.1.1]{emerton2006variation} for this implication). The failure of multiplicity one when $\ob{\rho}(\Frob_p)$ is scalar was already observed by Wiese in \cite{wiese2007multiplicities}.

Recall that $\cT$ is the completion of the strict  henselianization of $\mathfrak{h}$ at $\mathfrak{p}_f$. Hence, after strict henselianization and completion with respect to the maximal ideal, the module $\textbf{S}(N,\varLambda)_{f_\alpha}^{\ord}=\widehat{\mathbf{S}(N,\boldsymbol{\Lambda})}^{\ord}_{\mathfrak{p}_f} \simeq \Hom_{\varLambda}(\cT, \varLambda)$ is free of rank one over $\cT$ if and only if $\cT$ is Gorenstein.  If $f_\alpha$ is $p$-regular, then $\rH^1_{\mathrm{et}}(X_\infty,\Z_p)^{\mathrm{ord}}_{\mathfrak{p}_f}$ is $\mathfrak{h}_{\mathfrak{p}_f}$-free of rank $2$ since $\cT$ is regular by \cite{D-B} (hence Gorenstein).  The following theorem treats the case where $f_\alpha$ is $p$-irregular.

\begin{thm}\label{thm:freeness}
	Assume that $f_\alpha$ is irregular at $p$ and assume further that the hypotheses  \eqref{tag:discriminant} and \eqref{tag:nonvanishing_resultant_Q_PL} of \S\ref{sec:regulators} hold. Let $\rH^1_{\mathrm{et}}(X_\infty,\Z_p)^{\mathrm{ord},\pm}_{\mathfrak{p}_f}$ be the subspace of $\rH^1_{\mathrm{et}}(X_\infty,\Z_p)^{\mathrm{ord}}_{\mathfrak{p}_f}$ where the complex conjugation acts by $\pm1$. Then $\rH^1_{\mathrm{et}}(X_\infty,\Z_p)^{\mathrm{ord},\pm}_{\mathfrak{p}_f}$ is generically of rank $1$ over $\gh_{\gp_f}$, but its minimal number of generators is $2$. In particular, neither $\rH^1_{\mathrm{et}}(X_\infty,\Z_p)^{\mathrm{ord},\pm}_{\mathfrak{p}_f}$ nor $\rH^1_{\mathrm{et}}(X_\infty,\Z_p)^{\mathrm{ord}}_{\mathfrak{p}_f}$ is $\mathfrak{h}_{\mathfrak{p}_f}$-free.
\end{thm}
\begin{proof} 
	The fact that $\rH^1_{\mathrm{et}}(X_\infty,\Z_p)^{\mathrm{ord},\pm}_{\mathfrak{p}_f}$ is generically of rank $1$ follows from the oddness of $\rho_{\gh_{\Z_p}}$. Let $\cH^\pm$ (resp. $\cH$) be the strict henselianization and completion of $\rH^1_{\mathrm{et}}(X_\infty,\Z_p)^{\mathrm{ord},\pm}_{\mathfrak{p}_f}$ and  $\rH^1_{\mathrm{et}}(X_\infty,\Z_p)^{\mathrm{ord}}_{\mathfrak{p}_f}$ respectively, and consider the exact sequence  of $G_{\Q_p}$-stable finitely generated $\cT-$modules
	\begin{equation}\label{eq:exact_sequence_Ohta_completed}
		0 \to \cT \to \cH \to \Hom_{\varLambda}(\cT,\varLambda) \to 0
	\end{equation}
	obtained strict henselianizing and completing \eqref{nearbydual}. It follows from the description of $\cT$ that $\dim_{\ob{\Q}_p} \Hom(\cT,\varLambda) \otimes_{\cT} \cT/\Gm_\cT=3$, \textit{i.e.}, the minimal number of generators of $\Hom(\cT,\varLambda)$ over $\cT$ is $3$. Therefore, $\dim_{\ob{\Q}_p}\cH/\Gm_\cT\cH$ equals $3$ or $4$. 
	
	We now consider the $G_{\Q}$-action inherited by $\cH$ from $\rho_{\gh_{\Z_p}}$. Since any element $g\in G_{\Q}$ acting on $\cH/\Gm_\cT\cH$ is killed by the characteristic polynomial of $\rho(g)$, it follows from the main theorem of \cite{BLR91} that $\cH/\Gm_\cT\cH\simeq \rho^r$ for some integer $r\geq1$. Then $r$ must be $2$, and $\dim_{\ob{\Q}_p}\cH/\Gm_\cT\cH=4$. Moreover, we obtain by taking the $\pm$-eigenspaces of the complex conjugation that $\dim_{\ob{\Q}_p}\cH^\pm/\Gm_\cT\cH^\pm=2$. All the statements about $\rH^1_{\mathrm{et}}(X_\infty,\Z_p)^{\mathrm{ord},\pm}_{\mathfrak{p}_f}$ and  $\rH^1_{\mathrm{et}}(X_\infty,\Z_p)^{\mathrm{ord}}_{\mathfrak{p}_f}$ then follow from these facts, \end{proof}

\begin{cor}\label{non-splittingOhta}\

	\begin{enumerate}
		\item The fiber at $\gp_f \in \Spec \mathfrak{h}$ of \eqref{nearbydual}  remains exact.
		\item The localization at $\gp_f$ of the exact sequence \eqref{nearbydual} does not split as a sequence of $\gh_{\gp_f}$-modules.

	\end{enumerate}
\end{cor}

\begin{proof}
	With the notations of the proof of Theorem \ref{thm:freeness}, \eqref{eq:exact_sequence_Ohta_completed} remains exact after tensoring with $\cT/\Gm_\cT$ over $\cT$, since $\dim_{\ob{\Q}_p}\cH/\Gm_\cT\cH=4$ and $\dim_{\ob{\Q}_p} \Hom(\cT,\varLambda) \otimes_{\cT} \cT/\Gm_\cT=3$. One easily deduces (i) from this.
	
	Similarly, (ii) will follow from the fact that \eqref{eq:exact_sequence_Ohta_completed} does not split as $\cT$-modules. Assume that $\cT$ is a direct summand of $\cH$ and let $\cT^0=\cT/X\cT$. Then $\cT^0$ is a direct summand of $\cH/X\cH=\cH^+/X\cH^+\oplus\cH^-/X\cH^-$. Since $\cT^0$ is noetherian and artinian, $\cT^0$ (viewed as an indecomposable $\cT^0$-module) must be direct summand of $\cH^+/X\cH^+$ or $\cH^-/X\cH^-$ by the Krull-Schmidt theorem. As $\dim_{\ob{\Q}_p}\cT^0=\dim_{\ob{\Q}_p}\cH^\pm/X\cH^\pm=4$, either $\cH^+/X\cH^+$ or $\cH^-/X\cH^-$ must be $\cT^0$-free, hence contradicting Theorem \ref{thm:freeness}. This proves that \eqref{eq:exact_sequence_Ohta_completed} does not split as $\cT$-modules. 
\end{proof}

Let $\mathcal{F}$ be a Hida eigenfamily specializing to $f_\alpha$ and $\pi_{\mathcal{F}}: \cT \twoheadrightarrow \mathbf{I}_{\mathcal{F}}$ be the projection corresponding to the Hecke eigensystem of $\mathcal{F}$.  In fact, we know that $\mathbf{I}_{\mathcal{F}} \simeq \varLambda$ by Theorem \ref{thm:etaleness}.  We thus can specialize \eqref{eq:exact_sequence_Ohta_completed} along $\pi_{\mathcal{F}}$ (i.e. taking the base change) to obtain  \begin{equation}\label{eq:exact_sequence_Ohta_completedFsanstor}
		\mathbf{I}_{\mathcal{F}} \to \cH \otimes_{\cT} \mathbf{I}_{\mathcal{F}} \to \Hom_{\varLambda}(\cT,\varLambda) \otimes_{\cT} \mathbf{I}_{\mathcal{F}} \to 0.
	\end{equation}
	Since $\mathbf{I}_{\mathcal{F}}$ is a discrete valuation ring and $\Hom_{\varLambda}(\cT,\varLambda)$ is of generic rank $1$ over $\cT$, we can mod by the torsion part in the middle term and right term of \eqref{eq:exact_sequence_Ohta_completedFsanstor} to obtain an sequence of $ \mathbf{I}_{\mathcal{F}}[G_{\Q_p}]$-modules 
	
	\begin{equation}\label{eq:exact_sequence_Ohta_completedF}
		\mathbf{I}_{\mathcal{F}} \to  \mathbf{I}_{\mathcal{F}}^2 \to \mathbf{I}_{\mathcal{F}} \to 0
	\end{equation}
	which is exact on the right, where
	\begin{enumerate}
	\item the action of $G_{\Q_p}$  on the submodule is unramified and the Frobenius acts by $\mathbf{a}_p(\mathcal{F})$, and 
	\item  the action on the subquotient is given by the dual action.
	\end{enumerate}
	
	By construction and since $\rho$ is absolutely irreducible, the action on the middle term is given by the $\mathbf{I}_{\mathcal{F}}$-adic Galois representation $\rho_{\mathcal{F}}$ attached to $\mathcal{F}$ by \cite{nyssen1996pseudo,rouquier1996caracterisation}.

\begin{cor} The first map in \eqref{eq:exact_sequence_Ohta_completedF} is injective and its image is saturated. In particular, \eqref{eq:exact_sequence_Ohta_completedF} is exact and splits. 
\end{cor}
\begin{proof}  Since $\mathbf{I}_{\mathcal{F}} $ is a discrete valuation ring, the first map in \eqref{eq:exact_sequence_Ohta_completedFsanstor} is either zero or lands in the free part. But we know that is non-zero generically, so it suffices to show that its image is saturated. The reduction of \eqref{eq:exact_sequence_Ohta_completedFsanstor} modulo the maximal ideal of $\mathbf{I}_{\mathcal{F}} $ is exactly the fiber at $\gp_f \in \Spec \mathfrak{h}$ of \eqref{nearbydual}  which remains exact by Cor.\ref{non-splittingOhta}. Hence the image of the first map in \eqref{eq:exact_sequence_Ohta_completedF} is saturated.
\end{proof}

\subsection{Triangulation of Kedlaya--Pottharst--Xiao along the desingularization of $\cC$} 

For an affinoid algebra $A$ over $\overline{\Q}_p$, we let $\mathcal{R}_A$  be the direct limit on $r \mapsto 1$ of the rings of rigid analytic functions on the annuli $r < \mid z \mid_p   < 1$ over $A$ with $r \in p^{\Q}$ ({\it i.e.} $\mathcal{R}_A$ is the Robba ring over $A$). This ring is complete for the direct limit topology in the category of locally convex topological $\Q_p$-vector spaces induced by the LF topology. The $A$--algebra $\mathcal{R}_A$  is endowed with continuous actions of a Frobenius $\Phi$ and $\Gamma:=\Z_p^\times $:

\begin{enumerate}
	
	\item $\Phi \left( \sum_n a_n z^n \right) = \sum_n a_n \left( (1+z)^p-1 \right)^n$.
	
	\item  $\gamma \left( \sum_n a_n z^n \right) = \sum_n a_n \left( (1+z)^\gamma -1 \right)^n$ for all $\gamma \in \Gamma$.

\end{enumerate}

Note that the actions of $\Phi$ and $\Gamma$ commute.

\medskip A $(\Phi,\Gamma)$--module over $\mathcal{R}_A$ is a finite projective $\mathcal{R}_A$--module $\cM$ equipped with
\begin{enumerate}
	\item a commuting semi--linear actions of $\Phi_{\cM}$ and $\Gamma$ such that the action of $\Gamma$ is continuous for the LF topology, and 
	\item the linearization $\cM \otimes_{\mathcal{R}_A,\Phi} \mathcal{R}_A \to \cM $ of $\Phi_{\cM}$ is an isomorphism.
	
\end{enumerate}

\medskip

Let $\mathbf{\Phi\Gamma}_A$ be the category of $(\Phi,\Gamma)$--modules over $\mathcal{R}_A$ and $\mathbf{Rep}_A(G_{\Q_p} )$ be  the category of continuous representations of $G_{\Q_p}$ on finite projective $A$--modules.   By a result of Kedlaya--Liu \cite{KL15}, there exists a full embedding \[\mathscr{D}_{\mathrm{rig}}(.): \mathbf{Rep}_A(G_{\Q_p} ) \mapsto \mathbf{\Phi\Gamma}_A. \]

Let $\mathcal{U}=\mathrm{Sp}(A) \subset \cC$ be an open connected admissible affinoid enough small containing $f_\alpha$ so that we have a $p$-adic representation 
\[ \rho_{\mathcal{U}} : G_{\Q} \to \GL_2(\cO(\mathcal{U})) \] 
sending $\tr(\rho_{\mathcal{U}}(\Frob_\ell))$ on $\mathbf{T}_\ell \in \cO(\mathcal{U})$ for $\ell \nmid Np$. The existence of $\mathcal{U}$ follows from the fact that $\rho$ is irreducible and $G_{\Q}$ is compact.

So we can attach to $\rho_{\mathcal{U}} $ a $(\Phi,\Gamma)$--module over $\mathcal{R}_A$ that we denote by $\cM_A$ and it is free of rank $2$ over $\mathcal{R}_A$.

We let now $\widetilde{\mathcal{U}}=\mathrm{Sp}(\widetilde{A})$ be the normalization of $ \mathcal{U}$ and it is not connected any more when $\mathcal{U}$ is chosen small enough. In fact, $\widetilde{\mathcal{U}}$ has four irreducible connected components and all of them are smooth. 

We let $\cM_{\widetilde{A}}$ be the $(\Phi,\Gamma)$--module over $\mathcal{R}_{\widetilde{A}}$  defined by the pullback of $\cM_A$. Kedlaya--Pottharst--Xiao \cite[\S.6.4]{KPX12} constructed a triangulation, i.e., an exact sequence of $(\Phi,\Gamma)$--modules
 \begin{equation}\label{triangulation} 0 \to \cM^{+}_{\widetilde{A}} \to \cM_{\widetilde{A}} \overset{\lambda}{\rightarrow} \cM^{-}_{\widetilde{A}} , \end{equation}  such that
 \begin{enumerate}
 \item  both $\cM^{\pm}_{\widetilde{A}}$  are projective $\mathcal{R}_{\widetilde{A}}$--modules of rank $1$.
 
 \item the cokernel of $\lambda$  is supported on finite set of $\mathrm{Sp}(\widetilde{A})$.
 
 \end{enumerate}  
We know from the classification of rank 1 $(\Phi,\Gamma)$--modules that $\cM^{\pm}_{\widetilde{A}} \simeq  \mathcal{R}_{\widetilde{A}}(\delta_{\pm}) \otimes_{\mathrm{Sp}(\widetilde{A})}\mathcal{L}_{\pm}$ with
\begin{enumerate}
\item $\mathcal{L}_{\pm}$ is a line bundle on $\mathrm{Sp}(\widetilde{A})$.

\item $\delta_{+}: \Q_p^\times \to A^\times$ is the character sending $p$ on $\mathbf{U}_p$ and trivial on $\Z^\times_p$.
\item  $\delta_{-}: \Q_p^\times \to A^\times$ is the character sending $p$ on $\mathbf{U}^{-1}_p$ and on $\Z^\times_p$ it is the inverse of the universal weight nebentypus character.
\end{enumerate}
Kedlaya--Pottharst--Xiao showed in \cite[\S.6.4]{KPX12}  that \eqref{triangulation} interpolates the triangulation of the $(\Phi,\Gamma)$--modules attached to classical eigenforms of $\widetilde{\mathcal{U}}$ away from critical points, where the fiber of $\cM^{+}_{\widetilde{A}}$ at any critical point is not saturated anymore but has an index $t^{k-1}$ on its saturation and $k$ is the weight of the critical point. 

Let $\widetilde{\cC}$ be the normalization of $\cC$. Since  $\cM^{\pm}_{\widetilde{A}}$  (resp. $\cM_{A}$) can be glued along a covering of $\widetilde{\cC}$ (resp. $\cC$) to obtain a $(\Phi,\Gamma)$--module $\cM^{\pm}_{\cO_{\widetilde{\cC}}}$ (resp. $\cM_{\cO_{\cC}}$) with coefficients in $\cO_{\widetilde{\cC}}$ (resp. $\cO_{\cC}$), we can ask whether the triangulation over the entire desingularization descends to a triangulation on $\cC$.

\begin{cor}  Assume that $f_\alpha$ is irregular at $p$ and assume further that the hypotheses  \eqref{tag:discriminant} and \eqref{tag:nonvanishing_resultant_Q_PL} of \S\ref{sec:regulators} hold.
The triangulation $0 \to \cM^{+}_{\cO_{\widetilde{\cC}}} \rightarrow \cM_{\cO_{\widetilde{\cC}}} \to \cM^{-}_{\cO_{\widetilde{\cC}}}$ does not descend to a triangulation of $\cM_{\cO_{\cC}}$ on any open admissible neighborhood of  $ f_\alpha$ in $\cC$.
\end{cor}
\begin{proof} Assume that the triangulation descends to $0 \to \cM^{+}_{\cO_{\cC}} \rightarrow \cM_{\cO_{\cC}} \to \cM^{-}_{\cO_{\cC}}$. We let now $\cM^{\pm}_{\mathrm{x}}$ (resp. $\cM_{\mathrm{x}}$) be the fiber of $\cM^{\pm}_{\cO_{\cC}}$ (resp. $\cM_{\cO_{\cC}}$) at $\mathrm{x}$ (the point corresponding to $f_\alpha$). In fact, $\cM^{\pm}_{\mathrm{x}} = \mathcal{R}_{\Q_p}(\delta_{\pm,x})$ and $ \cM_{\mathrm{x}}= \mathcal{R}_{\Q_p}(\delta_{+,x})^2$, where $\delta_{\pm,x}$ is the specialization of $\delta_{\pm}$ at $\mathrm{x}$ but here \[ \delta_{+,x}=\delta_{-,x}\] by the $p$ irregularity hypothesis on $f_\alpha$. This comes together with a complex
\begin{equation}\label{filtrationphigamma}
\mathcal{R}_{\Q_p} \to \mathcal{R}_{\Q_p}^2 \to \mathcal{R}_{\Q_p}.
\end{equation}

 given by the twist of the fiber of  \eqref{triangulation} at $\mathrm{x}$   with $\delta^{-1}_{+,x}$.  Since 
 any $(\Phi,\Gamma)$--submodules of $\mathcal{R}_{\Q_p}$ is of the form $t^n \mathcal{R}_{\Q_p}$ for $n \in \N_{\geq 1}$ and $\mathcal{R}_{\Q_p}$ is an adequate Bezout domain, we can use the arguments of \cite[E.g.6.3.14]{KPX12} to show  \eqref{filtrationphigamma} is a short exact sequence.
 
 However, this leads to a contradiction because we have by Theorem \ref{thm:main_result}  four components of $\cC$  crossing at  $\mathrm{x}$,  and their residual triangulations (ordinary lines) are pairwise distinct points in $\mathbf{P}^1(\bar\Q_p)$ by Theorem \ref{thm:uniqueness_residual_slope} (their slopes in an adapted basis of $\rho$ are the simple roots of the polynomial  \eqref{eq:def_polynomial_computing_residual_slopes}).
\end{proof}
\bibliography{bib}
\bibliographystyle{alpha}
\end{document}